\documentclass[11pt]{elsarticle}
\usepackage[margin=.7in]{geometry}
\usepackage{graphicx} 
\usepackage{amsfonts}
\usepackage{amsmath}
\usepackage{mathtools}
\usepackage{amssymb}
\usepackage{amsthm}
\usepackage{mathrsfs}
\usepackage{dsfont}
 \newcommand{\R}{\mathbb{R}}
 \newcommand{\N}{\mathbb{N}}
 \newcommand{\E}{\mathbb{E}}
 \newcommand{\D}{\mathbb{D}}
 \renewcommand{\P}{\mathbb{P}}
 
\usepackage{autobreak}
\usepackage{hyperref}
\usepackage[noabbrev,capitalize]{cleveref}
\crefname{equation}{equation}{equations}
\numberwithin{equation}{section}

\newtheorem{theorem}{Theorem}[section]
\newtheorem{lemma}[theorem]{Lemma}

\newtheorem{definition}[theorem]{Definition}
\newtheorem{proposition}[theorem]{Proposition}
\newtheorem{corollary}[theorem]{Corollary}
\newcommand{\dd}{\mathrm{d}}
\begin{document}
\begin{frontmatter}

\title{Stochastic Tamed Navier--Stokes Equations with Wiener and Jump Noise on $\mathbb R^3$.
\\I.
Maximal Local $L^p$-Well-Posedness}

\author{Bikram Podder} \ead{bikrampoddar2@gmail.com} \author{Surendra Kumar} \ead{surendraiitr8@gmail.com}

\affiliation{organization={Department of Mathematics}, addressline={University of Delhi}, city={ Delhi}, postcode={110007}, country={India}}
\begin{abstract}
We establish maximal local $L^p$-well-posedness, for $p>3$, for the stochastic tamed Navier--Stokes equations on $\mathbb R^3$ driven simultaneously by multiplicative cylindrical Wiener noise and a compensated Poisson random measure.
For divergence-free initial data $ u_0\in L^p\bigl(\Omega,\mathcal F_0; L^p(\mathbb R^3;\mathbb R^3)\bigr), $ we prove existence and pathwise uniqueness of a local strong solution with $L^p$-valued c\`adl\`ag trajectories and the local $L^p$-energy regularity.
For solutions driven by the same noises, we establish localized Lipschitz dependence on the initial datum in the path supremum and space--time norms.
The discontinuous forcing makes the whole-space Gaussian construction non formal as stopping levels may be overshot by jumps, while convergence of the compensated-Poisson term requires simultaneous control of its quadratic and $p$th integrability modes.
We resolve these difficulties by a jump-compatible localization based on strict pre-exit bounds and predictable left limits.
We further establish bounded-time restart and stochastic pasting on the prescribed stochastic basis.
The resulting family of attainable lifetimes is upward directed and yields a unique maximal local strong solution, which inherits the localized dependence estimate.

Continuation criteria, blow-up alternatives, and global well-posedness under stronger finite-energy hypotheses are treated in the companion Part~II.
\end{abstract}


\end{frontmatter}
\tableofcontents
\section{Introduction}
\label{sec:introduction}

We consider the stochastic tamed Navier--Stokes equations (STNSE) on $\R^3$ driven simultaneously by a cylindrical Wiener process and a compensated Poisson random measure:
\begin{equation}
\label{eq:1_Main}
\begin{cases}
\dd u = \Big[ \nu\Delta u - \mathcal P((u\cdot\nabla)u) - \mathcal P\bigl(g_N(|u|^2)u\bigr) \Big]\dd t + \mathcal P\sigma(u)\,\dd W_t + \int_Z \mathcal PG(u(t-),z)\, \widetilde N(\dd t,\dd z), \\
\nabla\cdot u=0, \qquad (t,x)\in(0,\infty)\times\R^3, \\
u(0)=u_0 .
\end{cases}
\end{equation}
Here $\nu>0$ is the viscosity, $\mathcal P$ denotes the Helmholtz--Leray projection onto divergence-free vector fields, and $g_N$ is the taming function introduced in the STNSE framework of R\"ockner and Zhang \cite{MR2520127,MR2852224}.
The process $W$ is a cylindrical Wiener process on a separable Hilbert space, while $\widetilde N$ is a compensated Poisson random measure on a $\sigma$-finite measure space $(Z,\mathcal Z,\mu)$.
The multiplicative Wiener coefficient $\sigma$ and the jump coefficient  $G$ takes values in $\gamma \bigl(\mathcal U ; L^p(\R^3;\R^3)\bigr)$ and $L^p(\R^3;\R^3)$ respectively. 
The appearance of the predictable state $u(t-)$ in the jump coefficient is essential throughout.

The aim of the present paper is the maximal local well-posedness theory for \eqref{eq:1_Main} with initial data $ u_0 \in L^p\bigl( \Omega, \mathcal F_0;L^p(\R^3;\R^3) \bigr), \ p>3, \ \nabla\cdot u_0=0 \quad \P\text{-a.s.} $ The problem lies at the intersection of two features which require different analytic mechanisms.
On the one hand, the low spatial regularity of the initial datum places the equation naturally in a whole-space $L^p$ framework rather than in the classical $H^1$-based variational setting.
On the other hand, the compensated Poisson forcing produces c\`adl\`ag trajectories, so localization, stopping, restart, and concatenation must retain the distinction between the pre-jump state and the value after a jump.

This paper is the first part of a two-part study.
Its purpose is to construct the canonical maximal local solution on which the continuation and global analysis can subsequently be based.
The present work is restricted to existence, pathwise uniqueness, localized continuous dependence on the initial datum, and the construction of a unique maximal local strong solution on $[0,\tau_{\max})$.
In the second part, this maximal solution serves as the starting point for an intrinsic continuation theory and, under additional regularity assumptions on the initial data and the noise coefficients, for the analysis of non-explosion and global strong solvability.
No continuation criterion, blow-up alternative, or global existence statement is asserted in the present paper.

\subsection{Related literature}

The well-posedness theory for both the STNSE and the stochastic Navier--Stokes equations (SNSE) has been developed in a variety of domains and for different classes of initial data.
R\"ockner and Xicheng Zhang \cite{Rockner2009} studied the deterministic  three-dimensional  tamed NSE equation on $\R^3$, establishing existence, uniqueness and regularity for smooth data.
They also used the tamed approximation to construct suitable weak solutions of the classical Navier--Stokes equation.
In \cite{MR2520127}, R\"ockner and Zhang proved existence of strong solutions to the STNSE on $\R^3$ by first establishing existence of martingale solutions and pathwise uniqueness and then applying the Yamada--Watanabe theorem, with initial data in the divergence-free vector fields in $H^1$.
In \cite{MR2852224}, R\"ockner and T.
Zhang also proved existence and uniqueness of strong solutions to the STNSE on bounded domains in $\R^3$ with Dirichlet boundary conditions and $H^1$ initial data, using a Galerkin approximation together with a local monotonicity property of the coefficients.
Brze\'zniak and Dhariwal \cite{MR4085355} generalized the work of R\"ockner and Zhang by a different method, based on a truncated SPDE in an infinite-dimensional space, tightness arguments, the Skorokhod--Jakubowski theorem, and the Yamada--Watanabe theorem.
Dong and Zhang \cite{MR4127295} studied the well-posedness of the STNSE with multiplicative L\'evy noise for initial data $u_0\in L^2(\Omega;H^1)$ with periodic boundary condition, using a Galerkin scheme together with monotonicity properties of the coefficients to obtain existence and uniqueness of strong solutions.

We next recall several developments concerning existence and uniqueness for the SNSE.
Early contributions, dating back to the 1970s, were primarily formulated in Hilbert-space settings; see, for example, \cite{MR348841,MR1922695, MR2459085, MR3049076} for results in both two and three dimensions. In \cite{MR2865433}, Kim established local strong well-posedness for the three-dimensional SNSE on $\R^3$ with initial data in the subcritical Sobolev class $H^{1/2+\alpha}(\R^3)$, $0<\alpha<1/2$. For multiplicative non-degenerate noise, he further showed that, if the initial datum is sufficiently small, the corresponding strong solution is global with large probability.
More recently, Chen et al. \cite{MR3912800} proved the existence of martingale solutions for the three-dimensional stochastic nonhomogeneous incompressible NSE equations on bounded domains in $\R^3$, driven by L\'evy noise comprising Brownian, compensated-Poisson, and Poisson components. The velocity is treated in the energy framework $L^\infty(0,T;L^2)\cap L^2(0,T;H^1)$, while the construction relies on Faedo--Galerkin approximations, stopping-time estimates, Aubin--Simon compactness, stochastic tightness, and the Jakubowski--Skorokhod representation theorem.

The NSE with a damping term of the form $u^{\beta-1}\cdot u$ was studied by Cai et al. \cite{MR2401535}, who proved existence for $\beta\ge\frac{7}{2}$ and uniqueness of strong solutions for $\frac{7}{2}\le\beta \le5$.
Zhang et al. \cite{MR2754840} significantly improved this result to $\beta>3$ on $\R^3$ for initial data $u_0\in V\cap L^{\beta+1}$ for strong solution and uniqueness $3< \beta \le 5$ .

Recent developments for the SNSE with initial data in non-Hilbert spaces are particularly relevant to the present work.
Mohan and Sritharan \cite{MR3669657} established a local solution theory for the SNSE on $\R^m$, $m\ge2$, for initial data in $L^m(\R^m)$ for additive Wiener and Compensated Poison forcing.
Du and Zhang \cite{MR4002154} established local strong well-posedness for the SNSE on $\R^d$ with multiplicative Gaussian noise and initial data in critical homogeneous Besov spaces $\dot B_{q,r}^{d/q-1}$, for admissible parameters.
They also obtained global existence in probability for sufficiently small initial data in the critical norm.
In \cite{MR4703457}, Agresti and Veraar studied the SNSE with transport-type noise in the scaling-critical Besov spaces $B_{q,p}^{{d}/{q} -1}$, where $q\in[2,2d]$ and $p$ large enough.
They proved local well-posedness, established stochastic Serrin-type blow-up criteria, and derived a global existence theory.
Agresti and Veraar \cite{MR4716343} developed a critical variational framework in which suitable local Lipschitz assumptions replace the usual local monotonicity requirement.
Their applications include global well-posedness for STNSE on $\R^3$ with $H^1$ initial data and gradient-dependent Wiener forcing under the stated structural assumptions.

Kukavica et al.
further developed the $L^p$ theory for the SNSE .
In \cite{MR4385406}, local solutions for the SNSE driven by multiplicative noise on $\mathbb T^3$ were established for initial data in $L^p$, $p>5$,where the global solution is obtained for the small initial data with same condition in $p$ .
Kukavica et al.
\cite{MR4552356} subsequently proved local existence of strong solutions for initial data in $L^p$, $p>3$, thereby improving the previous range.
Their construction uses a spectral Galerkin scheme based on rectangular partial sums in Fourier space to obtain approximate solutions in $H^1(\mathbb T^3)$, which are then extended to the $L^p$ setting.
More recently, \cite{MR4908978} established local well-posedness for the three-dimensional SNSE on $\R^3$ with Gaussian noise and initial data in the same $L^p$ range, $p>3$, using the convolution-type Fourier multiplier ${\mathrm P}_{\leq k}$.

Discontinuous random perturbations also play an important role in stochastic fluid models.
While Gaussian noise describes continuous random fluctuations, L\'evy noise provides a natural mechanism for representing sudden and non-Gaussian disturbances.
Maximal inequalities for Poisson and L\'evy stochastic integrals in Banach spaces were developed by Hausenblas in \cite{MR2832576}, while optimal $L^p$-valued estimates for Poisson stochastic integrals were later obtained by Dirksen in \cite{MR3265175}.
Maximal inequalities for stochastic convolutions driven by compensated Poisson random measures in Banach spaces were established by Zhu, Brze\'zniak and Hausenblas in \cite{MR3634281}.
Their work shows that the jump-noise setting requires estimates distinct from those used for Wiener forcing and provides c\`adl\`ag modifications together with maximal bounds for compensated-Poisson stochastic convolutions.
Bechtel, Germ and Veraar \cite{MR5120871} extended the critical variational framework to SPDEs with L\'evy noise, obtaining maximal local solutions, blow-up criteria and global well-posedness under coercivity conditions.
Their framework also extends variational applications to tamed NSE equations on bounded and unbounded domains.
Motyl \cite{MR3034603} studied SNSE in two- and three-dimensional, possibly unbounded domains, driven simultaneously by a cylindrical Wiener process and a compensated Poisson random measure.
Martingale solutions were obtained by means of Faedo--Galerkin approximation, tightness, compactness, and a Skorokhod representation theorem in non-metric spaces.
SNSE driven by jump noise in $L^p$ spaces were studied by Fernando, R\"udiger and Sritharan in \cite{MR3411979}, who constructed local mild solutions using the Stokes semigroup framework and compensated Poisson random measures.
On the analytic side, Brze\'zniak and Hausenblas developed maximal regularity estimates for stochastic convolutions driven by time-homogeneous Poisson random measures and L\'evy-type noise in martingale-type Banach spaces in \cite{MR2529441}.
In \cite{MR3991628}, Zhu, Brze\'zniak and Liu developed an $L^p$ theory for two-dimensional SNSE driven by compensated Poisson jump noise.
They proved existence and uniqueness under spatially irregular jump perturbations and negative-order Sobolev initial data, showing that an $L^p$-based approach can accommodate weaker noise assumptions than the standard Galerkin framework.
Gy\"ongy and Wu \cite{MR4165650} established It\^o formulas for $L^p$-valued processes with jumps, extending the corresponding continuous-noise results.
Their formulas provide a reference for $L^p$ energy identities in SPDEs driven by Wiener and compensated-Poisson noise.

Kukavica and Xu \cite{MR3950967} established local existence for the stochastic Euler equations with L\'evy noise in $\R^d$, working in the Sobolev class $W^{s,p}$, $s>\frac{d}{p}+1$, $p>2$, and using mollified divergence-free approximations together with stopping times based on $\|u\|_\infty+\|\nabla u\|_\infty$.
Their estimates combine commutator bounds, Calder\'on--Zygmund estimates for the projected pressure contribution, and Burkholder--Davis--Gundy type controls for the Wiener and jump terms.
Mohan and Sritharan \cite{MR3942498} considered three-dimensional SNSE driven by L\'evy noise with hereditary viscosity and proved local solvability via a cut-off approximation, local monotonicity, and a stochastic Minty--Browder argument.
Mohan \cite{MR4439993} studied stochastic convective Brinkman--Forchheimer equations driven by multiplicative pure jump noise in bounded and periodic domains. Using monotonicity properties of the drift together with a stochastic Minty--Browder argument, he established global pathwise unique strong solutions for initial data $u_0\in L^2(\Omega;H)$, and, in periodic domains under additional regularity assumptions on the noise, higher-order regularity for $u_0\in L^2(\Omega;V)$.
\subsection{Main results and contributions}
The results reviewed above leave a specific whole-space problem unresolved.
The $L^p(\R^3)$ theory of Kukavica, Wang and Xu \cite{MR4908978} establishes local well-posedness for the three-dimensional SNSE with multiplicative cylindrical Wiener forcing, whereas the available L\'evy-driven theories cited above are formulated in different spatial regularity classes or solution frameworks.
The present work establishes a local $L^p$ theory for the STNSE on $\R^3$ with simultaneous cylindrical Wiener and compensated Poisson forcing, for $p>3$.
The results comprise existence, pathwise uniqueness, localized continuous dependence on the initial datum, and the construction of a unique maximal local strong solution.
The localization and pasting arguments retain the predictable pre-jump state and the possible jump at each terminal stopping time.

Under the coefficient assumptions stated in \cref{sec2}, namely \eqref{eq:2_WienerH}, \eqref{Hypo_W2}, \eqref{eq:4.G1}--\eqref{eq:4.G3}, and \eqref{eq:4_G5}, we first prove existence and pathwise uniqueness of a local strong solution of \eqref{eq:1_Main}.
More precisely, there exist an $(\mathcal F_t)$-stopping time $\tau$ with $\P(\tau>0)=1$ and an adapted $L^p$-valued c\`adl\`ag process $u$ satisfying \eqref{eq:1_Main} such that $ u \in L^p\bigl( \Omega;\D([0,\tau];L^p) \bigr) \cap L^p\bigl( \Omega;L^p(0,\tau;L^{3p}) \bigr) $ and
\begin{equation*}
\E \Bigg[ \sup_{0\le s\le\tau} \|u(s)\|_p^p + \int_0^\tau \left\| \nabla\bigl(|u(s)|^{p/2}\bigr) \right\|_2^2\,\dd s \Bigg] <\infty .
\end{equation*}

We then pass from this local solution to a unique maximal local strong solution
\begin{equation*}
\bigl( u,\{\tau_n\}_{n\ge1},\tau_{\max} \bigr), \qquad \tau_n\uparrow\tau_{\max} \quad \P\text{-a.s.},
\end{equation*}
where $\tau_{\max} = \operatorname*{ess\,sup}\mathscr T$ and $\mathscr T$ denotes the family of attainable local lifetimes.
The process is canonically defined on $[0,\tau_{\max})$, every $(u,\tau_n)$ belongs to the same local-energy class, and every other admissible local strong solution is a restriction of this maximal process up to its own lifetime.

The proof begins from the whole-space Fourier-multiplier strategy of \cite{MR4908978}, but the presence of the taming term and, more importantly, of the compensated Poisson forcing requires a different localization and convergence analysis.
We use the smooth convolution Fourier multipliers ${\mathrm P}_{\le k}$ (see \cref{sec2}) rather than a finite-dimensional eigenfunction expansion.
At a fixed scalar cutoff level $R$, the nonlinear and stochastic coefficients are regularized and the resulting equation is solved globally.
The estimates used to remove the spatial regularization are uniform in the spectral parameter, but they are not required to be uniform in $R$.
This distinction is important because the scalar cutoff is an auxiliary device for the local construction, whereas the maximal solution is subsequently obtained by pathwise uniqueness and stochastic pasting, not by an $R$-uniform global estimate.

The jump noise changes the localization mechanism in an essential way.
If a c\`adl\`ag trajectory reaches a localizing level by a jump, its value at the stopping time may overshoot that level.
Accordingly, we never use a deterministic post-jump bound at an exit time.
The localized estimates are instead formulated through the strict pre-exit bound for $u(s)$ together with the bound for the predictable left limit $u(s-)$.
In particular, for the localized stopping times used in the approximation argument, the relevant control has the form $ \sup_{0\le s<\tau}\|u(s)\|_p^p<\infty, \ \sup_{0\le s\le\tau}\|u(s-)\|_p^p<\infty, $ while the possible terminal jump is retained in the stopped Poisson integral.
The compensated-Poisson estimates are carried out simultaneously in the two natural $L^2(Z,\mu;L^p)$ and $L^p(Z,\mu;L^p)$ modes.
The Taylor-remainder martingale is controlled without imposing an additional fourth-moment assumption on the jump coefficient.

At the approximation level, the Cauchy argument separates genuine solution-difference terms from Fourier-multiplier defects.
The latter are estimated through the quantitative multiplier bound $ \| {\mathrm P}_{\le n}F - {\mathrm P}_{\le m}F \|_{L^q} \lesssim \left| \frac1n-\frac1m \right| \|\nabla F\|_{L^q}.
$ The spatial truncation of the initial datum gives finite $L^2$-energy for each approximation, but the corresponding $L^2$ bound need not be uniform in the spatial truncation index.
This causes no obstruction as the decay of the multiplier defect chosen sufficiently fast to overcome the finite $L^2$ contribution.
The resulting subsequence is Cauchy in the strong local path norm and converges almost surely to the desired $L^p$-valued c\`adl\`ag solution.

The second part of the argument concerns maximality and is structural rather than a priori.
For a bounded stopping time $\rho$, we prove a restart theorem on the shifted filtration $ \mathcal F_s^\rho=\mathcal F_{\rho+s}, \ W^\rho(s)=W(\rho+s)-W(\rho), $ under the stated independent-increment compatibility of the driving noises with the filtration.
For measurable $A\subset Z$ with $\mu(A)<\infty$, the shifted Poisson random measure is defined by $ N^\rho((0,s]\times A) = N((\rho,\rho+s]\times A).
$ The half-open interval $(\rho,\rho+s]$ is essential because of a possible jump at $\rho$ is already contained in the restart datum and is therefore excluded from the restarted Poisson integral, while a jump at the terminal endpoint is retained.

We next establish a stochastic pasting theorem for arbitrary local solutions driven by the same noises.
Rather than choosing a branch through terminal comparison events, the pasted process is written algebraically.
If $ \theta=\tau_1\wedge\tau_2, \ \tau=\tau_1\vee\tau_2, $ then $ u(t) = u^{(1)}(t\wedge\tau_1) + u^{(2)}(t\wedge\tau_2) - u^{(2)}(t\wedge\theta).
$ This representation makes adaptedness immediate from the stopped processes themselves.
In the weak formulation the additional branch is supported by $ \mathds 1_{\{\tau_1<s\le\tau_2\}} = \mathds 1_{\{s\le\tau_2\}} - \mathds 1_{\{s\le\tau_1\wedge\tau_2\}}.
$ So a possible Poisson jump at the switching time is counted exactly once and a possible terminal jump is preserved.
The pasting theorem implies that the family $\mathscr T$ of attainable lifetimes is upward directed.
A countable essential-supremum argument then produces an increasing sequence $ \tau_n\in\mathscr T, \ \tau_n\uparrow \operatorname*{ess\,sup}\mathscr T = \tau_{\max}, $ and pathwise uniqueness identifies the corresponding local branches with a single maximal process on $[0,\tau_{\max})$.

Thus the endpoint of the present paper is therefore maximal local well-posedness.
The dependence result compares solutions on common localized stochastic intervals and does not assert stability of the maximal lifetime.
No value of $u(\tau_{\max})$ is imposed on $\{\tau_{\max}<\infty\}$, and no continuation criterion at $\tau_{\max}$ is asserted here.
The maximal solution constructed above provides the canonical starting point for the second part of the work, where continuation and non-explosion are studied under additional regularity assumptions.

\subsection{Organization of the paper}

The paper is organized as follows.
In \cref{sec2} we introduce the stochastic framework, the Wiener and compensated-Poisson coefficient assumptions, the notion of local strong solution, and the convolution Fourier multipliers used throughout the approximation.
We also record the multiplier estimates and the stochastic-integration notation required later.

In \cref{sec3} we establish the $L^p$ theory for the linear stochastic heat equation with both Wiener and compensated-Poisson forcing.
The resulting estimate controls the path supremum together with the dissipation $ \int \left\| \nabla\bigl(|u|^{p/2}\bigr) \right\|_2^2 $ and provides the linear estimate used in the nonlinear construction.

In \cref{sec4} we study the spatially regularized and $L^p$-truncated STNSE equation.
For fixed spectral and scalar truncation parameters we prove existence and uniqueness of the auxiliary strong solution and derive the estimates required for the subsequent limiting procedure.

In \cref{sec5} we remove the spatial regularization at a fixed scalar cutoff level.
Estimates uniform in the spectral parameter and a Cauchy argument yield an almost surely convergent local approximation.
We identify the limiting equation, prove pathwise uniqueness, and remove the scalar cutoff to obtain a local strong solution of \eqref{eq:1_Main}.
We then establish a localized Lipschitz estimate with respect to the initial datum.

Finally, in \cref{sec:maximal-solution} we prove restart at bounded stopping times and establish stochastic pasting, retaining possible Poisson jumps at the switching and terminal times.
The family of attainable lifetimes is shown to be upward directed, and an essential-supremum argument yields the unique maximal local strong solution.
We conclude by transferring the localized dependence estimate to the maximal solutions on common localized stochastic intervals.
\section{Preliminaries and main results}
\label{sec2}

Throughout the paper, $\N$ denotes the positive integers and $C$ a generic positive constant, whose relevant dependencies are indicated when necessary.
For a Banach space $X$, its norm and dual are denoted by $\|\cdot\|_X$ and $X^\ast$, respectively.

\subsection{Analytic notation}

We write $C_c^\infty(\R^d)$, $\mathcal D'(\R^d)$, $\mathcal S(\R^d)$, and $\mathcal S'(\R^d)$ for the usual spaces of test functions, distributions, Schwartz functions, and tempered distributions.
With the Fourier convention $ \widehat f(\xi) = \int_{\R^d} e^{-2\pi i\xi\cdot x}f(x)\,\dd x, $ the Bessel potential operator is $ J^sf = \mathcal F^{-1} \left[ (1+4\pi^2|\xi|^2)^{s/2}\widehat f \right], $ and $$ W^{s,q}(\R^d;X) := \left\{ f\in\mathcal S'(\R^d;X): J^sf\in L^q(\R^d;X) \right\}, \ 1<q<\infty.
$$

\subsubsection{The smoothing multiplier \texorpdfstring{${\mathrm P}_{\le n}$}{P\_\{<=n\}}}

Let $ \psi(\xi):=e^{-|\xi|^2}, \ \psi_n(\xi):=\psi(\xi/n), $ and define
\begin{equation*}
{\mathrm P}_{\le n}f := \mathcal F^{-1} \bigl( \psi_n\widehat f \bigr).
\end{equation*}
Writing $ K:=\mathcal F^{-1}\psi = \pi^{d/2}e^{-\pi^2|\cdot|^2}, \ K_n(x):=n^dK(nx), $ we have $ {\mathrm P}_{\le n}f=K_n*f, \quad\|K_n\|_{L^1}=\|K\|_{L^1}=1.
$ Thus ${\mathrm P}_{\le n}$ is a scalar-kernel smoothing Fourier multiplier.
In particular, it is not a projection.
It acts componentwise, or in the Bochner sense, on vector-, tensor-, and Banach-valued functions.

The Helmholtz--Leray projector is $ (\mathcal Pu)_j = \sum_{\ell=1}^d (\delta_{j\ell}+R_jR_\ell)u_\ell, \ j=1,\ldots,d, $ and extends boundedly to $W^{s,q}(\R^d;\R^d)$ for $s\in\R$ and $1<q<\infty$.
When acting on tensor-valued functions it is understood componentwise.
Since ${\mathrm P}_{\le n}$, $\mathcal P$, and the spatial derivatives are Fourier multipliers, they commute on their common domains.

\subsection{Stochastic framework}

Let $ (\Omega,\mathcal F,(\mathcal F_t)_{t\ge0},\P) $ be a complete filtered probability space satisfying the usual conditions, and let $\mathscr P$ denote the predictable $\sigma$-algebra.
Let $\mathcal U$ be a separable real Hilbert space and let $W$ be an $(\mathcal F_t)$-adapted cylindrical $\mathcal U$-Wiener process.

Throughout the local $L^p$-theory, $ E:=L^p(\R^3;\R^3), \ p>3.
$ The space $E$ is UMD.
For $1<q<\infty$, the $\gamma$-Fubini isomorphism gives $ \gamma \bigl( \mathcal U; L^q(\R^3;\R^3) \bigr) \simeq L^q \bigl( \R^3; \gamma(\mathcal U;\R^3) \bigr).
$ We therefore set $$ \mathbb W^{s,q} := W^{s,q} \bigl( \R^3; \gamma(\mathcal U;\R^3) \bigr), \qquad \mathbb L^q:=\mathbb W^{0,q}, $$ and identify $\mathbb L^q$ with $\gamma(\mathcal U;L^q(\R^3;\R^3))$ up to equivalence of norms depending only on $q$.

The boundedness of $\mathcal P$ on the spatial Sobolev scale and the ideal property of $\gamma$-radonifying operators imply $ \|\mathcal PR\|_{\gamma(\mathcal U;W^{s,q})} \le C_{s,q} \|R\|_{\gamma(\mathcal U;W^{s,q})}.
$ We write $ \mathbb W_{\mathrm{sol}}^{s,q} := \mathcal P\mathbb W^{s,q}.
$

If $\Phi$ is an admissible progressively measurable $\mathbb L^p$-valued process, its stochastic integral with respect to $W$ is understood in the UMD/$\gamma$-radonifying sense.
For $r\ge2$,
\begin{equation}
\label{eq:2-W-BDG}
\E \sup_{t\le T} \left\| \int_0^t \Phi(s)\,\dd W_s \right\|_p^r \le C_{p,r} \E \left( \int_0^T \|\Phi(s)\|_{\mathbb L^p}^2\,\dd s \right)^{r/2}.
\end{equation}
We refer to \cite{MR2330977,MR3617205} for the stochastic integration theory in UMD spaces.

Let $(Z,\mathcal Z,\mu)$ be a $\sigma$-finite measure space and $N$ a Poisson random measure on $\R_+\times Z$, independent of $W$, with compensator $\dd t\otimes\mu(\dd z)$.
We write $ \widetilde N(\dd t,\dd z) = N(\dd t,\dd z)-\mu(\dd z)\dd t .
$ No separate L\'evy-measure assumption is imposed on $\mu$; the required integrability is encoded directly in the coefficient $G$.

For a predictable $E$-valued integrand $H$ and $p\ge2$, the $L^p$-valued Bichteler--Jacod inequality, equivalently the Poisson It\^o-isomorphism in the present range \cite{MR3265175}, yields
\begin{equation*}
\begin{aligned}
\E \sup_{t\le T} \left\| \int_0^t\int_Z H(s,z)\,\widetilde N(\dd s,\dd z) \right\|_p^p \le& C_p \E \left( \int_0^T\int_Z \|H(s,z)\|_p^2 \,\mu(\dd z)\,\dd s \right)^{p/2} + C_p \E \int_0^T\int_Z \|H(s,z)\|_p^p \,\mu(\dd z)\,\dd s .
\end{aligned}
\end{equation*}
These are precisely the two Poisson norms used below.
We refer to Kallenberg \cite{MR1464694} for the standard measurable and martingale structure of Poisson random measures.

\subsection{Noise coefficients}

The stochastic forcing in \eqref{eq:1_Main} is written as $ \mathcal P\sigma(u(t))\,\dd W_t + \int_Z \mathcal PG(t,u(t-),z)\, \widetilde N(\dd t,\dd z).
$ The dependence of $\sigma$ and $G$ on $(\omega,t)$ is suppressed whenever no ambiguity arises.

The Wiener coefficient $ \sigma: \Omega\times\R_+\times E \longrightarrow \mathbb L^p $ is assumed measurable with respect to $\mathscr P\otimes\mathcal B(E)$ in the variables for which a predictable representative is required.
We impose
\begin{equation}
\label{eq:2_WienerH}
\begin{aligned}
\|\sigma(u)\|_{\mathbb L^p} &\le C \left( 1+\|u\|_{(\frac{3p}{2})^-}^2 \right), \\
\|\sigma(u_1)-\sigma(u_2)\|_{\mathbb L^p} &\le C \left\| (|u_1|+|u_2|)^{1/2}|u_1-u_2| \right\|_p, \\
\|\nabla\sigma(u)\|_{\mathbb L^p} &\le C \left( 1+\|u\|_{\frac{3p}{2}}^2 \right).
\end{aligned}
\end{equation}
Here $(\frac{3p}{2})^-$ denotes a fixed exponent $q_\sigma<3p/2$, chosen sufficiently close to $3p/2$ whenever required by the interpolation arguments; the constant may depend on $q_\sigma$.

At the $L^2$-level we additionally assume
\begin{equation}
\label{Hypo_W2}
\|\sigma(u)\|_{\mathbb L^2} \le C(1+\|u\|_2).
\end{equation}

The jump coefficient $ G: \Omega\times\R_+\times E\times Z \longrightarrow E $ is $\mathscr P\otimes\mathcal B(E)\otimes\mathcal Z$-measurable.
For $r\in\{2,p\}$ we assume
\begin{equation}
\label{eq:4.G1}
\int_Z \|G(t,u,z)\|_p^r \,\mu(\dd z) \le C \bigl( 1+\|u\|_p^r \bigr),
\end{equation}
\begin{equation}
\label{eq:4.G2}
\int_Z \|G(t,u_1,z)-G(t,u_2,z)\|_p^r \,\mu(\dd z) \le C \left\| (|u_1|+|u_2|)^\alpha |u_1-u_2| \right\|_p^r, \qquad 0\le\alpha<\frac23,
\end{equation}
and
\begin{equation}
\label{eq:4.G3}
\int_Z \|\nabla G(t,u,z)\|_p^r \,\mu(\dd z) \le C \left( 1+\|u\|_{\frac{3p}{2}}^r \right).
\end{equation}
The $L^2$-growth assumption used in the persistence argument is
\begin{equation}
\label{eq:4_G5}
\int_Z \|G(t,u,z)\|_2^2 \,\mu(\dd z) \le C \left( 1+\|u\|_2^2 \right).
\end{equation}

If $u$ is adapted and c\`adl\`ag, then $u_-$ is predictable; hence $ (\omega,t,z) \longmapsto G(\omega,t,u(t-,\omega),z) $ is an admissible Poisson integrand.
The Wiener coefficient may be represented by a predictable version without changing the stochastic integral, since $u=u_-$ $\dd t\otimes\P$-a.e.
\subsubsection{A nonempty class of admissible coefficients}

We record a simple class of coefficients covered by the preceding hypotheses.
The construction separates the required state dependence from the spatial regularity.
A bounded nonlinear map provides the weighted Lipschitz structure, while convolution with a Schwartz kernel supplies the spatial derivatives without imposing any derivative on the velocity field..
For $\gamma\in[0,2/3)$, let $V_\gamma:\R^3\to\R^3$ be bounded and $C^1$ and suppose that
\begin{equation}
\label{eq:2-Theta}
\begin{aligned}
|V_\gamma(\xi)| &\le C\min \{|\xi|^{1+\gamma},1\}, \\
|V_\gamma(\xi)-V_\gamma(\eta)| &\le C(|\xi|+|\eta|)^\gamma|\xi-\eta|, \qquad \xi,\eta\in\R^3.
\end{aligned}
\end{equation}
For $\gamma=0$, the factor $(|\xi|+|\eta|)^\gamma$ is understood as $1$.
A model choice is $$ V_\gamma(\xi) = \frac{|\xi|^\gamma\xi}{1+|\xi|^{1+\gamma}}, \qquad V_\gamma(0)=0.
$$ Indeed, $$ |V_\gamma(\xi)| = \frac{|\xi|^{1+\gamma}}{1+|\xi|^{1+\gamma}} \le \min\{|\xi|^{1+\gamma},1\}, $$ while a direct differentiation gives $ \|DV_\gamma(\xi)\| \le C|\xi|^\gamma, \qquad \xi\in\R^3, $ with the preceding convention when $\gamma=0$; hence the second estimate in \eqref{eq:2-Theta} follows from the mean-value formula.
For the Wiener coefficient we use $V_{1/2}$.
Let $\kappa_\sigma,a_\sigma\in\mathcal S(\R^3),\ b_\sigma\in\mathcal S(\R^3;\R^3), $ let $\beta_\sigma$ be a bounded predictable scalar process, and let $\ell\in\mathcal U^\ast$.
Define $ S_\sigma(u) := \mathcal P\left[ \kappa_\sigma* \bigl(a_\sigma V_{1/2 }(u)+b_\sigma\bigr) \right] $ and
\begin{equation*}
\sigma(\omega,t,u)h := \beta_\sigma(\omega,t)\,\ell(h)\,S_\sigma(u), \qquad h\in\mathcal U.
\end{equation*}
Thus $\sigma(\omega,t,u)$ is a rank-one $\gamma(\mathcal U;E)$-valued operator.
For $q\in\{2,p\}$, the rank-one identity for $\gamma$-radonifying operators, together with the ideal property, gives $$ \|\sigma(\omega,t,u)\|_{\mathbb L^q} = |\beta_\sigma(\omega,t)| \,\|\ell\|_{\mathcal U^\ast} \,\|S_\sigma(u)\|_q.
$$

For the jump coefficient, fix the exponent $\alpha\in[0,2/3)$ appearing in \eqref{eq:4.G2}, and choose $ \kappa_G,a_G\in\mathcal S(\R^3), \ b_G\in\mathcal S(\R^3;\R^3).$ Let $ \lambda: \Omega\times\R_+\times Z\longrightarrow\R $ be $\mathscr P\otimes\mathcal Z$-measurable and assume that
\begin{equation}
\label{eq:2-example-lambda}
L_r := \operatorname*{ess\,sup}_{(\omega,t)} \int_Z |\lambda(\omega,t,z)|^r\,\mu(\dd z) <\infty, \qquad r\in\{2,p\},
\end{equation}
where the essential supremum is taken with respect to $\P\otimes\dd t$.
Set $$ S_G(u) := \mathcal P\left[ \kappa_G* \bigl(a_GV_\alpha(u)+b_G\bigr) \right], \qquad G(\omega,t,u,z) := \lambda(\omega,t,z)S_G(u).
$$

In particular, if $B\in\mathcal Z$ satisfies $0<\mu(B)<\infty$, then $ \lambda(\omega,t,z)=\mathds 1_B(z) $ satisfies \eqref{eq:2-example-lambda} and yields a concrete nontrivial jump coefficient.
Finite sums of the above Wiener coefficients, and finite sums of the above jump coefficients with a common exponent $\alpha$ in \eqref{eq:4.G2}, satisfy the same class of assumptions after an adjustment of the structural constants.
\subsection{C\`adl\`ag path space and strong solutions}

Let $X$ be a separable Banach space.
We denote by $\D([0,T];X)$ the space of $X$-valued c\`adl\`ag paths, equipped with the standard Skorokhod $J_1$ topology whenever a path-space topology is invoked.
For $v\in\D([0,T];X)$, $ v(t-):=\lim_{s\uparrow t}v(s), \ t\in(0,T], \ v(0-):=v(0).
$

Throughout the paper, $ L^p\bigl(\Omega;\D([0,T];X)\bigr) $ denotes the class of adapted processes admitting c\`adl\`ag $X$-valued trajectories and satisfying
\begin{equation}
\label{eq:cadlag-Lp-convention}
\E \sup_{t\le T} \|u(t)\|_X^p <\infty.
\end{equation}
Processes are identified up to indistinguishability, and $ u_n\to u \ \text{in}\ L^p\bigl(\Omega;\D([0,T];X)\bigr) $ means
\begin{equation}
\label{eq:cadlag-convergence}
\E \sup_{t\le T} \|u_n(t)-u(t)\|_X^p \longrightarrow0.
\end{equation}
This is a process-space convention, not a Bochner-space identification with $\D([0,T];X)$.
In particular, \eqref{eq:cadlag-convergence} is stronger than convergence in probability in the Skorokhod $J_1$ topology.

Consider
\begin{equation}
\label{eq:2.12}
u(t) = u_0 + \int_0^t \bigl( Au(s)+f(s) \bigr)\,\dd s + \int_0^t \mathcal P\sigma(u(s))\,\dd W_s + \int_0^t\int_Z \mathcal PG(s,u(s-),z) \,\widetilde N(\dd s,\dd z).
\end{equation}

\begin{definition}
[Local strong solution] A pair $(u,\tau)$ is a local strong solution of \eqref{eq:2.12} on the prescribed stochastic basis if:

\begin{enumerate}
\item $\tau$ is a stopping time and $\P(\tau>0)=1$;

\item $u$ is adapted and $u(\cdot\wedge\tau)$ is $E$-valued c\`adl\`ag;

\item for every $\varphi\in C_c^\infty(\R^3;\R^3)$ and $t\in[0,T]$,
\begin{equation}
\label{eq:2-strong-weak}
\begin{aligned}
\langle u(t\wedge\tau),\varphi\rangle =& \langle u_0,\varphi\rangle + \int_0^t \mathds 1_{\{s\le\tau\}} \Bigl[ \langle u(s),A^\ast\varphi\rangle + \langle f(s),\varphi\rangle \Bigr]\,\dd s \\
&+ \int_0^t \mathds 1_{\{s\le\tau\}} \left\langle (\mathcal P\sigma(u(s)))^\ast\varphi, \,\dd W_s \right\rangle_{\mathcal U} + \int_0^t\int_Z \mathds 1_{\{s\le\tau\}} \langle \mathcal PG(s,u(s-),z), \varphi \rangle \,\widetilde N(\dd s,\dd z),
\end{aligned}
\end{equation}
almost surely.
\end{enumerate}
\end{definition}

Here $\mathds 1_{\{s\le\tau\}}$ is predictable, so the possible jump at the terminal time is included in \eqref{eq:2-strong-weak}.
The term ``strong'' refers to the probabilistic sense.
The solution is constructed on the prescribed stochastic basis with the prescribed noises $W$ and $\widetilde N$.

\begin{definition}
[Pathwise uniqueness] Pathwise uniqueness holds if any two local strong solutions $(u,\tau)$ and $(v,\eta)$ with the same initial datum and the same driving noises satisfy $$ \P\left( u(t)=v(t) \text{ for every } t\in[0,\tau\wedge\eta] \right) =1.
$$
\end{definition}

\subsection{Main results}

\begin{theorem}
[Existence and pathwise uniqueness of a local strong solution] \label{eq:2_Main_thm} Let $p>3$ and let $ u_0 \in L^p\bigl( \Omega, \mathcal F_0; L^p(\R^3;\R^3) \bigr), \ \nabla\cdot u_0=0 \quad \P\text{-a.s.} $ Assume \eqref{eq:2_WienerH}, \eqref{Hypo_W2}, \eqref{eq:4.G1}--\eqref{eq:4.G3}, and \eqref{eq:4_G5}.
Then there exist an $(\mathcal F_t)$-stopping time $\tau$ and an adapted $L^p(\R^3;\R^3)$-valued c\`adl\`ag process $u$ such that $ \P(\tau>0)=1, $ $ u \in L^p\bigl( \Omega; \D([0,\tau];L^p) \bigr) \cap L^p\bigl( \Omega; L^p(0,\tau;L^{3p}) \bigr), $ and $(u,\tau)$ is a local strong solution of \eqref{eq:1_Main}.
Moreover,
\begin{equation}
\label{eq:2-main-energy}
\begin{aligned}
\E \Bigg[ & \sup_{0\le s\le\tau} \|u(s)\|_p^p + \int_0^\tau \left\| \nabla\bigl(|u(s)|^{p/2}\bigr) \right\|_2^2\,\dd s \Bigg] \le C \left( 1+\E\|u_0\|_p^p \right),
\end{aligned}
\end{equation}
where $C$ is independent of all auxiliary scalar and spatial regularization parameters.
The solution is pathwise unique in the class of local strong solutions satisfying the preceding local energy regularity.
\end{theorem}
\begin{proposition}
[Existence and uniqueness of a maximal local strong solution] \label{prop:6.4-maximal} Let $p>3$, and assume the hypotheses of \cref{eq:2_Main_thm}.
Fix the stochastic basis, the initial datum $u_0$, and the driving Wiener process and Poisson random measure.
Let
\begin{equation*}
\mathscr T :=\left\{\sigma:
\begin{array}
{l} \text{there exists a local strong solution }(v,\sigma) \text{ of \eqref{eq:1_Main}} \\
\text{with the prescribed initial datum and driving noises,} \\
\text{belonging to the local-energy uniqueness class}
\end{array}
\right\}.
\end{equation*}
Then there exist an $(\mathcal F_t)$-stopping time $\tau_{\max}$, an increasing sequence $(\tau_n)_{n\ge1}\subset\mathscr T$, and an $E$-valued progressively measurable process $u$ such that
\begin{equation}
\label{eq:6.4-maximal-time}
\tau_{\max} =\operatorname*{ess\,sup}_{\sigma\in\mathscr T}\sigma, 0<\tau_n\le n,\qquad \tau_n\uparrow\tau_{\max} \quad\P\text{-a.s.}
\end{equation}
In particular,
\begin{equation}
\label{eq:6.4-positive-maximal-time}
\P(\tau_{\max}>0)=1.
\end{equation}
For every $n$, the stopped process $u^{\tau_n}:=u(\cdot\wedge\tau_n)$ is adapted and $E$-valued c\`adl\`ag, and $(u^{\tau_n},\tau_n)$ is a local strong solution of \eqref{eq:1_Main} in the prescribed class.
Moreover,
\begin{equation}
\label{eq:6.4-local-energy}
\E\left[ \sup_{0\le s\le\tau_n}\|u(s)\|_p^p +\int_0^{\tau_n} \left\|\nabla\bigl(|u(s)|^{p/2}\bigr)\right\|_2^2\,\dd s \right]<\infty, \qquad n\ge1.
\end{equation}
The restriction of $u$ to $[0,\tau_{\max})$ is locally $E$-valued c\`adl\`ag.

If $(v,\sigma)$ is any other local strong solution with the same initial datum and driving noises in the same class, then
\begin{equation}
\label{eq:6.4-maximality-time}
\sigma\le\tau_{\max}\quad\P\text{-a.s.},
\end{equation}
\begin{equation}
\label{eq:6.4-maximality}
\P\left( v(t)=u(t)\text{ for every }0\le t<\sigma \right)=1,
\end{equation}
\begin{equation}
\label{eq:6.4-maximality2}
v(\sigma)=u(\sigma) \quad\P\text{-a.s. on }\{\sigma<\tau_{\max}\}.
\end{equation}
The same lifetime domination and agreement before the lifetime hold for every local solution admitting an increasing sequence of closed localizations in the prescribed class.

If $(\widetilde u,(\widetilde\tau_m)_{m\ge1}, \widetilde\tau_{\max})$ is another maximal construction for the same data and noises, then
\begin{equation}
\label{eq:6.4-maximal-uniqueness-time}
\widetilde\tau_{\max}=\tau_{\max} \quad\P\text{-a.s.},
\end{equation}
and
\begin{equation}
\label{eq:6.4-maximal-uniqueness-process}
\P\left( \widetilde u(t)=u(t) \text{ for every }0\le t<\tau_{\max} \right)=1.
\end{equation}
Thus the maximal lifetime and the process on $[0,\tau_{\max})$ are intrinsic and the increasing localizing sequence is auxiliary.
\end{proposition}

\subsection{Auxiliary properties of \texorpdfstring{${\mathrm P}_{\le n}$}{P\_\{<=n\}}}

Set
\begin{equation*}
C_\psi := \int_{\R^d} |K(y)|\,|y|\,\dd y <\infty.
\end{equation*}
The following statements are used for $X=\R^3$, tensor spaces, and $X=\gamma(\mathcal U;\R^3)$.

\begin{lemma}
[Uniform boundedness] \label{lem:2.2} Let $X$ be a Banach space and $1\le q\le\infty$.
Then $$ \|{\mathrm P}_{\le n}F\|_{L^q(\R^d;X)} \le \|F\|_{L^q(\R^d;X)}, \qquad n\in\N.
$$ If $1\le r<q\le\infty$, then $ \|{\mathrm P}_{\le n}F\|_{L^q(\R^d;X)} \le C_{n,r,q} \|F\|_{L^r(\R^d;X)}.
$ Moreover, ${\mathrm P}_{\le n}$ is self-adjoint on $L^2$ whenever $X$ is a real Hilbert space.
\end{lemma}

\begin{proof}
The two estimates follow immediately from Young's convolution inequality and $\|K_n\|_1=1$.
Self-adjointness follows from $K_n(x)=K_n(-x)$.
\end{proof}

\begin{lemma}
[Strong convergence] \label{lem:2.3} Let $X$ be a Banach space and $1\le q<\infty$.
Then $$ {\mathrm P}_{\le n}F \longrightarrow F \qquad \text{in } L^q(\R^d;X) $$ for every $F\in L^q(\R^d;X)$.
\end{lemma}

\begin{proof}
This is the standard approximate-identity property of $\{K_n\}_{n\ge1}$.
\end{proof}

\begin{lemma}
[Cauchy estimate for $W^{1,q}$-data] \label{lem:2.4} Let $X$ be a Banach space, $1\le q<\infty$, and $F\in W^{1,q}(\R^d;X)$.
Then $$ \|{\mathrm P}_{\le n}F-{\mathrm P}_{\le m}F\|_{L^q(\R^d;X)} \le C_\psi \left| \frac1n-\frac1m \right| \|\nabla F\|_{L^q(\R^d;X^d)}.
$$
\end{lemma}

\begin{proof}
Using $ {\mathrm P}_{\le n}F(x)-{\mathrm P}_{\le m}F(x) = \int_{\R^d} K(y) \left[ F\left(x-\frac yn\right) - F\left(x-\frac ym\right) \right]\dd y, $ the fundamental theorem of calculus and Minkowski's inequality give $
\begin{aligned}
\|{\mathrm P}_{\le n}F-{\mathrm P}_{\le m}F\|_{L^q} &\le \left| \frac1n-\frac1m \right| \|\nabla F\|_{L^q} \int_{\R^d} |K(y)|\,|y|\,\dd y,
\end{aligned}
$ which is the asserted estimate.
\end{proof}

In particular, \cref{lem:2.2,lem:2.3,lem:2.4} apply to $\gamma(\mathcal U;\R^3)$-valued functions.
Hence every use of ${\mathrm P}_{\le n}$ on Wiener coefficients is compatible with the $\mathbb L^q$-norm, while the same lemmas apply without change to the velocity field and the tensor-valued nonlinear fluxes.
\section{Stochastic heat equation on the whole space \texorpdfstring{$\R^3$}{R3}}
\label{sec3}

Consider the stochastic heat equation
\begin{equation}
\label{eq:heat1}
\begin{aligned}
\mathrm{d}u(t,x) &= \bigl( \Delta u(t,x) + \nabla \!\cdot f(t,x) + h(t,x) \bigr) \,\mathrm{d}t + g(t,x) \,\mathrm{d}W_t + \int_Z G_0(t,x,z) \,\widetilde{N}(\mathrm{d}t, \mathrm{d}z), \\
u(0,x) &= u_0(x), \qquad \P\text{-a.s.}
\end{aligned}
\end{equation}
The linear $L^p$-theory below extends the whole-space framework of \cite[Theorem~3.1]{MR4908978} by incorporating, in addition to the divergence-form forcing $\nabla \!\cdot f$, a non-divergence source term $h$, additive Wiener forcing $ g \colon [0,T] \times \Omega \times \R^d \longrightarrow \gamma(\mathcal{U}; \R^D), $ and additive jump forcing $ G_0 \colon [0,T] \times \Omega \times Z \longrightarrow L^p(\R^d; \R^D), $ with $G_0$ assumed $\mathscr{P}_T \otimes \mathcal{Z}$-measurable.
The precise progressive and predictability assumptions, together with the corresponding $L^p$- and $L^2$-integrability requirements for the stochastic coefficients, are stated in \cref{Thm:heat-exis}.

In analogy with the low-regularity treatment of the divergence-form forcing in \cite{MR4552356}, the direct source term $h$ is admitted down to the endpoint $$ q_h \ge \frac{dp}{2p+d-2}, $$ which is exactly the range compatible with the Gagliardo--Nirenberg interpolation and the nonlinear $L^p$-dissipation used in the energy estimate.
This additional forcing is required in the subsequent analysis of the STNSE, where the taming contribution is incorporated into the linear theory through a term of the form $h$.
Under the assumptions specified below, \eqref{eq:heat1} admits a unique global strong solution $$ u \in L^p\bigl(\Omega; \D([0,T]; E)\bigr) $$ in the sense of the c\`adl\`ag path-space convention introduced above.
For the divergence-form coefficient, we use $ \|\cdot\|_{q_f} := \|\cdot\|_{L^{q_f}(\R^d; \R^{D\times d})}.
$

\begin{theorem}
[Stochastic heat equation with Wiener and jump forcing] \label{Thm:heat-exis} Let $p>2$, $d\ge2$, $D\in\N$, and $0<T<\infty$, and set $ E:=L^p(\R^d;\R^D).
$ Let $u_0$ be $\mathcal F_0$-measurable and satisfy $ u_0\in L^p(\Omega;E).
$ Assume that the divergence-form forcing $f$ and the direct forcing $h$ are progressively measurable and satisfy $ f\in L^p\!\left( \Omega\times(0,T); L^{q_f}(\R^d;\R^{D\times d}) \right), $ and $ h\in L^p\!\left( \Omega\times(0,T); L^{q_h}(\R^d;\R^D) \right).
$ Let $ g:(0,T)\times\Omega \longrightarrow L^p\!\left( \R^d; \gamma(\mathcal U;\R^D) \right) $ be predictable and satisfy $ g\in L^p\!\left( \Omega\times(0,T); \mathbb L^p \right).
$ Finally, let $ G_0: (0,T)\times\Omega\times Z \longrightarrow E $ be $\mathscr P_T\otimes\mathcal Z$-measurable and satisfy $ G_0 \in L^p\!\left( \Omega; L^p((0,T)\times Z,\dd t\otimes\mu;E) \right) \cap L^p\!\left( \Omega; L^2((0,T)\times Z,\dd t\otimes\mu;E) \right).
$

Suppose that, if $d>2$, $ \frac{dp}{p+d-2} \le q_f\le p, \ \frac{dp}{2p+d-2} \le q_h\le p, $ whereas, if $d=2$, $ 2<q_f\le p, \ 1<q_h\le p.
$

Then the stochastic heat equation
\begin{equation}
\label{eq:heat}
\begin{aligned}
\dd u(t) =& \bigl( \Delta u(t) + \nabla\!\cdot f(t) + h(t) \bigr)\,\dd t + g(t)\,\dd W_t + \int_ZG_0(t,z)\, \widetilde N(\dd t,\dd z), \\
u(0)=&u_0,
\end{aligned}
\end{equation}
admits a unique adapted global strong solution $u\in L^p\!\left( \Omega; \D([0,T];E) \right)$ and $ |u|^{p/2} \in L^2\!\left( \Omega\times(0,T); H^1(\R^d) \right).$

More precisely,
\begin{equation}
\begin{aligned}
\E \Bigg[ \sup_{t\in[0,T]} \|u(t)\|_p^p +& \int_0^T \int_{\R^d} \left| \nabla\bigl(|u(t,x)|^{p/2}\bigr) \right|^2 \,\dd x\,\dd t \Bigg]\le C\, \E \Bigg[ \|u_0\|_p^p + \int_0^T \left( \|f(t)\|_{q_f}^p + \|h(t)\|_{q_h}^p + \|g(t)\|_{\mathbb L^p}^p \right)\dd t \\
&+ \int_0^T\int_Z \|G_0(t,z)\|_p^p \,\mu(\dd z)\,\dd t + \left( \int_0^T\int_Z \|G_0(t,z)\|_p^2 \,\mu(\dd z)\,\dd t \right)^{p/2} \Bigg],
\label{ineq:3_eng-est}
\end{aligned}
\end{equation}
where $ C=C(T,d,D,p,q_f,q_h)>0 $ is independent of $u_0,f,h,g$, and $G_0$.
Equivalently, the second assertion means $$ \E \int_0^T \left( \||u(t)|^{p/2}\|_2^2 + \|\nabla|u(t)|^{p/2}\|_2^2 \right)\dd t <\infty.
$$
\end{theorem}
We shall use the following lower-semicontinuity property of the nonlinear dissipation.
The argument is adapted from \cite[Lemma~4.4]{MR4385406} and \cite[Lemma~3.2]{MR4908978}; we include the proof in the present vector-valued c\`adl\`ag setting.

\begin{lemma}
[Lower semicontinuity of the nonlinear dissipation] \label{lem:L2_con} Let $2 \le p < \infty$ and $E_{p} = L^p(\R^d; \R^D)$.
Let $u_n, u \in L^p\bigl(\Omega; \D([0,T]; E_{p})\bigr)$ for $n \in \N$, where the path-space notation is understood according to the convention \eqref{eq:cadlag-Lp-convention}.
Suppose that
\begin{equation}
\label{eq:L2-con-path}
u_n \longrightarrow u \qquad \text{in } L^p\bigl(\Omega; \D([0,T]; E_{p})\bigr),
\end{equation}
Assume moreover that
\begin{equation}
\label{eq:L2-con-grad-bound}
\sup_{n\in\N} \E \int_0^T \int_{\R^d} \left| \nabla\bigl(|u_n(t,x)|^{p/2}\bigr) \right|^2 \,\mathrm{d}x\,\mathrm{d}t < \infty.
\end{equation}
Then $ |u|^{p/2} \in L^2\left( \Omega\times(0,T); H^1(\R^d) \right), $ and
\begin{equation}
\label{eq:L2-con-liminf}
\E \int_0^T \int_{\R^d} \left| \nabla\bigl(|u(t,x)|^{p/2}\bigr) \right|^2 \,\mathrm{d}x\,\mathrm{d}t \le \liminf_{n\to\infty} \E \int_0^T \int_{\R^d} \left| \nabla\bigl(|u_n(t,x)|^{p/2}\bigr) \right|^2 \,\mathrm{d}x\,\mathrm{d}t.
\end{equation}
Here, $|\cdot|$ denotes the Euclidean norm in $\R^D$.
\end{lemma}
\begin{proof}
Set $w_n=|u_n|^{p/2}$ and $w=|u|^{p/2}$.
For $a,b\in\R^D$, $ \big||a|^{p/2}-|b|^{p/2}\big| \le C_p\big(|a|^{p/2-1}+|b|^{p/2-1}\big)|a-b|, $ and hence the assumed strong convergence, together with the uniform $L^p$-moment bounds, implies $$ w_n\to w \quad\text{strongly in }L^2(\Omega\times(0,T);L^2(\R^d)).
$$ After passing to a subsequence realizing the liminf and extracting further if necessary, the gradient bound gives $$ \nabla w_n\rightharpoonup\mathcal{I} \quad\text{weakly in } L^2(\Omega\times(0,T);L^2(\R^d;\R^d)).
$$ Passing to the limit in the distributional identity $$ \int_{\R^d}w_n\,\partial_j\varphi\,dx = -\int_{\R^d}\partial_jw_n\,\varphi\,dx, \qquad \varphi\in C_c^\infty(\R^d), $$ tested also against bounded functions of $(\omega,t)$, identifies $\mathcal{I}=\nabla w$.
Thus $w\in L^2(\Omega\times(0,T);H^1(\R^d))$, and weak lower semicontinuity yields $$ \E\int_0^T\|\nabla w(t)\|_2^2\,dt \le \liminf_{n\to\infty} \E\int_0^T\|\nabla w_n(t)\|_2^2\,dt.
$$.
\end{proof}
\begin{proof}
[Proof of \cref{Thm:heat-exis}] Throughout the proof all coefficients are understood to satisfy the standing progressive/predictability assumptions required for the deterministic, Wiener, and Poisson integrals.
We write $J_p(v):=|v|^{p-2}v,$ and denote by $(S(t))_{t\ge0}$ the heat semigroup on $E_{p}$.
Let $(\rho_\varepsilon)_{\varepsilon>0}$ be a standard family of spatial mollifiers and set $$ u_0^\varepsilon=\rho_\varepsilon*u_0,\qquad f^\varepsilon=\rho_\varepsilon*f,\qquad h^\varepsilon=\rho_\varepsilon*h, \qquad g^\varepsilon=\rho_\varepsilon*g,\qquad G_0^\varepsilon=\rho_\varepsilon*G_0.
$$ Since the convolution acts only in the spatial variable, it preserves the relevant measurability properties.
Moreover, by Young's inequality and the approximate-identity property,
\begin{align}
u_0^\varepsilon&\longrightarrow u_0 &&\text{in }L^p(\Omega;E_{p}),
\label{eq:heat-proof-moll-u0}\\
f^\varepsilon&\longrightarrow f &&\text{in } L^p(\Omega\times(0,T);L^{q_f}),
\label{eq:heat-proof-moll-f}\\
h^\varepsilon&\longrightarrow h &&\text{in } L^p(\Omega\times(0,T);L^{q_h}),
\label{eq:heat-proof-moll-h}\\
g^\varepsilon&\longrightarrow g &&\text{in } L^p(\Omega\times(0,T);\mathbb L^p),
\label{eq:heat-proof-moll-g}
\end{align}
and
\begin{align}
G_0^\varepsilon&\longrightarrow G_0 &&\text{in } L^p\!\left( \Omega; L^p((0,T)\times Z,\dd t\otimes\mu;E_{p}) \right),
\label{eq:heat-proof-moll-Gp}\\
G_0^\varepsilon&\longrightarrow G_0 &&\text{in } L^p\!\left( \Omega; L^2((0,T)\times Z,\dd t\otimes\mu;E_{p}) \right).
\label{eq:heat-proof-moll-G2}
\end{align}

For every fixed $\varepsilon>0$, the mollified coefficients possess arbitrarily high spatial regularity.
So, the linear problem
\begin{equation}
\label{eq:heat-proof-reg}
\begin{aligned}
\dd u^\varepsilon =& \bigl( \Delta u^\varepsilon +\nabla\!\cdot f^\varepsilon +h^\varepsilon \bigr)\,\dd t +g^\varepsilon\,\dd W_t +\int_ZG_0^\varepsilon(t-,z)\, \widetilde N(\dd t,\dd z), \\
u^\varepsilon(0)=&u_0^\varepsilon
\end{aligned}
\end{equation}
has a unique global adapted solution with $E_{p}$-valued c\`adl\`ag paths and sufficient spatial regularity for $ \Delta u^\varepsilon,\ \nabla J_p(u^\varepsilon) $ to be well defined in the subsequent identities.
Equivalently, this follows directly from the heat-semigroup representation of \eqref{eq:heat-proof-reg} and the standard deterministic, Wiener, and Poisson convolution estimates.
See, respectively, \cite{MR2330977} and \cite{MR3265175,MR2529441}.

The mollification is used here solely to justify the strong semimartingale formulation, the spatial integration by parts, and the application of It\^o's formula to the $L^p$-energy; c\`adl\`ag path regularity itself is not an obstruction to It\^o's formula.

Consider $$ F:E\longrightarrow\R, \qquad F(v)=\|v\|_p^p.
$$ For $p>2$, $F\in C^2(E;\R)$ and
\begin{equation}
\begin{aligned}
DF(v)w &= p\int_{\R^d}|v|^{p-2}v\cdot w\,\dd x, \\
D^2F(v)[w,\widetilde w] &= p\int_{\R^d}|v|^{p-2} w\cdot\widetilde w\,\dd x + p(p-2) \int_{\R^d} |v|^{p-4} (v\cdot w)(v\cdot\widetilde w)\,\dd x,
\end{aligned}
\end{equation}
where the second integrand is understood by continuity at $v=0$.

To avoid any issue caused by the polynomial growth of $DF$ and $D^2F$, define $$ \tau_M^\varepsilon := \inf\left\{ t\ge0: \|u^\varepsilon(t)\|_p \vee \|u^\varepsilon(t-)\|_p \ge M \right\}, \qquad \inf\varnothing:=\infty.
$$ The Banach-space It\^o formula with jumps, applied to $F(u^\varepsilon(t\wedge\tau_M^\varepsilon))$, is therefore legitimate.
All estimates obtained below are independent of $M$.
Since an $E_{p}$-valued c\`adl\`ag path is bounded on compact time intervals, $$ \tau_M^\varepsilon\uparrow\infty \qquad\text{a.s.
as }M\uparrow\infty, $$ and the localization will subsequently be removed.

For $v,H\in E$, set $$ \mathcal R_p(v,H) := F(v+H)-F(v)-DF(v)H.
$$ The It\^o formula gives, up to $t\wedge\tau_M^\varepsilon$,
\begin{align}
\|u^\varepsilon(t)\|_p^p =& \|u_0^\varepsilon\|_p^p +I_\Delta^\varepsilon(t) +I_f^\varepsilon(t) +I_h^\varepsilon(t) +I_{g,\mathrm{corr}}^\varepsilon(t)+ M_W^\varepsilon(t) + I_{G,\mathrm{corr}}^\varepsilon(t) + M_{G,1}^\varepsilon(t) + M_{G,2}^\varepsilon(t),
\label{eq:heat-proof-Ito}
\end{align}
where $$ I_\Delta^\varepsilon(t) = p\int_0^t \int_{\R^d} J_p(u^\varepsilon)\cdot\Delta u^\varepsilon \,\dd x\,\dd s,\quad I_f^\varepsilon(t) = p\int_0^t \int_{\R^d} J_p(u^\varepsilon)\cdot\nabla\!\cdot f^\varepsilon \,\dd x\,\dd s,$$$$ I_h^\varepsilon(t)= p\int_0^t \int_{\R^d} J_p(u^\varepsilon)\cdot h^\varepsilon \dd x\dd s.
$$ The Wiener correction is
\begin{align}
I_{g,\mathrm{corr}}^\varepsilon(t) = \frac p2 \int_0^t\int_{\R^d} |u^\varepsilon|^{p-2} \|g^\varepsilon\|_{\gamma(\mathcal U;\R^D)}^2 \,\dd x\,\dd s + \frac{p(p-2)}2 \int_0^t\int_{\R^d} |u^\varepsilon|^{p-4} \|(g^\varepsilon)^*u^\varepsilon\|_{\mathcal U}^2 \,\dd x\,\dd s,
\label{eq:heat-proof-Wcorr}
\end{align}
while
\begin{align}
M_W^\varepsilon(t) &= p\int_0^t \left\langle (g^\varepsilon(s))^* J_p(u^\varepsilon(s)), \,\dd W_s \right\rangle_{\mathcal U},
\label{eq:heat-proof-MW}\\
I_{G,\mathrm{corr}}^\varepsilon(t) &= \int_0^t\int_Z \mathcal R_p \bigl( u^\varepsilon(s-), G_0^\varepsilon(s,z) \bigr) \,\mu(\dd z)\,\dd s,
\label{eq:heat-proof-Gcorr}\\
M_{G,1}^\varepsilon(t) &= p\int_0^t\int_Z \left\langle J_p(u^\varepsilon(s-)), G_0^\varepsilon(s,z) \right\rangle \widetilde N(\dd s,\dd z),
\label{eq:heat-proof-MG1}\\
M_{G,2}^\varepsilon(t) &= \int_0^t\int_Z \mathcal R_p \bigl( u^\varepsilon(s-), G_0^\varepsilon(s,z) \bigr) \widetilde N(\dd s,\dd z).
\label{eq:heat-proof-MG2}
\end{align}
Here and below, pairings without a subscript denote the natural $L^p$--$L^{p'}$ spatial duality.

Introduce the $\R^D$-valued field $$ \mathcal W^\varepsilon := u^\varepsilon|u^\varepsilon|^{\frac{p-2}{2}}.
$$ Then $$ |\mathcal W^\varepsilon| = |u^\varepsilon|^{p/2}, \qquad \|\mathcal W^\varepsilon\|_2^2 = \|u^\varepsilon\|_p^p, $$ and, a.e.
in $x$,
\begin{equation}
|\nabla\mathcal W^\varepsilon|_F^2 = |u^\varepsilon|^{p-2}|\nabla u^\varepsilon|_F^2+ \frac{p^2-4}{4} |u^\varepsilon|^{p-4} \sum_{i=1}^d \bigl( u^\varepsilon\cdot\partial_i u^\varepsilon \bigr)^2.
\label{eq:heat-proof-Widentity}
\end{equation}
In particular,
\begin{equation}
\label{eq:heat-proof-Wlower}
|u^\varepsilon|^{\frac{p-2}{2}} |\nabla u^\varepsilon|_F \le |\nabla\mathcal W^\varepsilon|_F.
\end{equation}

Spatial integration by parts is justified by the regularity of $u^\varepsilon$ and a standard cutoff/density argument on $\R^d$.
Thus
\begin{equation}
-I_\Delta^\varepsilon(t) = p\int_0^t\int_{\R^d} \Bigg[ |u^\varepsilon|^{p-2} |\nabla u^\varepsilon|_F^2 + (p-2)|u^\varepsilon|^{p-4} \sum_{i=1}^d \bigl( u^\varepsilon\cdot\partial_i u^\varepsilon \bigr)^2 \Bigg] \,\dd x\,\dd s.
\label{eq:heat-proof-coercive-exact}
\end{equation}
Since $$ \left| \nabla |u^\varepsilon|^{p/2} \right|^2 = \frac{p^2}{4} |u^\varepsilon|^{p-4} \sum_{i=1}^d \bigl( u^\varepsilon\cdot\partial_i u^\varepsilon \bigr)^2, $$ and comparison of \eqref{eq:heat-proof-Widentity} and \eqref{eq:heat-proof-coercive-exact} yields $$ -I_\Delta^\varepsilon(t) \ge \frac{4(p-1)}p \int_0^t \left\| \nabla |u^\varepsilon(s)|^{p/2} \right\|_2^2\dd s, \quad -I_\Delta^\varepsilon(t) \ge \frac{4p}{p+2} \int_0^t \|\nabla\mathcal W^\varepsilon(s)\|_2^2\dd s.
$$ Splitting $-I_\Delta^\varepsilon$ into two equal parts gives
\begin{equation}
-I_\Delta^\varepsilon(t) \ge \frac{2p}{p+2} \int_0^t \|\nabla\mathcal W^\varepsilon(s)\|_2^2\,\dd s + \frac{2(p-1)}p \int_0^t \left\| \nabla |u^\varepsilon(s)|^{p/2} \right\|_2^2\,\dd s.
\label{eq:heat-proof-coercivity-final}
\end{equation}
For the divergence-form forcing, \eqref{eq:heat-proof-Wlower} gives $$ |\nabla J_p(u^\varepsilon)| \le C_p |\mathcal W^\varepsilon|^{\frac{p-2}{p}} |\nabla\mathcal W^\varepsilon|.
$$ Let $ \rho_f := \frac{2q_f(p-2)}{p(q_f-2)}.
$ The assumptions on $q_f$ imply $ 2\le\rho_f\le\frac{2d}{d-2}, \ d>2, $ whereas $2\le\rho_f<\infty$ when $d=2$.
Hence $$ \|\mathcal W^\varepsilon\|_{\rho_f} \le C \|\mathcal W^\varepsilon\|_2^{1-\alpha_f} \|\nabla\mathcal W^\varepsilon\|_2^{\alpha_f}, \qquad \alpha_f = d\left( \frac12-\frac1{\rho_f} \right).
$$ Consequently, $$ \left| p\int_{\R^d} J_p(u^\varepsilon)\cdot\nabla\cdot f^\varepsilon \dd x \right|\le C_p \|f^\varepsilon\|_{q_f} \|u^\varepsilon\|_p^{ \frac{(1-\alpha_f)(p-2)}2} \|\nabla\mathcal W^\varepsilon\|_2^{ 1+\frac{\alpha_f(p-2)}p}.
$$ Writing $ \gamma_f = \frac12 + \frac{\alpha_f(p-2)}{2p}, $ we have $ \frac1p + \frac{(1-\alpha_f)(p-2)}{2p} + \gamma_f =1.
$ Weighted Young's inequality therefore yields, for every $\delta>0$,
\begin{equation}
\label{eq:heat-proof-f-final}
\left| p\int_{\R^d} J_p(u^\varepsilon)\cdot\nabla\!\cdot f^\varepsilon \,\dd x \right| \le \delta \|\nabla\mathcal W^\varepsilon\|_2^2 + C_\delta \left( \|u^\varepsilon\|_p^p + \|f^\varepsilon\|_{q_f}^p \right).
\end{equation}

For the direct forcing, let $ r_h = \frac{q_h(p-1)}{q_h-1}, \qquad \rho_h = \frac{2r_h}{p} = \frac{2q_h(p-1)}{p(q_h-1)}.
$ The stated lower bound on $q_h$ gives $ 2\le\rho_h\le\frac{2d}{d-2}, \, d>2.
$ For $d=2$ the same argument is valid for $q_h>1$, in which case $2\le\rho_h<\infty$.
Therefore, $$ \|\mathcal W^\varepsilon\|_{\rho_h} \le C \|\mathcal W^\varepsilon\|_2^{1-\alpha_h} \|\nabla\mathcal W^\varepsilon\|_2^{\alpha_h}, \qquad \alpha_h = d\left( \frac12-\frac1{\rho_h} \right).
$$ It follows that $$ \left| p\int_{\R^d} J_p(u^\varepsilon)\cdot h^\varepsilon\dd x \right| \le p\|h^\varepsilon\|_{q_h} \|u^\varepsilon\|_{r_h}^{p-1} \le C \|h^\varepsilon\|_{q_h} \|u^\varepsilon\|_p^{ (p-1)(1-\alpha_h)} \|\nabla\mathcal W^\varepsilon\|_2^{ \frac{2(p-1)\alpha_h}{p}}.$$ Since $ \frac1p + \frac{(p-1)(1-\alpha_h)}p + \frac{(p-1)\alpha_h}{p} =1, $ another weighted Young inequality gives
\begin{equation}
\label{eq:heat-proof-h-final}
\left| p\int_{\R^d} J_p(u^\varepsilon)\cdot h^\varepsilon\,\dd x \right| \le \delta \|\nabla\mathcal W^\varepsilon\|_2^2 + C_\delta \left( \|u^\varepsilon\|_p^p + \|h^\varepsilon\|_{q_h}^p \right).
\end{equation}

For every $(s,x)$, $ \|(g^\varepsilon(s,x))^* u^\varepsilon(s,x)\|_{\mathcal U} \le |u^\varepsilon(s,x)| \|g^\varepsilon(s,x)\|_ {\gamma(\mathcal U;\R^D)}.
$ Hence
\begin{align}
I_{g,\mathrm{corr}}^\varepsilon(t) \le C_p \int_0^t \|u^\varepsilon(s)\|_p^{p-2} \|g^\varepsilon(s)\|_{\mathbb L^p}^{2} \,\dd s\le C_p \int_0^t \left( \|u^\varepsilon(s)\|_p^p + \|g^\varepsilon(s)\|_{\mathbb L^p}^p \right)\dd s.
\label{eq:heat-proof-Wcorr-final}
\end{align}

Moreover, $ \left\| (g^\varepsilon)^* J_p(u^\varepsilon) \right\|_{\mathcal U} \le C_p \|u^\varepsilon\|_p^{p-1} \|g^\varepsilon\|_{\mathbb L^p}.
$ Therefore, by the real-valued BDG inequality and Young's inequality, for every $\delta>0$,
\begin{align}
\E \sup_{0\le r\le t} |M_W^\varepsilon(r)| &\le C_p \E \left[ \sup_{0\le r\le t} \|u^\varepsilon(r)\|_p^{p-1} \left( \int_0^t \|g^\varepsilon(s)\|_{\mathbb L^p}^2\,\dd s \right)^{1/2} \right] \le \delta \E \sup_{0\le r\le t} \|u^\varepsilon(r)\|_p^p + C_{\delta,p,T} \E \int_0^t \|g^\varepsilon(s)\|_{\mathbb L^p}^p\,\dd s.
\label{eq:heat-proof-MW-final}
\end{align}
Convexity of $F$ implies $ \mathcal R_p(v,H)\ge0.
$ Moreover, for $p>2$,
\begin{equation}
\label{eq:heat-proof-Taylor}
0\le \mathcal R_p(v,H) \le C_p \left( \|v\|_p^{p-2}\|H\|_p^2 + \|H\|_p^p \right).
\end{equation}
Introduce $$ A_{2,\varepsilon}(t) := \int_0^t\int_Z \|G_0^\varepsilon(s,z)\|_p^2 \mu(\dd z)\dd s,\quad A_{p,\varepsilon}(t) := \int_0^t\int_Z \|G_0^\varepsilon(s,z)\|_p^p \mu(\dd z)\dd s.
$$ Then \eqref{eq:heat-proof-Taylor} and Young's inequality imply
\begin{align}
\E I_{G,\mathrm{corr}}^\varepsilon(t) &\le C_p \E \left[ \sup_{0\le r\le t} \|u^\varepsilon(r)\|_p^{p-2} A_{2,\varepsilon}(t) + A_{p,\varepsilon}(t) \right]\le \delta \E \sup_{0\le r\le t} \|u^\varepsilon(r)\|_p^p + C_{\delta,p} \E A_{2,\varepsilon}(t)^{p/2} + C_p \E A_{p,\varepsilon}(t).
\label{eq:heat-proof-Gcorr-final}
\end{align}

For the linear jump martingale, Davis' inequality gives
\begin{align*}
\E \sup_{0\le r\le t} |M_{G,1}^\varepsilon(r)| &\le C_p \E \left[ \sup_{0\le r\le t} \|u^\varepsilon(r)\|_p^{p-1} \left( \int_0^t\int_Z \|G_0^\varepsilon(s,z)\|_p^2 \,N(\dd s,\dd z) \right)^{1/2} \right] \\
&\le \delta \E \sup_{0\le r\le t} \|u^\varepsilon(r)\|_p^p+ C_{\delta,p} \E \left( \int_0^t\int_Z \|G_0^\varepsilon(s,z)\|_p^2 \,N(\dd s,\dd z) \right)^{p/2}.
\end{align*}
The standard Poisson moment inequality, applied with $ F(s,z)=\|G_0^\varepsilon(s,z)\|_p^2, \ \frac p2>1, $ yields $$ \E \left( \int_0^t\int_Z \|G_0^\varepsilon(s,z)\|_p^2 N(\dd s,\dd z) \right)^{p/2} \le C_p \E A_{2,\varepsilon}(t)^{p/2} + C_p \E A_{p,\varepsilon}(t).
$$ So,
\begin{align}
\E \sup_{0\le r\le t} |M_{G,1}^\varepsilon(r)| &\le \delta \E \sup_{0\le r\le t} \|u^\varepsilon(r)\|_p^p+ C_{\delta,p} \E A_{2,\varepsilon}(t)^{p/2} + C_{\delta,p} \E A_{p,\varepsilon}(t).
\label{eq:heat-proof-MG1-final}
\end{align}

For the nonlinear Taylor-remainder martingale, positivity of $\mathcal R_p$ permits one to avoid any higher moment of $G_0^\varepsilon$.
Indeed, pathwise,
\begin{align}
\sup_{0\le r\le t} |M_{G,2}^\varepsilon(r)| &\le \int_0^t\int_Z \mathcal R_p \bigl( u^\varepsilon(s-), G_0^\varepsilon(s,z) \bigr) \,N(\dd s,\dd z)+ \int_0^t\int_Z \mathcal R_p \bigl( u^\varepsilon(s-), G_0^\varepsilon(s,z) \bigr) \,\mu(\dd z)\,\dd s.
\label{eq:heat-proof-MG2-path}
\end{align}
The compensator identity and \eqref{eq:heat-proof-Gcorr-final} therefore imply
\begin{align}
\E \sup_{0\le r\le t} |M_{G,2}^\varepsilon(r)| &\le 2\, \E I_{G,\mathrm{corr}}^\varepsilon(t) \le \delta \E \sup_{0\le r\le t} \|u^\varepsilon(r)\|_p^p + C_{\delta,p} \E A_{2,\varepsilon}(t)^{p/2} + C_{\delta,p} \E A_{p,\varepsilon}(t).
\label{eq:heat-proof-MG2-final}
\end{align}
In particular, no $2p$-moment assumption on $G_0$ is introduced.

Insert \eqref{eq:heat-proof-coercivity-final}, \eqref{eq:heat-proof-f-final}, \eqref{eq:heat-proof-h-final}, \eqref{eq:heat-proof-Wcorr-final}, \eqref{eq:heat-proof-MW-final}, \eqref{eq:heat-proof-Gcorr-final}, \eqref{eq:heat-proof-MG1-final}, and \eqref{eq:heat-proof-MG2-final} into \eqref{eq:heat-proof-Ito}.
Choose $\delta>0$ sufficiently small so that the contributions containing $\|\nabla\mathcal W^\varepsilon\|_2^2$ and $\sup\|u^\varepsilon\|_p^p$ are absorbed into the left-hand side.
After removing the localization $\tau_M^\varepsilon$ by letting $M\uparrow\infty$, we obtain for $0\le t\le T$
\begin{align*}
& \E \sup_{0\le r\le t} \|u^\varepsilon(r)\|_p^p + c_p \E \int_0^t \left\| \nabla|u^\varepsilon(s)|^{p/2} \right\|_2^2\,\dd s\le C_p\E\|u_0^\varepsilon\|_p^p + C \int_0^t \E \sup_{0\le r\le s} \|u^\varepsilon(r)\|_p^p \,\dd s
\nonumber\\
&\quad+ C \E \int_0^t \left( \|f^\varepsilon(s)\|_{q_f}^p + \|h^\varepsilon(s)\|_{q_h}^p + \|g^\varepsilon(s)\|_{\mathbb L^p}^p \right)\dd s + C \E A_{p,\varepsilon}(t) + C \E A_{2,\varepsilon}(t)^{p/2}.
\end{align*}
Gronwall's inequality, together with the contraction properties of spatial mollification, yields
\begin{align}
& \E \left[ \sup_{0\le t\le T} \|u^\varepsilon(t)\|_p^p + \int_0^T \left\| \nabla|u^\varepsilon(s)|^{p/2} \right\|_2^2\,\dd s \right]\le C \E \Bigg[ \|u_0\|_p^p + \int_0^T \left( \|f(s)\|_{q_f}^p + \|h(s)\|_{q_h}^p + \|g(s)\|_{\mathbb L^p}^p \right)\dd s
\nonumber\\
&\hspace{19mm} + \int_0^T\int_Z \|G_0(s,z)\|_p^p \,\mu(\dd z)\,\dd s + \left( \int_0^T\int_Z \|G_0(s,z)\|_p^2 \,\mu(\dd z)\,\dd s \right)^{p/2} \Bigg],
\label{eq:heat-proof-uniform}
\end{align}
where $ C=C(T,d,D,p,q_f,q_h) $ is independent of $\varepsilon$.

Let $\varepsilon,\delta>0$ and set $ v^{\varepsilon,\delta} = u^\varepsilon-u^\delta.
$ Since \eqref{eq:heat-proof-reg} is linear and additive, $v^{\varepsilon,\delta}$ satisfies the same equation with data replaced by the corresponding differences.
Repeating the preceding estimate gives
\begin{align*}
\E \sup_{0\le t\le T} \|u^\varepsilon(t)-u^\delta(t)\|_p^p
\nonumber
&\le C \E \Bigg[ \|u_0^\varepsilon-u_0^\delta\|_p^p + \int_0^T \|f^\varepsilon-f^\delta\|_{q_f}^p\,\dd s + \int_0^T \|h^\varepsilon-h^\delta\|_{q_h}^p\,\dd s + \int_0^T \|g^\varepsilon-g^\delta\|_{\mathbb L^p}^p\,\dd s\nonumber \\
&\qquad + \int_0^T\int_Z \|G_0^\varepsilon-G_0^\delta\|_p^p \,\mu(\dd z)\,\dd s + \left( \int_0^T\int_Z \|G_0^\varepsilon-G_0^\delta\|_p^2 \,\mu(\dd z)\,\dd s \right)^{p/2} \Bigg].
\end{align*}
By \eqref{eq:heat-proof-moll-u0}--\eqref{eq:heat-proof-moll-G2}, the right-hand side tends to zero as $\varepsilon,\delta\downarrow0$.
Thus
\begin{equation}
\label{eq:heat-proof-Spconv}
\E \sup_{0\le t\le T} \|u^\varepsilon(t)-u(t)\|_p^p \longrightarrow0
\end{equation}
for some adapted $E_{p}$-valued process $u$.

To identify its path regularity, choose a sequence $\varepsilon_n\downarrow0$ such that $$ \sum_{n=1}^\infty \left( \E \sup_{0\le t\le T} \|u^{\varepsilon_{n+1}}(t) -u^{\varepsilon_n}(t)\|_p^p \right)^{1/p} <\infty.
$$ It follows that, outside a null set, $(u^{\varepsilon_n})_n$ is uniformly Cauchy on $[0,T]$ in $E_{p}$.
Since the uniform limit of $E_{p}$-valued c\`adl\`ag functions is c\`adl\`ag, $u$ has an $E_{p}$-valued c\`adl\`ag modification.
Moreover, \eqref{eq:heat-proof-Spconv} implies $$ \E \sup_{0\le t\le T} \|u(t)\|_p^p<\infty.
$$ For each $t$, $u^\varepsilon(t)\to u(t)$ in $L^p(\Omega;E_{p})$; hence $u(t)$ is $\mathcal F_t$-measurable.
Thus the c\`adl\`ag version may be chosen adapted.

Let $ \varphi\in C_c^\infty(\R^d;\R^D).
$ For every $\varepsilon>0$,
\begin{align}
\langle u^\varepsilon(t),\varphi\rangle =& \langle u_0^\varepsilon,\varphi\rangle + \int_0^t \Big[ \langle u^\varepsilon(s),\Delta\varphi\rangle - \langle f^\varepsilon(s),\nabla\varphi\rangle + \langle h^\varepsilon(s),\varphi\rangle \Big]\dd s
\nonumber\\
&+ \int_0^t \left\langle (g^\varepsilon(s))^*\varphi,\dd W_s \right\rangle_{\mathcal U} + \int_0^t\int_Z \langle G_0^\varepsilon(s,z),\varphi\rangle \,\widetilde N(\dd s,\dd z).
\label{eq:heat-proof-weak-eps}
\end{align}
The deterministic terms converge by \eqref{eq:heat-proof-Spconv} and \eqref{eq:heat-proof-moll-f}--\eqref{eq:heat-proof-moll-h}.

For the Wiener term, BDG and H\"older's inequality give $$ \E \sup_{0\le t\le T} \left| \int_0^t \left\langle (g^\varepsilon-g)^*(s)\varphi,\dd W_s \right\rangle_{\mathcal U} \right|^p \le C_{p,T,\varphi} \E \int_0^T \|g^\varepsilon(s)-g(s)\|_{\mathbb L^p}^p\dd s \longrightarrow0.
$$ Similarly, the BDG inequality for compensated Poisson integrals gives
\begin{align*}
&\E \sup_{0\le t\le T} \left| \int_0^t\int_Z \langle G_0^\varepsilon(s,z)-G_0(s,z), \varphi \rangle \,\widetilde N(\dd s,\dd z) \right|^p \\
& \qquad\le C_\varphi \E \int_0^T\int_Z \|G_0^\varepsilon-G_0\|_p^p \,\mu(\dd z)\,\dd s+ C_\varphi \E \left( \int_0^T\int_Z \|G_0^\varepsilon-G_0\|_p^2 \,\mu(\dd z)\,\dd s \right)^{p/2} \longrightarrow0.
\end{align*}
Passing to the limit in \eqref{eq:heat-proof-weak-eps} therefore yields
\begin{align}
\langle u(t),\varphi\rangle =& \langle u_0,\varphi\rangle + \int_0^t \Big[ \langle u(s),\Delta\varphi\rangle - \langle f(s),\nabla\varphi\rangle + \langle h(s),\varphi\rangle \Big]\dd s
\nonumber\\
&+ \int_0^t \left\langle g(s)^*\varphi,\dd W_s \right\rangle_{\mathcal U} + \int_0^t\int_Z \langle G_0(s,z),\varphi\rangle \,\widetilde N(\dd s,\dd z)
\label{eq:heat-proof-weak-limit}
\end{align}
indistinguishably in $t\in[0,T]$.
Thus $u$ is a probabilistically strong solution of \eqref{eq:heat} in the stated spatial distributional sense.

By \eqref{eq:heat-proof-Spconv} and the c\`adl\`ag path-space convention, $ u^\varepsilon \longrightarrow u \ \text{in } L^p\bigl(\Omega; \D([0,T]; E_{p})\bigr).
$ Moreover, \eqref{eq:heat-proof-uniform} yields $$ \sup_{\varepsilon>0} \E \int_0^T \left\| \nabla |u^\varepsilon(s)|^{p/2} \right\|_2^2 \mathrm{d}s < \infty.
$$ Hence, applying \cref{lem:L2_con} to any sequence $\varepsilon_n \downarrow 0$, we obtain $ |u|^{p/2} \in L^2\left( \Omega\times(0,T); H^1(\R^d) \right), $ and
\begin{equation*}
\E \int_0^T \left\| \nabla |u(s)|^{p/2} \right\|_2^2 \,\mathrm{d}s \le \liminf_{n\to\infty} \E \int_0^T \left\| \nabla |u^{\varepsilon_n}(s)|^{p/2} \right\|_2^2 \,\mathrm{d}s.
\end{equation*}
The supremum term passes to the limit directly.
Indeed, \eqref{eq:heat-proof-Spconv} and the reverse triangle inequality in $L^p(\Omega)$ give
\begin{equation*}
\begin{aligned}
& \left| \left( \E \sup_{t\in[0,T]} \|u^\varepsilon(t)\|_p^p \right)^{1/p} - \left( \E \sup_{t\in[0,T]} \|u(t)\|_p^p \right)^{1/p} \right| \le \left( \E \sup_{t\in[0,T]} \|u^\varepsilon(t) - u(t)\|_p^p \right)^{1/p} \longrightarrow 0.
\end{aligned}
\end{equation*}
Therefore,
\begin{equation}
\label{eq:heat-proof-energy-liminf}
\E \Bigg[ \sup_{t\in[0,T]} \|u(t)\|_p^p + \int_0^T \left\| \nabla |u(s)|^{p/2} \right\|_2^2 \,\mathrm{d}s \Bigg] \le \liminf_{n\to\infty} \E \Bigg[ \sup_{t\in[0,T]} \|u^{\varepsilon_n}(t)\|_p^p + \int_0^T \left\| \nabla |u^{\varepsilon_n}(s)|^{p/2} \right\|_2^2 \,\mathrm{d}s \Bigg].
\end{equation}
Finally, \eqref{eq:heat-proof-moll-u0}--\eqref{eq:heat-proof-moll-G2} imply convergence of all data terms appearing on the right-hand side of \eqref{eq:heat-proof-uniform}.
Therefore, letting $n \to \infty$ in \eqref{eq:heat-proof-uniform} along $\varepsilon_n \downarrow 0$ and using \eqref{eq:heat-proof-energy-liminf}, we obtain \eqref{ineq:3_eng-est}.

Let $u_1$ and $u_2$ be two solutions on the same stochastic basis with the same initial datum and the same coefficients and driving noises.
Their difference $ v:=u_1-u_2 $ satisfies
\begin{equation}
\label{eq:heat-proof-difference}
\partial_t v=\Delta v, \qquad v(0)=0,
\end{equation}
in $\mathcal D'((0,T)\times\R^d)$, and $ v\in L^\infty(0,T;L^p(\R^d)) \ \text{a.s.} $ Fix $t\in[0,T]$ and $\varphi\in C_c^\infty(\R^d;\R^D)$, and let $ \Phi(s)=S(t-s)\varphi, \ 0\le s\le t.
$ Using $\Phi$ as a time-dependent test function in \eqref{eq:heat-proof-difference}, justified by the standard approximation in the weak formulation, gives $$ \langle v(t),\varphi\rangle = \langle v(0),S(t)\varphi\rangle =0.
$$ Since $\varphi$ is arbitrary, $$ v(t)=0 \qquad\text{in }L^p(\R^d) $$ for every $t\in[0,T]$.
Thus $u_1$ and $u_2$ are indistinguishable, and uniqueness follows.
\end{proof}
\section{Stochastic truncated tamed Navier--Stokes equations}
\label{sec4}

In this section we establish global well-posedness for a spatially regularized and $L^p$-truncated version of the STNSE (\ref{eq:1_Main}).
This auxiliary problem is the starting point for the construction of local solutions to the untruncated equation in the subsequent section.

The approximation combines the convolution Fourier multiplier ${\mathrm P}_{\le k}$ introduced in \cref{sec2} with a scalar cutoff in the $L^p$-norm.
The multiplier regularizes all nonlinear coefficients at the fixed frequency level $k$, while the scalar cutoff restricts the nonlinear and the stochastic coefficients to a bounded subset of $L^p(\R^3;\R^3)$.
This is particularly useful in the Picard scheme as after one freezes the preceding iterate, the tamed drift no longer contributes through its coercive sign and is instead treated as a direct forcing term in the linear theory of \cref{sec3}.

The stochastic forcing contains both a cylindrical Wiener component and a compensated Poisson component.
We retain throughout the abstract jump framework of \cref{sec2}, $N$ is an $(\mathcal F_t)$-Poisson random measure on the $\sigma$-finite measurable space $(Z,\mathcal Z,\mu)$ with compensator $\dd t\,\mu(\dd z)$, and $ \widetilde N(\dd t,\dd z) = N(\dd t,\dd z)-\dd t\,\mu(\dd z).
$ The integrability required for the compensated Poisson integral is directly from Hypotheses \eqref{eq:4.G1}--\eqref{eq:4.G2}, which control the two $L^2(Z,\mu)$ and $L^p(Z,\mu)$ modes entering the compensated Poisson maximal inequality.

Fix $ R\ge1, \ k\in\N,\ N>0, $ and choose $ \phi\in C^\infty([0,\infty);[0,1]) $ such that
\begin{equation}
\label{eq:4.cutoff}
\phi(r)=1 \quad\text{for }0\le r\le2R, \qquad \phi(r)=0 \quad\text{for }r\ge4R,
\end{equation}
and
\begin{equation}
\label{eq:4.cutoff-Lipschitz}
|\phi(r_1)-\phi(r_2)| \le C_\phi|r_1-r_2|, \qquad r_1,r_2\ge0.
\end{equation}
In particular, for every $a>0$,
\begin{equation}
\label{eq:4.cutoff-uniform}
\sup_{\rho\ge0} \phi(\rho)\rho^a \le (4R)^a.
\end{equation}
This elementary consequence of the support of $\phi$ will be used repeatedly to obtain bounds independent of the Picard iteration indices.

Set $ E:=L^p(\R^3;\R^3).
$ For an $E$-valued c\`adl\`ag process $u$, write $ \phi_u(t) := \phi(\|u(t)\|_p), \ \phi_u(t-) := \phi(\|u(t-)\|_p).
$ We consider the truncated equation
\begin{equation}
\label{eq:4_truncation}
\begin{aligned}
\dd u(t) =& \Big[ \nu\Delta u(t) - \phi_u(t)^2 {\mathrm P}_{\le k}\mathcal P \Big( (u(t)\cdot\nabla) {\mathrm P}_{\le k}u(t) \Big) - \phi_u(t)^2 {\mathrm P}_{\le k}\mathcal P \Big( g_N(|{\mathrm P}_{\le k}u(t)|^2) {\mathrm P}_{\le k}u(t) \Big) \Big]\,\dd t \\
& \qquad \qquad + \phi_u(t)^2 {\mathrm P}_{\le k}\mathcal P \sigma({\mathrm P}_{\le k}u(t))\,\dd W_t + \int_Z \phi_u(t-)^2 {\mathrm P}_{\le k}\mathcal P G({\mathrm P}_{\le k}u(t-),z) \,\widetilde N(\dd t,\dd z), \\
\nabla\cdot u(t) =&0, \\
u(0) =& {\mathrm P}_{\le k}u_0, \qquad \P\text{-a.s.}
\end{aligned}
\end{equation}

The use of left limits in the Poisson coefficient is essential.
For an adapted c\`adl\`ag process, $u(\cdot-)$ is predictable and therefore provides the appropriate integrand for the compensated Poisson integral.
For the $\dd t$- and Wiener terms we retain the notation $u(t)$; whenever a predictable representative is required for the Wiener coefficient, it may be replaced by its left-limit version, since $ u(t)=u(t-) \ \dd t\otimes\dd\P\text{-a.e.} $ and hence the resulting Wiener integral is unchanged.

Since ${\mathrm P}_{\le k}$ is a scalar Fourier multiplier, it commutes with spatial differentiation and with the Leray projector.
For a tensor field $F=(F_{ij})$, define $\mathcal P^{(1)}F$ by $ \widehat{(\mathcal P^{(1)}F)}_{ij}(\xi) := \widehat{\mathcal P}_{i\ell}(\xi) \widehat F_{\ell j}(\xi).
$ If $\nabla\cdot u=0$, then
\begin{equation}
\label{eq:4.commutation}
{\mathrm P}_{\le k}\mathcal P \Big( (u\cdot\nabla){\mathrm P}_{\le k}u \Big)= {\mathrm P}_{\le k}\mathcal P \Big( \nabla\cdot (u\otimes{\mathrm P}_{\le k}u) \Big) = \nabla\cdot \Big( \mathcal P^{(1)}{\mathrm P}_{\le k} (u\otimes{\mathrm P}_{\le k}u) \Big)
\end{equation}
in $\mathcal D'(\R^3)$.
Thus the convective term enters \cref{Thm:heat-exis} through the divergence-form forcing, whereas the taming term is treated there as a direct forcing.

The truncated problem \eqref{eq:4_truncation} is the analogue, in the present tamed Wiener--Poisson setting, of the truncated construction in Section~$4$ of \cite{MR4908978}.
The jump component requires two additional features: the solution is sought in a c\`adl\`ag path space, and both Poisson integrability modes appearing in \cref{Thm:heat-exis} must be retained throughout the fixed-point and convergence arguments.

\begin{theorem}
[Global well-posedness of the truncated equation] \label{thm:main_existence} Let $p>2$, and assume \eqref{eq:2_WienerH} and \eqref{eq:4.G1}--\eqref{eq:4.G2}.
Let $ u_0 \in L^p(\Omega,\mathcal F_0;E), \ \nabla\cdot u_0=0 \quad \P\text{-a.s.} $ Then, for every $T>0$, equation \eqref{eq:4_truncation} admits a unique strong solution $ u \in L^p\bigl( \Omega; \D([0,T];E) \bigr) $ in the sense of the c\`adl\`ag path-space convention \eqref{eq:cadlag-Lp-convention}.
Moreover, $ |u|^{p/2} \in L^2\!\left( \Omega\times(0,T); H^1(\R^3) \right), $ and
\begin{equation}
\label{eq:4_main_energy_bound}
\begin{aligned}
\E \Bigg[ & \sup_{0\le s\le T} \|u(s)\|_p^p + \int_0^T \left\| \nabla \bigl( |u(s)|^{p/2} \bigr) \right\|_2^2 \,\dd s \Bigg] \le C \left( 1+\E\|u_0\|_p^p \right),
\end{aligned}
\end{equation}
where $C$ depends only on $p,k,\nu,N,R,T$, $C_\phi$, and the structural constants in \eqref{eq:2_WienerH} and \eqref{eq:4.G1}--\eqref{eq:4.G2}.
\end{theorem}

Throughout this section, the dependence of constants on the fixed truncation level $R$ will usually be suppressed.

We construct the solution by the telescoping Picard scheme of Section~$4$ of \cite{MR4908978}, modified to include the direct taming forcing and the compensated Poisson term.
Define
\begin{equation}
\label{eq:iteration_0}
\begin{aligned}
\dd u^{(0)}(t) &= \nu\Delta u^{(0)}(t)\,\dd t, \\
\nabla\cdot u^{(0)}(t) &=0, \\
u^{(0)}(0) &= {\mathrm P}_{\le k}u_0, \qquad \P\text{-a.s.}
\end{aligned}
\end{equation}
Equivalently, $ u^{(0)}(t) = e^{\nu t\Delta}{\mathrm P}_{\le k}u_0.
$ In particular, $ u^{(0)} \in L^p\bigl( \Omega; \D([0,T];E) \bigr).
$

For $n\ge1$, set $ \phi^{(n)}(t) := \phi(\|u^{(n)}(t)\|_p), \ \phi^{(n)}(t-) := \phi(\|u^{(n)}(t-)\|_p), $ and consider
\begin{equation}
\label{eq:4_iteration_n}
\begin{aligned}
\dd u^{(n)}(t) =& \Big[ \nu\Delta u^{(n)}(t) - \phi^{(n)}(t)\phi^{(n-1)}(t) {\mathrm P}_{\le k}\mathcal P \Big( (u^{(n-1)}(t)\cdot\nabla) {\mathrm P}_{\le k}u^{(n-1)}(t) \Big) \\
&\qquad - \phi^{(n)}(t)\phi^{(n-1)}(t) {\mathrm P}_{\le k}\mathcal P \Big( g_N( |{\mathrm P}_{\le k}u^{(n-1)}(t)|^2 ) {\mathrm P}_{\le k}u^{(n-1)}(t) \Big) \Big]\,\dd t \\
&+ \phi^{(n)}(t)\phi^{(n-1)}(t) {\mathrm P}_{\le k}\mathcal P \sigma( {\mathrm P}_{\le k}u^{(n-1)}(t) )\,\dd W_t \\
&+ \int_Z \phi^{(n)}(t-)\phi^{(n-1)}(t-) {\mathrm P}_{\le k}\mathcal P G( {\mathrm P}_{\le k}u^{(n-1)}(t-),z ) \,\widetilde N(\dd t,\dd z), \\
\nabla\cdot u^{(n)}(t) =&0, \\
u^{(n)}(0) =& {\mathrm P}_{\le k}u_0, \qquad \P\text{-a.s.}
\end{aligned}
\end{equation}
The dependence of \eqref{eq:4_iteration_n} on $u^{(n)}$ occurs only through the scalar factor $\phi^{(n)}$.
Thus, once $u^{(n-1)}$ is fixed, the $n$-th equation is obtained by an inner fixed-point argument based on \cref{Thm:heat-exis}.
The preceding cutoff estimate \eqref{eq:4.cutoff-uniform} is what makes the resulting coefficient bounds uniform in the outer index.

Since $u^{(n-1)}$ is divergence-free, $ (u^{(n-1)}\cdot\nabla) {\mathrm P}_{\le k}u^{(n-1)} = \nabla\cdot \bigl( u^{(n-1)} \otimes {\mathrm P}_{\le k}u^{(n-1)} \bigr), $ and hence the convection has precisely the divergence structure required by \cref{Thm:heat-exis}.
The taming contribution is placed in its direct-forcing component, while the jump coefficient is predictable by construction and is estimated simultaneously in the $L^2(Z,\mu;E)$ and $L^p(Z,\mu;E)$ norms.

The following lemma establishes the global solvability and the uniform estimate for every outer iterate.

\begin{lemma}
\label{lem:4_trunexist} Let $p>2$, and let $ u_0 \in L^p(\Omega,\mathcal F_0;E), \ \nabla\cdot u_0=0 \ \P\text{-a.s.} $ Then, for every $n\ge0$ and every $T>0$, problem \eqref{eq:4_iteration_n}, with \eqref{eq:iteration_0} for $n=0$, admits a unique strong solution $ u^{(n)} \in L^p\bigl( \Omega; \D([0,T];E) \bigr).
$ The estimate \eqref{eq:4_main_energy_bound} holds with $u$ replaced by $u^{(n)}$ and with a constant independent of $n$.
\end{lemma}
\begin{proof}
The argument follows the fixed-point construction of \cite[Lemma~4.2]{MR4908978}.
We use throughout its immediate $\nu\Delta$-variant.

We argue inductively in $n$.
The process $u^{(0)}$ is given by \eqref{eq:iteration_0}.
Fix $n\ge1$ and set $ v:=u^{(n-1)}, \ \phi_v(t):=\phi(\|v(t)\|_p).
$ By the induction hypothesis, $ v\in L^p\bigl(\Omega;\D([0,T];E)\bigr), \ \nabla\cdot v=0.
$ For later use, define $$
\begin{aligned}
F_v(t) &:= -\phi_v(t)\, \mathcal P^{(1)}{\mathrm P}_{\le k} \bigl( v(t)\otimes{\mathrm P}_{\le k}v(t) \bigr), \\
H_v(t) &:= -\phi_v(t)\, {\mathrm P}_{\le k}\mathcal P \bigl( g_N(|{\mathrm P}_{\le k}v(t)|^2) {\mathrm P}_{\le k}v(t) \bigr), \\
\Sigma_v(t) &:= \phi_v(t)\, {\mathrm P}_{\le k}\mathcal P \sigma({\mathrm P}_{\le k}v(t)), \\
\Gamma_v(t,z) &:= \phi_v(t-)\, {\mathrm P}_{\le k}\mathcal P G({\mathrm P}_{\le k}v(t-),z).
\end{aligned}
$$ Here $ \widehat{(\mathcal P^{(1)}F)}_{ij}(\xi) = \widehat{\mathcal P}_{i\ell}(\xi)\widehat F_{\ell j}(\xi).
$ Since ${\mathrm P}_{\le k}$, $\mathcal P$, and the spatial derivatives commute, $ {\mathrm P}_{\le k}\mathcal P \bigl( \nabla\cdot(v\otimes{\mathrm P}_{\le k}v) \bigr) = \nabla\cdot \Bigl( \mathcal P^{(1)}{\mathrm P}_{\le k} (v\otimes{\mathrm P}_{\le k}v) \Bigr) $ in $\mathcal D'(\R^3)$.

Choose $ q_f \in \left[ \max\left\{ \frac p2,\frac{3p}{p+1} \right\}, p \right), \ \frac1{q_f} = \frac1p+\frac1{\ell_f}, $ and $ q_h \in \left[ \max\left\{ \frac p3,\frac{3p}{2p+1} \right\}, p \right].
$ Then $\ell_f\ge p$ and $3q_h\ge p$.
By \cref{lem:2.2}, the growth of $g_N$, \eqref{eq:2_WienerH}, \eqref{eq:4.G1}, and \eqref{eq:4.cutoff-uniform},
\begin{equation}
\label{eq:4.frozen-bounds}
\|F_v(t)\|_{q_f} \le C_{p,q_f,k,R},\quad \|H_v(t)\|_{q_h} \le C_{p,q_h,k,R,N,\nu},\quad \|\Sigma_v(t)\|_{\mathbb L^p} \le C_{p,k,R},
\end{equation}
and, for $r\in\{2,p\}$,
\begin{equation}
\label{eq:4.frozen-jump-bounds}
\int_Z \|\Gamma_v(t,z)\|_p^r\,\mu(\dd z) \le C_{p,k,R}.
\end{equation}
The same estimates hold with $t$ replaced by $t-$.

The Poisson coefficient in \eqref{eq:4.frozen-jump-bounds} is predictable.
For the Wiener integral we use, whenever necessary, the predictable representative obtained by replacing the c\`adl\`ag factors by their left limits; this does not change the integral since a c\`adl\`ag process agrees with its left-limit process $\dd t\otimes\dd\P$-a.e.

Let $ U^{(0)}(t) := e^{\nu t\Delta}{\mathrm P}_{\le k}u_0 $ and, for $m\ge1$, put $ \Phi^{(m)}(t):=\phi(\|U^{(m)}(t)\|_p), $ where $U^{(m)}$ is the solution of
\begin{equation*}
\begin{aligned}
\dd U^{(m)}(t) =& \Big[ \nu\Delta U^{(m)}(t) + \nabla\cdot \bigl( \Phi^{(m-1)}(t)F_v(t) \bigr) + \Phi^{(m-1)}(t)H_v(t) \Big]\dd t \\
&+ \Phi^{(m-1)}(t)\Sigma_v(t)\,\dd W_t + \int_Z \Phi^{(m-1)}(t-)\Gamma_v(t,z) \,\widetilde N(\dd t,\dd z), \\
U^{(m)}(0) =& {\mathrm P}_{\le k}u_0.
\end{aligned}
\end{equation*}
All forcing terms are solenoidal, and hence $\nabla\cdot U^{(m)}=0$.

Since $0\le\Phi^{(m-1)}\le1$, \eqref{eq:4.frozen-bounds} gives
\begin{equation*}
\E\int_0^T \Big( \|\Phi^{(m-1)}F_v\|_{q_f}^p + \|\Phi^{(m-1)}H_v\|_{q_h}^p + \|\Phi^{(m-1)}\Sigma_v\|_{\mathbb L^p}^p \Big)\,\dd t \le C_{p,k,R,N,\nu,T}.
\end{equation*}
, The two Poisson norms required by \cref{Thm:heat-exis} satisfy
\begin{equation*}
\begin{aligned}
\E \int_0^T\int_Z \| \Phi^{(m-1)}(t-)\Gamma_v(t,z) \|_p^p \,\mu(\dd z)\,\dd t &\le C_{p,k,R}T, \\
\E \left( \int_0^T\int_Z \| \Phi^{(m-1)}(t-)\Gamma_v(t,z) \|_p^2 \,\mu(\dd z)\,\dd t \right)^{p/2} &\le C_{p,k,R}T^{p/2}.
\end{aligned}
\end{equation*}
Thus \cref{Thm:heat-exis} yields a unique $ U^{(m)} \in L^p\bigl(\Omega;\D([0,T];E)\bigr) $ for every $m\ge1$.

We next prove contraction on a deterministic interval.
Set $ z^{(m)} := U^{(m+1)}-U^{(m)}, \ \delta_m := \Phi^{(m)}-\Phi^{(m-1)}.
$ Then $z^{(m)}(0)=0$ and
\begin{equation}
\label{eq:4.diff-delta}
\dd z^{(m)} =\Big[ \nu\Delta z^{(m)} + \nabla\cdot(\delta_m F_v) + \delta_m H_v \Big]\dd t + \delta_m\Sigma_v\,\dd W_t + \int_Z \delta_m(t-)\Gamma_v(t,z) \,\widetilde N(\dd t,\dd z).
\end{equation}
The global Lipschitz continuity of $\phi$ gives $|\delta_m(t)| \le C_\phi \|z^{(m-1)}(t)\|_p, \ |\delta_m(t-)| \le C_\phi \|z^{(m-1)}(t-)\|_p.$

Therefore, for $0<t\le T$,
\begin{equation}
\begin{aligned}
\E\int_0^t \Big( \|\delta_mF_v\|_{q_f}^p + \|\delta_mH_v\|_{q_h}^p + \|\delta_m\Sigma_v\|_{\mathbb L^p}^p \Big)\,\dd s + &\E\int_0^t\int_Z \| \delta_m(s-)\Gamma_v(s,z) \|_p^p \,\mu(\dd z)\,\dd s \\
&\qquad\le C_{p,q_f,q_h,k,R,N,\nu}\, t\, \E \sup_{0\le s\le t} \|z^{(m-1)}(s)\|_p^p,
\label{eq:4.diff-p}
\end{aligned}
\end{equation}
whereas the quadratic Poisson norm satisfies
\begin{equation}
\label{eq:4.diff-2}
\E \left( \int_0^t\int_Z \| \delta_m(s-)\Gamma_v(s,z) \|_p^2\mu(\dd z)\dd s \right)^\frac p2 \le C_{p,k,R} \E \left( \int_0^t \|z^{(m-1)}(s-)\|_p^2\,\dd s \right)^{p/2} \le C_{p,k,R}\, t^\frac p2\E \sup_{0\le s\le t} \|z^{(m-1)}(s)\|_p^p.
\end{equation}

Here we used $ z^{(m-1)}(s-)=z^{(m-1)}(s) \ \text{for }\dd s\text{-a.e.
}s.
$

Applying \cref{Thm:heat-exis} to \eqref{eq:4.diff-delta} and using \eqref{eq:4.diff-p}--\eqref{eq:4.diff-2}, we obtain
\begin{equation*}
\E \sup_{0\le s\le t} \|z^{(m)}(s)\|_p^p \le C \bigl( t+t^{p/2} \bigr) \E \sup_{0\le s\le t} \|z^{(m-1)}(s)\|_p^p,
\end{equation*}
where $ C=C_{p,q_f,q_h,k,R,N,\nu} $ is independent of $m$, $n$, and the initial datum.
Choose $t^*\in(0,T]$ such that $ C \bigl( t^*+(t^*)^{p/2} \bigr) \le \frac12.
$ It follows that $$ \E \sup_{0\le s\le t^*} \|U^{(m+1)}(s)-U^{(m)}(s)\|_p^p \le 2^{-m} \E \sup_{0\le s\le t^*} \|U^{(1)}(s)-U^{(0)}(s)\|_p^p.
$$ Consequently, $$ \sum_{m=0}^\infty \left( \E \sup_{0\le s\le t^*} \|U^{(m+1)}(s)-U^{(m)}(s)\|_p^p \right)^{1/p} <\infty.
$$ Hence, in the complete process space associated with the convention \eqref{eq:cadlag-Lp-convention}, there exists $ u \in L^p\bigl(\Omega;\D([0,t^*];E)\bigr) $ such that
\begin{equation}
\label{eq:4.strong-picard-limit}
\E \sup_{0\le s\le t^*} \|U^{(m)}(s)-u(s)\|_p^p \longrightarrow0.
\end{equation}

Put $ \Phi_u(t):=\phi(\|u(t)\|_p).
$ Since uniform convergence of c\`adl\`ag paths controls their left limits, $$ \sup_{0\le s\le t^*} \|U^{(m)}(s-)-u(s-)\|_p \le \sup_{0\le s\le t^*} \|U^{(m)}(s)-u(s)\|_p, $$ and therefore
\begin{equation}
\label{eq:4.phi-limit}
\begin{aligned}
\E \sup_{0\le s\le t^*} |\Phi^{(m)}(s)-\Phi_u(s)|^p &\longrightarrow0, \\
\E \sup_{0\le s\le t^*} |\Phi^{(m)}(s-)-\Phi_u(s-)|^p &\longrightarrow0.
\end{aligned}
\end{equation}

In particular, \eqref{eq:4.frozen-bounds} and \eqref{eq:4.phi-limit} imply
\begin{equation}
\label{eq:4.coeff-limit}
\begin{aligned}
\E\int_0^{t^*} \Big( \| (\Phi^{(m-1)}-\Phi_u)F_v \|_{q_f}^p + \| (\Phi^{(m-1)}-\Phi_u)H_v \|_{q_h}^p + \| (\Phi^{(m-1)}-\Phi_u)\Sigma_v \|_{\mathbb L^p}^p \Big)\,\dd s &\longrightarrow0, \\
\E \int_0^{t^*}\int_Z \| (\Phi^{(m-1)}(s-)-\Phi_u(s-)) \Gamma_v(s,z) \|_p^p \,\mu(\dd z)\,\dd s &\longrightarrow0, \\
\E \left( \int_0^{t^*}\int_Z \| (\Phi^{(m-1)}(s-)-\Phi_u(s-)) \Gamma_v(s,z) \|_p^2 \,\mu(\dd z)\,\dd s \right)^{p/2} &\longrightarrow0.
\end{aligned}
\end{equation}
The last two limits are precisely the two compensated-Poisson integrability modes entering the Bichteler--Jacod estimate.
The corresponding compensated Poisson integrals converge in $ L^p\bigl( \Omega; L^\infty(0,t^*;E) \bigr), $ and the Wiener integrals converge by \eqref{eq:2-W-BDG}.

Let $\bar u$ denote the solution obtained by \cref{Thm:heat-exis} with coefficients $ \Phi_uF_v, \ \Phi_uH_v, \ \Phi_u\Sigma_v, \ \Phi_u(\cdot-)\Gamma_v.
$ Applying the difference estimate in \cref{Thm:heat-exis} and \eqref{eq:4.coeff-limit} yields $ \E \sup_{0\le s\le t^*} \|U^{(m)}(s)-\bar u(s)\|_p^p \longrightarrow0.
$ Together with \eqref{eq:4.strong-picard-limit}, this gives $u=\bar u$ up to indistinguishability.
Hence $u$ solves \eqref{eq:4_iteration_n} on $[0,t^*]$.
The same estimate applied to two solutions gives uniqueness on this interval.

It remains to extend the solution to $[0,T]$.
Let $ N^* := \left\lceil\frac{T}{t^*}\right\rceil, \qquad t_j:=\min\{jt^*,T\}, \qquad j=0,\ldots,N^*.
$ Suppose that $u^{(n)}$ has been constructed on $[0,t_j]$.
Then $ u^{(n)}(t_j) \in L^p(\Omega,\mathcal F_{t_j};E).
$ On the shifted stochastic basis set $ \mathcal F_r^{(j)} := \mathcal F_{t_j+r}, \ W^{(j)}(r) := W(t_j+r)-W(t_j).
$ For $A\in\mathcal Z$ with $\mu(A)<\infty$, define $ N^{(j)}((0,r]\times A) := N((t_j,t_j+r]\times A), $ and let $ \widetilde N^{(j)}(\dd r,\dd z) := N^{(j)}(\dd r,\dd z)-\dd r\,\mu(\dd z).
$ Then $N^{(j)}$ is an $(\mathcal F_r^{(j)})$-Poisson random measure with compensator $\dd r\,\mu(\dd z)$, and for every admissible predictable integrand $K$, $ \int_0^r\int_Z K(s,z)\, \widetilde N^{(j)}(\dd s,\dd z) = \int_{t_j}^{t_j+r}\int_Z K(s-t_j,z)\, \widetilde N(\dd s,\dd z).
$ Thus the preceding contraction argument applies on $[t_j,t_{j+1}]$ with the same $t^*$ and the same contraction constant.
Since the shifted Poisson measure uses $(t_j,t_j+r]$, a possible jump at $t_j$ is already contained in the initial value $u^{(n)}(t_j)$ and is not counted a second time.
Finite iteration gives a unique adapted c\`adl\`ag solution $ u^{(n)} \in L^p\bigl(\Omega;\D([0,T];E)\bigr).
$

Finally, we have the constructed process $u^{(n)}$ on $[0,T]$ as the solution of the linear heat equation with the now fixed coefficients defined earlier.
Since $0\le\Phi_{u^{(n)}}\le1$, the estimates \eqref{eq:4.frozen-bounds}--\eqref{eq:4.frozen-jump-bounds} give
\begin{equation*}
\begin{aligned}
& \E\int_0^T \Big( \| \Phi_{u^{(n)}}F_v \|_{q_f}^p + \| \Phi_{u^{(n)}}H_v \|_{q_h}^p + \| \Phi_{u^{(n)}}\Sigma_v \|_{\mathbb L^p}^p \Big)\,\dd t \\
&\quad+ \E \int_0^T\int_Z \| \Phi_{u^{(n)}}(t-)\Gamma_v(t,z) \|_p^p \,\mu(\dd z)\,\dd t + \E \left( \int_0^T\int_Z \| \Phi_{u^{(n)}}(t-)\Gamma_v(t,z) \|_p^2 \,\mu(\dd z)\,\dd t \right)^{p/2} \\
&\qquad\le C_{p,q_f,q_h,k,R,N,\nu} \bigl( T+T^{p/2} \bigr).
\end{aligned}
\end{equation*}
Applying \cref{Thm:heat-exis} on the full interval $[0,T]$ and using $ \|{\mathrm P}_{\le k}u_0\|_p \le \|u_0\|_p, $ we obtain \eqref{eq:4_main_energy_bound} where $C$ may depend on $p,k,R,N,\nu,T$ and the structural constants of the noise, but is independent of $n$.
\end{proof}
\begin{proof}
[Proof of \cref{thm:main_existence}] By \cref{lem:4_trunexist}, for every $n\ge0$ the iterate $u^{(n)}$ is globally defined on $[0,T]$, $ u^{(n)} \in L^p\bigl(\Omega;\D([0,T];E)\bigr), $ and satisfies \eqref{eq:4_main_energy_bound} with a constant independent of $n$.
It remains to prove convergence of the outer iteration.

Set $ \Phi(a):=\phi(\|a\|_p), $ and define $$
\begin{aligned}
\mathcal B(a) &:= \mathcal P^{(1)}{\mathrm P}_{\le k} \bigl( a\otimes{\mathrm P}_{\le k}a \bigr), \mathcal T(a) := {\mathrm P}_{\le k}\mathcal P \bigl( g_N(|{\mathrm P}_{\le k}a|^2) {\mathrm P}_{\le k}a \bigr), \\
\Sigma(a) &:= {\mathrm P}_{\le k}\mathcal P \sigma({\mathrm P}_{\le k}a), \ \Gamma(a,z):= {\mathrm P}_{\le k}\mathcal P G({\mathrm P}_{\le k}a,z).
\end{aligned}
$$ Here $\mathcal P^{(1)}$ is the tensorial Leray multiplier introduced in \eqref{eq:4.commutation}.
Choose $ q_f, q_h $ such that these exponents lie in the admissible ranges of \cref{Thm:heat-exis}.

Since $k$ is fixed, Lemma~\ref{lem:2.2} yields, for every finite $r\ge p$,
\begin{equation}
\label{eq:4.fixed-k-smoothing}
\|{\mathrm P}_{\le k}a\|_r \le C_{k,p,r}\|a\|_p.
\end{equation}
One get,
\begin{align}
\|\mathcal B(a)\|_{q_f} &\le C_k\|a\|_p^2, & \|\mathcal B(a)-\mathcal B(b)\|_{q_f} &\le C_k (\|a\|_p+\|b\|_p) \|a-b\|_p,
\label{eq:4.B-local}
\\
\|\mathcal T(a)\|_{q_h} &\le C_{k,N,\nu}\|a\|_p^3, & \|\mathcal T(a)-\mathcal T(b)\|_{q_h} &\le C_{k,N,\nu} (\|a\|_p^2+\|b\|_p^2) \|a-b\|_p.
\label{eq:4.T-local}
\end{align}
The latter estimates follow from
\begin{equation*}
\begin{aligned}
|g_N(|\xi|^2)\xi| &\le C_{N,\nu}|\xi|^3, \\
\left| g_N(|\xi|^2)\xi - g_N(|\eta|^2)\eta \right| &\le C_{N,\nu} (|\xi|^2+|\eta|^2)|\xi-\eta|.
\end{aligned}
\end{equation*}
Likewise, \eqref{eq:2_WienerH} and \eqref{eq:4.fixed-k-smoothing} imply
\begin{align}
\|\Sigma(a)\|_{\mathbb L^p} &\le C_k(1+\|a\|_p^2),
\label{eq:4.Sigma-growth}
\\
\|\Sigma(a)-\Sigma(b)\|_{\mathbb L^p} &\le C_{k,R}\|a-b\|_p
\label{eq:4.Sigma-local}
\end{align}
whenever $\|a\|_p\vee\|b\|_p\le4R$.

For the jump coefficient, \eqref{eq:4.G1} gives, for $r\in\{2,p\}$,
\begin{equation}
\label{eq:4.Gamma-growth}
\int_Z \|\Gamma(a,z)\|_p^r\,\mu(\dd z) \le C_k(1+\|a\|_p^r).
\end{equation}
Provided $\|a\|_p\vee\|b\|_p\le4R$, \eqref{eq:4.G2}, H\"older's inequality, and \eqref{eq:4.fixed-k-smoothing},
\begin{equation}
\label{eq:4.Gamma-local}
\begin{aligned}
\int_Z \|\Gamma(a,z)-\Gamma(b,z)\|_p^r \,\mu(\dd z) &\le C \left\| \bigl( |{\mathrm P}_{\le k}a| + |{\mathrm P}_{\le k}b| \bigr)^\alpha {\mathrm P}_{\le k}(a-b) \right\|_p^r \le C (\|a\|_p+\|b\|_p)^{\alpha r} \| {\mathrm P}_{\le k}(a-b) \|_{p/(1-\alpha)}^r \\
&\le C_{k,R}\|a-b\|_p^r, \qquad r\in\{2,p\},
\end{aligned}
\end{equation}

The support condition on $\phi$, together with \eqref{eq:4.B-local}--\eqref{eq:4.Gamma-local}, implies that the truncated maps $$
\begin{aligned}
a&\longmapsto\Phi(a)\mathcal B(a), & a&\longmapsto\Phi(a)\mathcal T(a), \\
a&\longmapsto\Phi(a)\Sigma(a), & a&\longmapsto \Phi(a)\Gamma(a,\cdot)
\end{aligned}
$$ are globally Lipschitz from $E$ into, respectively, $ L^{q_f}, \ L^{q_h}, \ \mathbb L^p, \ L^r(Z,\mu;E), \ r\in\{2,p\}, $ with constants depending only on the fixed parameters.
For each $A \in \{\mathcal{B}, \mathcal{T}, \Sigma\}$ with its corresponding coefficient space $Y_A$, there exists a constant $C_{k,R} > 0$ such that $\| \Phi(a)A(a) - \Phi(b)A(b) \|_{Y_A} \le C_{k,R} \| a - b \|_p$ for all $a,b \in E$.
Likewise, for $r \in \{2,p\}$, we have $\int_Z \| \Phi(a)\Gamma(a,z) - \Phi(b)\Gamma(b,z) \|_p^r \, \mu(dz) \le C_{k,R} \| a - b \|_p^r$.
Indeed, these estimates follow from the preceding local Lipschitz bounds when $\|a\|_p \vee \|b\|_p \le 5R$.
If, say, $\|a\|_p \le 4R$ and $\|b\|_p > 5R$, then $\Phi(b) = 0$ and $\|a - b\|_p \ge R$; the boundedness of the truncated coefficient on $\{\|a\|_p \le 4R\}$ gives the same estimate.
The remaining cases are immediate from the support and Lipschitz properties of $\Phi$ where the transition region is controlled by \eqref{eq:4.cutoff-Lipschitz}.
In particular, all these truncated maps are uniformly bounded.

Set $ v^{(n)} := u^{(n+1)}-u^{(n)}.
$ Subtracting the equations for $u^{(n+1)}$ and $u^{(n)}$ gives
\begin{equation}
\label{eq:4.outer-difference}
\begin{aligned}
\dd v^{(n)} =& \Big[ \nu\Delta v^{(n)} + \nabla\cdot f^{(n)} + h^{(n)} \Big]\,\dd t + g^{(n)}\,\dd W_t + \int_Z G^{(n)}(t-,z)\, \widetilde N(\dd t,\dd z), \\
v^{(n)}(0) =&0, \qquad \nabla\cdot v^{(n)}=0,
\end{aligned}
\end{equation}
where
\begin{align*}
f^{(n)} &= - \phi^{(n+1)}\phi^{(n)} \mathcal B(u^{(n)}) + \phi^{(n)}\phi^{(n-1)} \mathcal B(u^{(n-1)}), \\
h^{(n)} &= - \phi^{(n+1)}\phi^{(n)} \mathcal T(u^{(n)}) + \phi^{(n)}\phi^{(n-1)} \mathcal T(u^{(n-1)}), \\
g^{(n)} &= \phi^{(n+1)}\phi^{(n)} \Sigma(u^{(n)}) - \phi^{(n)}\phi^{(n-1)} \Sigma(u^{(n-1)}), \\
G^{(n)}(t-,z) &= \phi^{(n+1)}(t-)\phi^{(n)}(t-) \Gamma(u^{(n)}(t-),z) - \phi^{(n)}(t-)\phi^{(n-1)}(t-) \Gamma(u^{(n-1)}(t-),z).
\end{align*}

The telescoping decomposition used in Section $4$ of \cite{MR4908978} reads
\begin{equation}
\label{eq:4.abstract-telescope}
\begin{aligned}
& \phi^{(n+1)}\phi^{(n)}A(u^{(n)}) - \phi^{(n)}\phi^{(n-1)}A(u^{(n-1)}) \\
&\quad= \phi^{(n)} (\phi^{(n+1)}-\phi^{(n)}) A(u^{(n)}) + \phi^{(n)} (\phi^{(n)}-\phi^{(n-1)}) A(u^{(n)})+\phi^{(n)}\phi^{(n-1)} \bigl( A(u^{(n)})-A(u^{(n-1)}) \bigr).
\end{aligned}
\end{equation}
Also,
\begin{equation}
\label{eq:4.cutoff-difference}
|\phi^{(n+1)}-\phi^{(n)}| \le C_\phi\|v^{(n)}\|_p, \qquad |\phi^{(n)}-\phi^{(n-1)}| \le C_\phi\|v^{(n-1)}\|_p.
\end{equation}
Whenever $\phi^{(j)}\ne0$, $ \|u^{(j)}\|_p\le4R.
$ Hence for $ A_n(t) := \E \sup_{0\le s\le t} \|v^{(n)}(s)\|_p^p.
$ with \eqref{eq:4.B-local}--\eqref{eq:4.Gamma-local}, \eqref{eq:4.abstract-telescope}, and \eqref{eq:4.cutoff-difference} imply
\begin{equation}
\label{eq:4.outer-det-W-bounds} \E \int_0^t \Big( \|f^{(n)}(s)\|_{q_f}^p + \|h^{(n)}(s)\|_{q_h}^p + \|g^{(n)}(s)\|_{\mathbb L^p}^p \Big)\,\dd s \le Ct \Big( A_n(t)+A_{n-1}(t) \Big),
\end{equation}
We record separately the estimates associated with the compensated Poisson term.
By \eqref{eq:4.Gamma-growth}--\eqref{eq:4.Gamma-local}, for $r\in\{2,p\}$,
\begin{equation*}
\int_Z \|G^{(n)}(s-,z)\|_p^r \,\mu(\dd z) \le C \Big( \|v^{(n)}(s-)\|_p^r + \|v^{(n-1)}(s-)\|_p^r \Big).
\end{equation*}
Consequently,
\begin{equation*}
\E \int_0^t\int_Z \|G^{(n)}(s-,z)\|_p^p \,\mu(\dd z)\,\dd s \le Ct \Big( A_n(t)+A_{n-1}(t) \Big),
\end{equation*}
while the quadratic Poisson norm satisfies
\begin{equation*}
\E \left( \int_0^t\int_Z \|G^{(n)}(s-,z)\|_p^2 \,\mu(\dd z)\,\dd s \right)^{p/2} \le Ct^{p/2} \Big( A_n(t)+A_{n-1}(t) \Big).
\end{equation*}
Here and below we use $ v^{(j)}(s-)=v^{(j)}(s) \ \text{for }\dd s\text{-a.e.
}s.
$ Likewise, \eqref{eq:4.outer-det-W-bounds} implies
\begin{equation}
\label{eq:4.outer-W-quadratic}
\E \left( \int_0^t \|g^{(n)}(s)\|_{\mathbb L^p}^2 \,\dd s \right)^{p/2} \le t^{\frac p2-1} \E \int_0^t \|g^{(n)}(s)\|_{\mathbb L^p}^p \dd s \le Ct^{p/2} \Big( A_n(t)+A_{n-1}(t) \Big).
\end{equation}

Applying \cref{Thm:heat-exis} to \eqref{eq:4.outer-difference}, together with \eqref{eq:4.outer-det-W-bounds}--\eqref{eq:4.outer-W-quadratic}, gives
\begin{equation*}
A_n(t) \le C \bigl( t+t^{p/2} \bigr) \bigl( A_n(t)+A_{n-1}(t) \bigr), \qquad 0<t\le T.
\end{equation*}
Choose $T^*>0$, depending only on the fixed parameters $p,k,\nu,N,R$ and the structural constants of the noise, such that $ C \bigl( T^*+(T^*)^{p/2} \bigr) \le \frac14.
$ Then
\begin{equation*}
A_n(T^*) \le \frac12 A_{n-1}(T^*).
\end{equation*}
Hence $ \sum_{n=0}^{\infty} A_n(T^*)^{1/p} <\infty.
$ It follows that $\{u^{(n)}\}$ is Cauchy in $ L^p\bigl( \Omega; L^\infty(0,T^*;E) \bigr).
$ The preceding summability implies almost sure uniform convergence of the Picard series.
One may invoke \cite[Lemma~5.2]{MR4385406}.
Since a uniform limit of $E$-valued c\`adl\`ag paths is c\`adl\`ag, there exists $ u \in L^p\bigl( \Omega;\D([0,T^*];E) \bigr) $ such that
\begin{equation}
\label{eq:4.outer-process-conv}
\E \sup_{0\le t\le T^*} \|u^{(n)}(t)-u(t)\|_p^p \longrightarrow0.
\end{equation}
In particular,
\begin{equation}
\label{eq:4.outer-left-limit-conv}
\E \sup_{0\le t\le T^*} \|u^{(n)}(t-)-u(t-)\|_p^p \longrightarrow0.
\end{equation}

We next identify the limiting equation.
To avoid confusing the coefficients of the difference equation with those of the $n$-th iterate, set $$
\begin{aligned}
\widehat f_n &:= -\phi^{(n)}\phi^{(n-1)} \mathcal B(u^{(n-1)}),\ \widehat h_n := -\phi^{(n)}\phi^{(n-1)} \mathcal T(u^{(n-1)}), \\
\widehat g_n &:= \phi^{(n)}\phi^{(n-1)} \Sigma(u^{(n-1)}),\ \widehat G_n(t,z) := \phi^{(n)}(t-)\phi^{(n-1)}(t-) \Gamma(u^{(n-1)}(t-),z).
\end{aligned}
$$ Thus $ f^{(n)} = \widehat f_{n+1}-\widehat f_n, $ and analogously for $h^{(n)},g^{(n)},G^{(n)}$.

Define $$
\begin{aligned}
f_u &:= -\Phi(u)^2\mathcal B(u), \ h_u := -\Phi(u)^2\mathcal T(u), \\
g_u &:= \Phi(u)^2\Sigma(u),\ G_u(t,z) := \Phi(u(t-))^2\Gamma(u(t-),z).
\end{aligned}
$$ The global Lipschitz properties of the truncated maps and \eqref{eq:4.outer-process-conv} give
\begin{align}
\widehat f_n &\longrightarrow f_u &&\text{in } L^p(\Omega\times(0,T^*);L^{q_f}),
\label{eq:4.limit-f}
\\
\widehat h_n &\longrightarrow h_u &&\text{in } L^p(\Omega\times(0,T^*);L^{q_h}),
\label{eq:4.limit-h}
\\
\widehat g_n &\longrightarrow g_u &&\text{in } L^p(\Omega\times(0,T^*);\mathbb L^p).
\label{eq:4.limit-g}
\end{align}

For completeness, we record the corresponding convergence of the jump coefficients.
From \eqref{eq:4.outer-process-conv}--\eqref{eq:4.outer-left-limit-conv} and the global Lipschitz property of $a\mapsto\Phi(a)\Gamma(a,\cdot)$, for $r\in\{2,p\}$,
\begin{equation*}
\begin{aligned}
\int_Z \| \widehat G_n(t,z)-G_u(t,z) \|_p^r \,\mu(\dd z) \le C \Big( \|u^{(n)}(t-)-u(t-)\|_p^r + \|u^{(n-1)}(t-)-u(t-)\|_p^r \Big).
\end{aligned}
\end{equation*}
Hence
\begin{equation*}
\E \int_0^{T^*}\int_Z \| \widehat G_n(t,z)-G_u(t,z) \|_p^p \,\mu(\dd z)\,\dd t \longrightarrow0,
\end{equation*}
\begin{equation}
\label{eq:4.limit-G2}
\E \left( \int_0^{T^*}\int_Z \| \widehat G_n(t,z)-G_u(t,z) \|_p^2 \,\mu(\dd z)\,\dd t \right)^{p/2} \longrightarrow0.
\end{equation}
Thus the compensated Poisson maximal inequality yields
\begin{equation*}
\E \sup_{0\le t\le T^*} \left\| \int_0^t\int_Z \bigl( \widehat G_n(s,z)-G_u(s,z) \bigr) \,\widetilde N(\dd s,\dd z) \right\|_p^p \longrightarrow0.
\end{equation*}
Similarly, by \eqref{eq:2-W-BDG},
\begin{equation*}
\E \sup_{0\le t\le T^*} \left\| \int_0^t \bigl( \widehat g_n(s)-g_u(s) \bigr)\,\dd W_s \right\|_p^p \longrightarrow0.
\end{equation*}

Let $\bar u$ denote the solution of \cref{Thm:heat-exis} with initial value ${\mathrm P}_{\le k}u_0$ and coefficients $ f_u,\ h_u,\ g_u,\ G_u.
$ The cutoff bounds ensure that these coefficients satisfy all hypotheses of \cref{Thm:heat-exis}.
Applying the corresponding difference estimate to $u^{(n)}-\bar u$ and using \eqref{eq:4.limit-f}--\eqref{eq:4.limit-G2}, we obtain $$ \E \sup_{0\le t\le T^*} \|u^{(n)}(t)-\bar u(t)\|_p^p \longrightarrow0.
$$ Together with \eqref{eq:4.outer-process-conv}, this yields $ u=\bar u$ up to indistinguishability .
Hence $u$ satisfies \eqref{eq:4_truncation} on $[0,T^*]$.

We next pass to the energy estimate.
By \cref{lem:4_trunexist}, $ \sup_{n\ge0} \E \int_0^{T^*} \left\| \nabla|u^{(n)}(s)|^{p/2} \right\|_2^2 \,\dd s <\infty.
$ Therefore \eqref{eq:4.outer-process-conv} and \cref{lem:L2_con} imply $ |u|^{p/2} \in L^2\!\left( \Omega\times(0,T^*); H^1(\R^3) \right) $ and
\begin{equation*}
\E \int_0^{T^*} \left\| \nabla|u(s)|^{p/2} \right\|_2^2\,\dd s \le \liminf_{n\to\infty} \E \int_0^{T^*} \left\| \nabla|u^{(n)}(s)|^{p/2} \right\|_2^2\,\dd s.
\end{equation*}
Moreover, $$ \left| \left( \E \sup_{0\le s\le T^*} \|u^{(n)}(s)\|_p^p \right)^{1/p} - \left( \E \sup_{0\le s\le T^*} \|u(s)\|_p^p \right)^{1/p} \right| \le \left( \E \sup_{0\le s\le T^*} \|u^{(n)}(s)-u(s)\|_p^p \right)^{1/p} \longrightarrow0.
$$ Thus \eqref{eq:4_main_energy_bound} holds on $[0,T^*]$.

We next prove pathwise uniqueness.
Let $u$ and $\widetilde u$ be two strong solutions of \eqref{eq:4_truncation} on $[0,T^*]$ with the same initial datum and driving noises, and set $ w:=u-\widetilde u.
$ Using the same telescoping decomposition and the global Lipschitz bounds for the truncated coefficients, the deterministic and Wiener terms satisfy $$ \E \int_0^t \Big( \|f^w(s)\|_{q_f}^p + \|h^w(s)\|_{q_h}^p + \|g^w(s)\|_{\mathbb L^p}^p \Big)\dd s \le Ct \E \sup_{0\le s\le t} \|w(s)\|_p^p.
$$ For the Poisson term, for $r\in\{2,p\}$, $ \int_Z \|G^w(s-,z)\|_p^r\,\mu(\dd z) \le C\|w(s-)\|_p^r, $ and hence
\begin{align*}
\E \int_0^t\int_Z \|G^w(s-,z)\|_p^p \,\mu(\dd z)\,\dd s &\le Ct \E \sup_{0\le s\le t} \|w(s)\|_p^p, \\
\E \left( \int_0^t\int_Z \|G^w(s-,z)\|_p^2 \,\mu(\dd z)\,\dd s \right)^{p/2} &\le Ct^{p/2} \E \sup_{0\le s\le t} \|w(s)\|_p^p.
\end{align*}
Since $w(0)=0$, \cref{Thm:heat-exis} gives $$ \E \sup_{0\le s\le t} \|w(s)\|_p^p \le C \bigl( t+t^{p/2} \bigr) \E \sup_{0\le s\le t} \|w(s)\|_p^p.
$$ After decreasing $T^*$ if necessary, $ C \bigl( T^*+(T^*)^{p/2} \bigr) <1, $ and therefore $ u=\widetilde u \ \P\text{-a.s.
on }[0,T^*].
$ Finally, $T^*$ depends only on the fixed parameters and not on the size of the initial datum.
The solution may therefore be extended to $[0,T]$ by finitely many deterministic restarts.
Let $ J := \left\lceil\frac{T}{T^*}\right\rceil, \ t_j := \min\{jT^*,T\}, \ j=0,\ldots,J.
$ Suppose that $u$ has been constructed on $[0,t_j]$ and set $ u_j:=u(t_j).
$ Then $ u_j\in L^p(\Omega,\mathcal F_{t_j};E), \ \nabla\cdot u_j=0.
$ On the shifted stochastic basis define $ \mathcal F_r^{(j)} := \mathcal F_{t_j+r}, \ W^{(j)}(r) := W(t_j+r)-W(t_j).
$ For $A\in\mathcal Z$ with $\mu(A)<\infty$, set $ N^{(j)}((0,r]\times A) := N((t_j,t_j+r]\times A) $ and $ \widetilde N^{(j)}(\dd r,\dd z) := N^{(j)}(\dd r,\dd z) - \dd r\,\mu(\dd z).
$ Then $N^{(j)}$ is an $(\mathcal F_r^{(j)})$-Poisson random measure with compensator $\dd r\,\mu(\dd z)$, and for every admissible predictable integrand $K$, $ \int_0^r\int_Z K(s,z)\, \widetilde N^{(j)}(\dd s,\dd z) = \int_{t_j}^{t_j+r}\int_Z K(s-t_j,z)\, \widetilde N(\dd s,\dd z).
$ The use of the interval $(t_j,t_j+r]$ ensures that a possible jump at $t_j$ is already contained in the initial value $u(t_j)$ and is not counted again after the restart.
Since $t_{j+1}-t_j\le T^*$, the preceding existence and uniqueness argument applies successively on every subinterval.
This yields a unique adapted c\`adl\`ag strong solution $ u \in L^p\bigl( \Omega;\D([0,T];E) \bigr).
$

It remains only to verify the estimate on the whole interval.
Regard the resulting process $u$ as the solution of the linear heat equation with coefficients $ f_u,\ h_u,\ g_u,\ G_u $ defined above.
By the support property of $\phi$, $$
\begin{aligned}
\|f_u(t)\|_{q_f} + \|h_u(t)\|_{q_h} + \|g_u(t)\|_{\mathbb L^p} &\le C, \\
\int_Z \|G_u(t,z)\|_p^r\,\mu(\dd z) &\le C, \qquad r\in\{2,p\},
\end{aligned}
$$ with $C$ independent of $u$.
Hence
\begin{equation*}
\begin{aligned}
& \E \int_0^T \Big( \|f_u(t)\|_{q_f}^p + \|h_u(t)\|_{q_h}^p + \|g_u(t)\|_{\mathbb L^p}^p \Big)\,\dd t \\
&\quad+ \E \int_0^T\int_Z \|G_u(t,z)\|_p^p \,\mu(\dd z)\,\dd t + \E \left( \int_0^T\int_Z \|G_u(t,z)\|_p^2 \,\mu(\dd z)\,\dd t \right)^{p/2} \\
&\qquad\le C_{p,k,\nu,N,R,T}.
\end{aligned}
\end{equation*}
Applying \cref{Thm:heat-exis} on $[0,T]$ and using $ \|{\mathrm P}_{\le k}u_0\|_p \le \|u_0\|_p, $ we obtain \cref{eq:4_main_energy_bound} where $C$ may depend on $p,k,\nu,N,R,T$ and the structural constants of the noise.
Pathwise uniqueness on $[0,T]$ follows by applying the preceding local uniqueness argument successively on the same deterministic subintervals.
\end{proof}

\begin{lemma}
[$L^2$-persistence for the truncated equation] \label{lem:4_L2} Let $p>3$ and retain the hypotheses of \cref{thm:main_existence}.
Assume additionally \eqref{Hypo_W2}, \eqref{eq:4_G5}, and $ u_0\in L^p(\Omega,\mathcal F_0;E) \cap L^2(\Omega,\mathcal F_0;L^2(\R^3)); \ \nabla\cdot u_0=0 \quad\P\text{-a.s.}.
$ For fixed $k\in\N$ and $R\ge1$, let $u$ denote the unique solution of \eqref{eq:4_truncation}, with $u(0)={\mathrm P}_{\le k}u_0$.
Then, for every $T>0$, $ u\in L^p\bigl(\Omega;\D([0,T];E)\bigr) \cap L^2\bigl(\Omega;\D([0,T];L^2)\bigr) \cap L^2\bigl(\Omega;L^2(0,T;H^1)\bigr), $ and
\begin{equation}
\label{eq:4.L2-energy}
\E\left[ \sup_{0\le t\le T}\|u(t)\|_2^2 + \nu\int_0^T\|\nabla u(s)\|_2^2\,\dd s \right] \le C_T\left(1+\E\|u_0\|_2^2\right).
\end{equation}
Here $C_T$ depends only on $T$, $\nu$, and the constants in \eqref{Hypo_W2} and \eqref{eq:4_G5}; in particular, it is independent of $k$ and of the scalar truncation level $R$.
\end{lemma}

\begin{proof}
Set $ H:=L^2(\R^3;\R^3), \ S_m:={\mathrm P}_{\le m}, \ S_k:={\mathrm P}_{\le k}.
$ For $m\in\N$, let $u_m$ be the unique $L^p$-solution of
\begin{equation}
\label{eq:4.L2-regularized}
\begin{aligned}
\dd u_m =& \Big[ \nu\Delta u_m - a_m S_mS_k\mathcal P \bigl((w_m\cdot\nabla)v_m\bigr) - a_m S_mS_k\mathcal P \bigl(g_N(|v_m|^2)v_m\bigr) \Big]\,\dd t \\
&+ a_m S_mS_k\mathcal P\sigma(v_m)\,\dd W_t + \int_Z a_m(t-) S_mS_k\mathcal PG(v_m(t-),z)\, \widetilde N(\dd t,\dd z), \\
u_m(0) =& S_mS_ku_0,
\end{aligned}
\end{equation}
where $ w_m:=S_mu_m, \ v_m:=S_kw_m, \ a_m(t):=\phi(\|u_m(t)\|_p)^2.
$ The scalar factors in the stochastic terms are understood through their predictable representatives.
The existence and uniqueness of $u_m$ follow from the fixed-$R$ construction used in \cref{thm:main_existence}, since the insertion of $S_m$ preserves the required multiplier bounds.

We first record that, for each fixed $m,k,R$, the solution $u_m$ also belongs to the Hilbert energy class.
Indeed, on the support of $a_m$ the fixed-frequency smoothing estimates imply
\begin{equation*}
\left\| a_mS_mS_k\mathcal P \bigl((w_m\cdot\nabla)v_m\bigr) \right\|_2+ \left\| a_mS_mS_k\mathcal P \bigl(g_N(|v_m|^2)v_m\bigr) \right\|_2 \le C_{m,k,R,N,\nu,p}\|u_m\|_2.
\end{equation*}
By \eqref{Hypo_W2}, \eqref{eq:4_G5}, and the $L^2$-contractivity of $S_m$, $S_k$, and $\mathcal P$,
\begin{equation*}
\|S_mS_k\mathcal P\sigma(v)\|_{\mathbb L^2}^2 + \int_Z \|S_mS_k\mathcal PG(v,z)\|_2^2\,\mu(\dd z) \le C\bigl(1+\|v\|_2^2\bigr).
\end{equation*}
Running the Picard construction for \eqref{eq:4.L2-regularized}, with the coefficients frozen at the preceding iterate, and using the standard Hilbert-space heat-convolution estimate in the Gelfand triple $ H^1(\R^3; \R^3) \hookrightarrow L^2(\R^3; \R^3) \hookrightarrow H^{-1}(\R^3; \R^3), $ with $\nabla \cdot u =0$ gives, on every deterministic finite interval,
\begin{equation}
\label{eq:4.L2-preliminary}
u_m \in L^2\bigl(\Omega;\D([0,T];H)\bigr) \cap L^2\bigl(\Omega;L^2(0,T;H^1)\bigr).
\end{equation}
Here the constants used only to establish \eqref{eq:4.L2-preliminary} may depend on $m,k,R$.

Define $ Q_m^W := S_mS_k\mathcal P\sigma(v_m), \ Q_m^J(z) := S_mS_k\mathcal PG(v_m,z).
$ Since $S_m$ and $S_k$ are scalar, self-adjoint $L^2$-contractions commuting with $\mathcal P$ and with spatial derivatives, while $ \nabla\cdot u_m = \nabla\cdot w_m = \nabla\cdot v_m = 0, $ we have
\begin{align*}
\Bigl( u_m, S_mS_k\mathcal P ((w_m\cdot\nabla)v_m) \Bigr)_2 &= \bigl( v_m,(w_m\cdot\nabla)v_m \bigr)_2 =0, \\
\Bigl( u_m, S_mS_k\mathcal P (g_N(|v_m|^2)v_m) \Bigr)_2 &= \int_{\R^3} g_N(|v_m|^2)|v_m|^2\,\dd x \ge0.
\end{align*}
The first identity follows by density from $\nabla\cdot w_m=0$.

By \eqref{eq:4.L2-preliminary}, the variational It\^o formula with jumps in the above Gelfand triple applies to $\|u_m(t)\|_2^2$.
Now, for every $t\in[0,T]$,
\begin{equation}
\label{eq:4.L2-Ito-m}
\begin{aligned}
\|u_m(t)\|_2^2 &+ 2\nu\int_0^t\|\nabla u_m(s)\|_2^2\,\dd s + 2\int_0^t a_m(s) \int_{\R^3} g_N(|v_m(s)|^2)|v_m(s)|^2\,\dd x\,\dd s = \|S_mS_ku_0\|_2^2 + \\
& \int_0^t a_m(s)^2 \|Q_m^W(s)\|_{\mathbb L^2}^2\,\dd s + \int_0^t\int_Z a_m(s-)^2 \|Q_m^J(s-,z)\|_2^2\, \mu(\dd z)\,\dd s + M_m^W(t) + M_m^{J,1}(t) + M_m^{J,2}(t),
\end{aligned}
\end{equation}
where
\begin{align*}
M_m^W(t) := 2\int_0^t a_m(s-) \bigl( u_m(s-),Q_m^W(s-)\,\dd W_s \bigr)_2,\ & M_m^{J,1}(t) := 2\int_0^t\int_Z a_m(s-) \bigl( u_m(s-),Q_m^J(s-,z) \bigr)_2\, \widetilde N(\dd s,\dd z), \\
M_m^{J,2}(t) &:= \int_0^t\int_Z a_m(s-)^2 \|Q_m^J(s-,z)\|_2^2\, \widetilde N(\dd s,\dd z).
\end{align*}
The last term is understood as $ \int_0^t\int_ZK_m\,N(\dd s,\dd z) - \int_0^t\int_ZK_m\,\mu(\dd z)\,\dd s, \ K_m := a_m^2\|Q_m^J\|_2^2.
$ Thus only the second moment of the jump coefficient is used.

Since $0\le a_m\le1$, \eqref{Hypo_W2}, \eqref{eq:4_G5}, and $L^2$-contractivity give
\begin{equation}
\label{eq:4.L2-noise-growth-m}
\|Q_m^W(t)\|_{\mathbb L^2}^2 + \int_Z \|Q_m^J(t,z)\|_2^2\,\mu(\dd z) \le C\bigl(1+\|v_m(t)\|_2^2\bigr) \le C\bigl(1+\|u_m(t)\|_2^2\bigr),
\end{equation}
where $C$ is independent of $m$, $k$, and $R$.

For $M\ge1$, set $ \tau_{m,M} := \inf\{t\ge0:\|u_m(t)\|_2\ge M\}, \ \inf\varnothing:=\infty.
$ Since $u_m$ is $H$-valued c\`adl\`ag, $\tau_{m,M}$ is an $(\mathcal F_t)$-stopping time.
A jump may overshoot the level $M$, but the Poisson integrands are evaluated at the predictable pre-jump state.

For $\tau=t\wedge\tau_{m,M}$, the Burkholder--Davis--Gundy inequality and Young's inequality yield, for every $\delta>0$,
\begin{equation*}
\E \sup_{0\le r\le\tau} |M_m^W(r)| \le \delta \E \sup_{0\le r\le\tau}\|u_m(r)\|_2^2+ C_\delta \E \int_0^\tau a_m(s)^2 \|Q_m^W(s)\|_{\mathbb L^2}^2\,\dd s .
\end{equation*}
For the linear jump martingale, Davis' inequality applied to its optional quadratic variation gives
\begin{equation*}
\E \sup_{0\le r\le\tau} |M_m^{J,1}(r)| \le \delta \E \sup_{0\le r\le\tau} \|u_m(r)\|_2^2 + C_\delta \E \int_0^\tau\int_Z a_m(s-)^2 \|Q_m^J(s-,z)\|_2^2\, \mu(\dd z)\,\dd s .
\end{equation*}
Indeed, $ [M_m^{J,1}]_\tau = 4\int_0^\tau\int_Z a_m(s-)^2 \bigl| (u_m(s-),Q_m^J(s-,z))_2 \bigr|^2 N(\dd s,\dd z), $ and the compensator identity is used after Young's inequality.

For the quadratic jump term, $K_m\ge0$ gives pathwise $ \sup_{0\le r\le\tau} \left| \int_0^r\int_ZK_m\,\widetilde N(\dd s,\dd z) \right| \le \int_0^\tau\int_ZK_m\,N(\dd s,\dd z) + \int_0^\tau\int_ZK_m\,\mu(\dd z)\,\dd s.
$ Hence
\begin{equation}
\label{eq:4.L2-quadratic-jump-bound}
\E \sup_{0\le r\le\tau} |M_m^{J,2}(r)| \le 2\E \int_0^\tau\int_Z a_m(s-)^2 \|Q_m^J(s-,z)\|_2^2\, \mu(\dd z)\,\dd s.
\end{equation}
In particular, no fourth-moment assumption on $G$ is required.

Combining \eqref{eq:4.L2-Ito-m}--\eqref{eq:4.L2-quadratic-jump-bound}, discarding the nonnegative taming contribution, choosing $\delta>0$ sufficiently small, and using \eqref{eq:4.L2-noise-growth-m}, we obtain
\begin{equation*}
\begin{aligned}
& \E\left[ \sup_{0\le r\le t\wedge\tau_{m,M}} \|u_m(r)\|_2^2 + \nu \int_0^{t\wedge\tau_{m,M}} \|\nabla u_m(s)\|_2^2\,\dd s \right] \\
&\qquad\le C\left(1+\E\|u_0\|_2^2\right) + C\int_0^t \left[ 1+ \E \sup_{0\le r\le s\wedge\tau_{m,M}} \|u_m(r)\|_2^2 \right]\dd s ,
\end{aligned}
\end{equation*}
where $C$ is independent of $m$, $k$, $R$, and $M$.
Gronwall's lemma therefore yields
\begin{equation}
\label{eq:4.L2-stopped-uniform}
\E\left[ \sup_{0\le r\le T\wedge\tau_{m,M}} \|u_m(r)\|_2^2 + \nu \int_0^{T\wedge\tau_{m,M}} \|\nabla u_m(s)\|_2^2\,\dd s \right] \le C_T \left(1+\E\|u_0\|_2^2\right).
\end{equation}
For each fixed $m$, the $H$-c\`adl\`ag path of $u_m$ is bounded on $[0,T]$ almost surely; hence $\tau_{m,M}>T$ for all sufficiently large $M$, pathwise.
Letting $M\to\infty$ in \eqref{eq:4.L2-stopped-uniform} and using Fatou's lemma gives
\begin{equation}
\label{eq:4.L2-approx-uniform}
\sup_{m\ge1} \E\left[ \sup_{0\le t\le T}\|u_m(t)\|_2^2 + \nu\int_0^T\|\nabla u_m(s)\|_2^2\,\dd s \right] \le C_T\left(1+\E\|u_0\|_2^2\right).
\end{equation}

It remains to remove the auxiliary multiplier $S_m$.
Since $S_m\to I$ strongly on every Lebesgue and $\gamma$-radonifying space entering the fixed-$k$ construction, and since for fixed $k,R$ the truncated coefficient maps are globally Lipschitz in the corresponding forcing spaces, the same stability argument as in the proof of \cref{thm:main_existence} gives
\begin{equation*}
A_m(t) \le C_{k,R}\bigl(t+t^{p/2}\bigr)A_m(t) + \varepsilon_m(t), \qquad A_m(t) := \E \sup_{0\le s\le t} \|u_m(s)-u(s)\|_p^p,
\end{equation*}
where $\varepsilon_m(T)\to0$ as $m\to\infty$.
The consistency error contains the initial-data term $ \|S_mS_ku_0-S_ku_0\|_p^p $ and the corresponding deterministic, Wiener, and compensated Poisson coefficient errors, the latter in both Poisson norms of \cref{Thm:heat-exis}.
Choosing $t_*>0$ such that $ C_{k,R} \bigl(t_*+t_*^{p/2}\bigr) \le\frac12 $ and iterating over finitely many deterministic subintervals yields
\begin{equation}
\label{eq:4.L2-Lp-convergence}
\E \sup_{0\le t\le T} \|u_m(t)-u(t)\|_p^p \longrightarrow0.
\end{equation}
The dependence of the stability constant on $k$ and $R$ is used only for the identification of the limit in \eqref{eq:4.L2-Lp-convergence}; it does not enter the uniform estimate \eqref{eq:4.L2-approx-uniform}.
Passing to a deterministic subsequence, we may assume
\begin{equation}
\label{eq:4.L2-pathwise-Lp-convergence}
\sup_{0\le t\le T} \|u_m(t)-u(t)\|_p \longrightarrow0 \qquad \P\text{-a.s.}
\end{equation}
Set $ X_m := \sup_{0\le t\le T}\|u_m(t)\|_2^2 + \nu\int_0^T\|\nabla u_m(s)\|_2^2\,\dd s.
$ By \eqref{eq:4.L2-approx-uniform}, $\liminf_mX_m<\infty$ almost surely.
For such an $\omega$, choose a subsequence realizing the lower limit.
Along this subsequence, $\{u_m(\omega)\}$ is bounded in $ L^\infty(0,T;L^2) \cap L^2(0,T;H^1).
$ The convergence \eqref{eq:4.L2-pathwise-Lp-convergence} implies distributional convergence on $(0,T)\times\R^3$ and, for every fixed $t\in[0,T]$, distributional convergence in $\R^3$.
Weak compactness therefore identifies every weak limit with $u$.
By weak lower semicontinuity,
\begin{equation*}
\sup_{0\le t\le T}\|u(t)\|_2^2 + \nu\int_0^T\|\nabla u(s)\|_2^2\,\dd s \le \liminf_{m\to\infty} \left[ \sup_{0\le t\le T}\|u_m(t)\|_2^2 + \nu\int_0^T\|\nabla u_m(s)\|_2^2\,\dd s \right] \qquad \P\text{-a.s.}
\end{equation*}
Fatou's lemma and \eqref{eq:4.L2-approx-uniform} now give \eqref{eq:4.L2-energy}.
In particular,
\begin{equation}
\label{eq:4.L2-energy-membership}
u\in L^2\bigl(\Omega;L^\infty(0,T;L^2)\bigr) \cap L^2\bigl(\Omega;L^2(0,T;H^1)\bigr).
\end{equation}

It remains only to identify the $L^2$ path regularity.
For fixed $k$, the smoothing properties of $S_k$ give
\begin{equation}
\label{eq:4.L2-final-drift}
\begin{aligned}
\| S_k\mathcal P ((u\cdot\nabla)S_ku) \|_2 &\le C_k\|u\|_2^2, \\
\| S_k\mathcal P (g_N(|S_ku|^2)S_ku) \|_2 &\le C_{k,N,\nu} \bigl( \|u\|_2+\|u\|_2^3 \bigr).
\end{aligned}
\end{equation}
Thus the deterministic drift is pathwise integrable in $L^2$ on $[0,T]$.
Together with \eqref{Hypo_W2}, \eqref{eq:4_G5}, and \eqref{eq:4.L2-energy-membership}, the Wiener and jump coefficients belong to the corresponding Hilbert-space stochastic-integrability classes.
Hence the $L^2$-valued representation determined by \eqref{eq:4_truncation} is well defined.
After pairing with Schwartz test functions, this representation agrees with the already constructed $L^p$-solution.
Therefore the two processes coincide as distributions and hence, up to indistinguishability, as representatives of the same solution.
The deterministic heat convolution and the Wiener stochastic convolution have continuous $L^2$ paths, while the compensated Poisson convolution admits an $L^2$-valued c\`adl\`ag modification.
Therefore, $ u\in L^2\bigl( \Omega; \D([0,T];L^2(\R^3)) \bigr).
$.
\end{proof}
\section{Convergence of approximations and local well-posedness}
\label{sec5}

We now remove the spatial regularization introduced in \cref{sec4}.
Throughout this section the scalar truncation level $R\ge1$ is fixed, and $\phi$ denotes the cutoff used in \cref{sec4}.
The constants in this section may therefore depend on $R$; their uniformity will only be required with respect to the spectral regularization parameter.

We use the tensor convention $$ (a\otimes b)_{ij}:=b_i a_j, \qquad (\nabla\cdot F)_i:=\sum_{j=1}^3\partial_jF_{ij}.
$$ Consequently, $\nabla\cdot(a\otimes b)=(a\cdot\nabla)b$ whenever $\nabla\cdot a=0$.
The tensorial Leray multiplier $\mathcal P^{(1)}$ acts on the first index of $F$.
These conventions agree with the component form of the regularized convective flux used below.

Let $ \chi\in C_c^\infty(\R^3;[0,1]), \ \chi(x)=1 \ \text{for }|x|\le1, $ be the spatial cutoff used to approximate the initial datum.
Let $k:\N\to\N$ be strictly increasing with $ k(n)\longrightarrow\infty , $ and set $ P_n:={\mathrm P}_{\le k(n)} .
$ Recall that $\{P_n\}_{n\ge1}$ is uniformly bounded on every $L^q(\R^3)$, $1<q<\infty$, and $ P_nf\longrightarrow f \ \text{strongly in }L^q(\R^3).
$

Let $d=3$ and $p>3$.
For $n\in\N$, consider
\begin{equation}
\label{eq:5.1_truncated}
\begin{aligned}
\dd u^{(n)} =& \Big[ \nu\Delta u^{(n)} - (\phi^{(n)})^2 P_n\mathcal P \bigl( (u^{(n)}\cdot\nabla)P_nu^{(n)} \bigr) - (\phi^{(n)})^2 P_n\mathcal P \bigl( g_N(|P_nu^{(n)}|^2)P_nu^{(n)} \bigr) \Big]\,\dd t \\
&+ (\phi^{(n)})^2 P_n\mathcal P \sigma(P_nu^{(n)})\,\dd W_t + \int_Z (\phi^{(n)}(t-))^2 P_n\mathcal P G(P_nu^{(n)}(t-),z) \,\widetilde N(\dd t,\dd z), \\
\nabla\cdot u^{(n)} =&0, \\
u^{(n)}(0) =& P_n\mathcal P \left( \chi\left(\frac{\cdot}{n}\right)u_0 \right), \qquad \P\text{-a.s.},
\end{aligned}
\end{equation}
where $ \phi^{(n)}(t) := \phi(\|u^{(n)}(t)\|_p), \ \phi^{(n)}(t-) := \phi(\|u^{(n)}(t-)\|_p).
$

Since $p>2$ and $\chi(\cdot/n)u_0$ has compact support, $ \chi\left(\frac{\cdot}{n}\right)u_0 \in L^2(\R^3;\R^3)\cap E \quad \P\text{-a.s.} $.
Also $ \left\| \chi\left(\frac{\cdot}{n}\right)u_0 \right\|_p \le \|u_0\|_p $ and, by dominated convergence, $ \chi\left(\frac{\cdot}{n}\right)u_0 \longrightarrow u_0 \ \text{in}\ E \ \P\text{-a.s.} $ Since $\nabla\cdot u_0=0$ and $\mathcal Pu_0=u_0$,
\begin{equation*}
\left\| \mathcal P \left( \chi\left(\frac{\cdot}{n}\right)u_0 \right) -u_0 \right\|_p \le C_p \left\| \left( \chi\left(\frac{\cdot}{n}\right)-1 \right)u_0 \right\|_p \longrightarrow0 .
\end{equation*}
Consequently,
\begin{equation}
\label{eq:5_initialbd}
\left\| P_n\mathcal P \left( \chi\left(\frac{\cdot}{n}\right)u_0 \right) -u_0 \right\|_p \le C_p \left\| \mathcal P \left( \chi\left(\frac{\cdot}{n}\right)u_0 \right) -u_0 \right\|_p + \|(P_n-I)u_0\|_p \longrightarrow0 .
\end{equation}
If $u_0\in L^p(\Omega;E)$, the convergence in \eqref{eq:5_initialbd} also holds in $L^p(\Omega;E)$.

By \cref{thm:main_existence}, for every $n\in\N$ and $T>0$, \eqref{eq:5.1_truncated} admits a unique strong solution $ u^{(n)} \in L^p\bigl( \Omega;\D([0,T];E) \bigr).
$ The estimate from \cref{thm:main_existence} depends on $k(n)$ and cannot be used directly in the limit $n\to\infty$.
We therefore derive estimates uniform in $n$ after localization.

Fix $K\ge1$ and assume $ \|u_0\|_p\le K \ \P\text{-a.s.} $ The uniform $L^p$-boundedness of $P_n$ and $\mathcal P$ gives a constant $C_p$, independent of $n$, such that $ \|u^{(n)}(0)\|_p^p \le C_pK^p \ \P\text{-a.s.} $ Choose $M_0\ge1$, depending only on $p$ and the multiplier bounds, such that
\begin{equation}
\label{eq:5.M0-initial}
\sup_{n\ge1} \|u^{(n)}(0)\|_p^p \le \frac14M_0K^p \qquad \P\text{-a.s.}
\end{equation}

For $M\ge M_0$, define
\begin{equation*}
\mathcal Q_n(t) := \sup_{0\le s\le t} \|u^{(n)}(s)\|_p^p + \int_0^t \|u^{(n)}(s)\|_{3p}^p\,\dd s
\end{equation*}
and
\begin{equation*}
\tau_M^n := \inf \left\{ t>0: \mathcal Q_n(t)\ge MK^p \right\}, \qquad \inf\varnothing:=\infty .
\end{equation*}
Since $u^{(n)}$ is adapted and c\`adl\`ag, $\mathcal Q_n$ is adapted, nondecreasing, and right-continuous; hence $\tau_M^n$ is an $(\mathcal F_t)$-stopping time.
Furthermore, \eqref{eq:5.M0-initial} and right-continuity at $t=0$ imply $ \tau_M^n>0 \ \P\text{-a.s.} $

For $ \theta_n(t):=t\wedge\tau_M^n, $ we have
\begin{equation}
\label{eq:5.preexit-control}
\begin{aligned}
\sup_{0\le s<\theta_n(t)} \|u^{(n)}(s)\|_p^p \le MK^p,& \sup_{0\le s\le\theta_n(t)} \|u^{(n)}(s-)\|_p^p \le MK^p, \\
\int_0^{\theta_n(t)} \|u^{(n)}(s)\|_{3p}^p\,\dd s &\le MK^p .
\end{aligned}
\end{equation}
The first inequality does not include the post-jump value at $\tau_M^n$, the second controls precisely the predictable state which enters the Poisson coefficient.
The Lebesgue- and Wiener-time estimates below are unaffected by the value at the single terminal time.

The arguments follow the algebraic splitting used in [Section $5$, of \cite{MR4908978}], with the necessary modifications for the direct taming forcing and the compensated L\'evy term.

\begin{theorem}
\label{thm:5.1} Let $p>3$, $K\ge1$, and assume $ \|u_0\|_p\le K \qquad \P\text{-a.s.} $ For every $M\ge M_0$ and $T>0$, the solution $u^{(n)}$ of \eqref{eq:5.1_truncated} satisfies
\begin{equation}
\label{eq:5_eng_bnd}
\E \Bigg[ \sup_{0\le s\le T\wedge\tau_M^n} \|u^{(n)}(s)\|_p^p + \int_0^{T\wedge\tau_M^n} \left\| \nabla\bigl(|u^{(n)}(s)|^{p/2}\bigr) \right\|_2^2\,\dd s \Bigg] \le C_{R,M,T}K^p
\end{equation}
where $C_{R,M,T}$ may additionally depend on $p,\nu,N$ and the structural constants in the Wiener and jump hypotheses, but is independent of $n$ and $K$.

Moreover,
\begin{equation}
\label{eq:5_prob_bnd}
\lim_{t\downarrow0} \sup_{n\ge1} \P\Bigg( \sup_{0\le s\le t\wedge\tau_M^n} \|u^{(n)}(s)\|_p^p + \int_0^{t\wedge\tau_M^n} \|u^{(n)}(s)\|_{3p}^p\,\dd s \ge M_0K^p \Bigg) =0 .
\end{equation}
\end{theorem}

\begin{proof}
Write $ u_n:=u^{(n)}, \qquad \phi_n(t):=\phi(\|u_n(t)\|_p), $ and define
\begin{align*}
f_n &:= -\phi_n^2 \mathcal P^{(1)}P_n \bigl( u_n\otimes P_nu_n \bigr),\ h_n := -\phi_n^2 P_n\mathcal P \bigl( g_N(|P_nu_n|^2)P_nu_n \bigr), \\
g_n &:= \phi_n^2 P_n\mathcal P \sigma(P_nu_n),\ H_n(t,z) := \phi_n(t-)^2 P_n\mathcal P G(P_nu_n(t-),z).
\end{align*}
Then
\begin{equation}
\label{eq:5.heat-form}
\dd u_n = \bigl( \nu\Delta u_n+\nabla\cdot f_n+h_n \bigr)\,\dd t + g_n\,\dd W_t + \int_Z H_n(t,z)\,\widetilde N(\dd t,\dd z).
\end{equation}

Choose
\begin{equation}
\label{eq:5.qf-choice}
\max \left\{ \frac p2, \frac{3p}{p+1} \right\} <q_f<\frac{3p}{4}, \qquad \frac1{q_f} = \frac1p+\frac1\ell .
\end{equation}
Then $ p<\ell<3p.
$ Let $\alpha_f\in(0,1)$ be determined by $ \frac1\ell = \frac{\alpha_f}{p} + \frac{1-\alpha_f}{3p}, \ \alpha_f = \frac{3p/\ell-1}{2}.
$ The uniform $L^q$-boundedness of $P_n$ and $\mathcal P$, H\"older's inequality, interpolation, and $\phi_n^{2p}\le\phi_n^2$ give
\begin{equation}
\label{eq:5.f-est}
\begin{aligned}
\|f_n(s)\|_{q_f}^p &\le C\phi_n(s)^{2p} \|u_n(s)\|_p^p \|P_nu_n(s)\|_\ell^p \le C\phi_n(s)^2 \|u_n(s)\|_p^p \|u_n(s)\|_\ell^p \\
&\le C\phi_n(s)^2 \|u_n(s)\|_p^p \|u_n(s)\|_p^{p\alpha_f} \|u_n(s)\|_{3p}^{p(1-\alpha_f)} .
\end{aligned}
\end{equation}
Since $\phi_n(s)\ne0$ implies $\|u_n(s)\|_p\le4R$, $ \phi_n(s)^2\|u_n(s)\|_p^p \le (4R)^p .
$ Weighted Young's inequality therefore yields
\begin{equation*}
\|f_n(s)\|_{q_f}^p \le C_R \left( \|u_n(s)\|_p^p + \|u_n(s)\|_{3p}^p \right) \ \dd s\text{-a.e.}
\end{equation*}

For the taming term choose
\begin{equation}
\label{eq:5.qh-choice}
\max \left\{ \frac{3p}{2p+1}, \frac p3 \right\} <q_h<\frac{3p}{7}.
\end{equation}
Set $ \theta_h := \frac{p/q_h-1}{2}, \ r_h:=3(1-\theta_h).
$ Then $ \frac23<\theta_h<1, \ 0<r_h<1, $ and $ \frac1{3q_h} = \frac{\theta_h}{p} + \frac{1-\theta_h}{3p}.
$ The growth of $g_N$, interpolation, and the multiplier bounds give
\begin{equation*}
\|h_n(s)\|_{q_h}^p\le C_{N,\nu} \phi_n(s)^{2p} \|P_nu_n(s)\|_{3q_h}^{3p} \le C_{N,\nu} \phi_n(s)^{2p} \|u_n(s)\|_p^{3p\theta_h} \|u_n(s)\|_{3p}^{pr_h} \le C_{R,N,\nu} \|u_n(s)\|_{3p}^{pr_h},
\end{equation*}
where the last estimate again uses the support of $\phi$.

For the Wiener coefficient fix $ q_\sigma = \frac{3p}{2}-\varepsilon $ with $\varepsilon>0$ sufficiently small, and let $\theta_\sigma\in(\frac12,1)$ satisfy $ \frac1{q_\sigma} = \frac{\theta_\sigma}{p} + \frac{1-\theta_\sigma}{3p}.
$ Set $ \rho := 2p\theta_\sigma-p \in(0,p).
$ The ideal property of $\gamma$-radonifying operators, the multiplier bounds, and \eqref{eq:2_WienerH} imply
\begin{equation*}
\begin{aligned}
\|g_n(s)\|_{\mathbb L^p}^p &\le C\phi_n(s)^{2p} \left( 1+ \|P_nu_n(s)\|_{q_\sigma}^{2p} \right)\le C\phi_n(s)^2 \left[ 1+ \|u_n(s)\|_p^{p+\rho} \|u_n(s)\|_{3p}^{p-\rho} \right] \\
&\le C_R \left( 1+ \|u_n(s)\|_{3p}^{p-\rho} \right).
\end{aligned}
\end{equation*}

Finally, by the $L^p$-boundedness of $P_n\mathcal P$ and \eqref{eq:4.G1}, for $r\in\{2,p\}$,
\begin{equation}
\label{eq:5.H-growth}
\int_Z \|H_n(s,z)\|_p^r\,\mu(\dd z) \le C \left( 1+\|u_n(s-)\|_p^r \right),
\end{equation}
with $C$ independent of $n$ and $R$.

Let $ \theta_n(t):=t\wedge\tau_M^n .
$ All Lebesgue-time estimates below are understood modulo the single terminal time $\theta_n(t)$; hence \eqref{eq:5.preexit-control} gives

\begin{equation}
\label{eq:5.f-integrated}
\begin{aligned}
\E \int_0^{\theta_n(t)}\|f_n(s)\|_{q_f}^p\,\dd s \le C_{R,M,T}K^p,\,& \E \int_0^{\theta_n(t)}\|h_n(s)\|_{q_h}^p\,\dd s \le C_{R,M,T,N,\nu}K^p, \\
\ \E \int_0^{\theta_n(t)}\|g_n(s)\|_{\mathbb L^p}^p\,\dd s &\le C_{R,M,T}K^p.
\end{aligned}
\end{equation}

Indeed, $$ \int_0^{\theta_n(t)} \|u_n(s)\|_{3p}^{pr_h}\dd s \le T^{1-r_h} \left( \int_0^{\theta_n(t)} \|u_n(s)\|_{3p}^p\dd s \right)^{r_h}, \int_0^{\theta_n(t)} \|u_n(s)\|_{3p}^{p-\rho}\dd s \le T^{\rho/p} \left( \int_0^{\theta_n(t)} \|u_n(s)\|_{3p}^p\dd s \right)^{1-\rho/p}.
$$ Since $K\ge1$, $r_h<1$, and $\rho>0$, the resulting lower powers of $K^p$ are bounded by $K^p$.

For the jump coefficient, the predictable bound in \eqref{eq:5.preexit-control} and \eqref{eq:5.H-growth} yield
\begin{equation}
\begin{aligned}
\label{eq:5.H-integrated}
\E \int_0^{\theta_n(t)} \int_Z \|H_n(s,z)\|_p^p \,\mu(\dd z)\,\dd s &\le C_{M,T}K^p, \\
\E \left( \int_0^{\theta_n(t)} \int_Z \|H_n(s,z)\|_p^2 \,\mu(\dd z)\,\dd s \right)^{p/2} &\le C T^{p/2} \E \left[ 1+ \sup_{0\le s\le\theta_n(t)} \|u_n(s-)\|_p^p \right] \le C_{M,T}K^p .
\end{aligned}
\end{equation}

Apply the localized version of \eqref{ineq:3_eng-est} from \cref{Thm:heat-exis}, with the $\nu\Delta$-variant used in \cref{sec4}, to the stopped equation.
The stopping indicator $ \mathds 1_{\{s\le\theta_n(T)\}} $ is predictable.
For the Poisson term this retains the possible jump at $\theta_n(T)$, while its coefficient is evaluated at $u_n(\theta_n(T)-)$ and is therefore controlled by \eqref{eq:5.preexit-control}.
Using \eqref{eq:5.f-est} -- \eqref{eq:5.H-integrated} and $ \E\|u_n(0)\|_p^p\le CK^p, $ we obtain \eqref{eq:5_eng_bnd} uniformly in $n$.

We next establish \eqref{eq:5_prob_bnd}.
The strict upper bounds in \eqref{eq:5.qf-choice} and \eqref{eq:5.qh-choice} imply positive small-time exponents.
From \eqref{eq:5.f-est} and \eqref{eq:5.preexit-control},
\begin{equation}
\label{eq:5.f-smalltime}
\int_0^{\theta_n(t)} \|f_n(s)\|_{q_f}^p\,\dd s \le C M^2K^{2p}t^{\alpha_f}.
\end{equation}
Indeed, $ \int_0^{\theta_n(t)} \|u_n(s)\|_p^{p(1+\alpha_f)} \|u_n(s)\|_{3p}^{p(1-\alpha_f)} \,\dd s \le (MK^p)^{1+\alpha_f} t^{\alpha_f} \left( \int_0^{\theta_n(t)} \|u_n(s)\|_{3p}^p\,\dd s \right)^{1-\alpha_f}.$

Set $ \alpha_h:=1-r_h>0.
$ Then
\begin{equation}
\int_0^{\theta_n(t)} \|h_n(s)\|_{q_h}^p\,\dd s \le C_{N,\nu} M^3K^{3p}t^{\alpha_h},
\end{equation}
because $ 3\theta_h+r_h=3.
$ Likewise,
\begin{equation}
\int_0^{\theta_n(t)} \|g_n(s)\|_{\mathbb L^p}^p\,\dd s \le Ct + CM^2K^{2p}t^{\rho/p}.
\end{equation}
Furthermore,
\begin{equation}
\label{eq:5.H2-smalltime}
\begin{aligned}
\int_0^{\theta_n(t)} \int_Z \|H_n(s,z)\|_p^p \,\mu(\dd z)\,\dd s &\le C_MK^p t, \\
\int_0^{\theta_n(t)} \int_Z \|H_n(s,z)\|_p^2 \,\mu(\dd z)\,\dd s &\le Ct \left[ 1+(MK^p)^{2/p} \right].
\end{aligned}
\end{equation}

Let $ J_p(v):=|v|^{p-2}v $ and $ \mathcal R_p(v,h) := \|v+h\|_p^p - \|v\|_p^p - p\langle J_p(v),h\rangle .
$ For $p>2$, $$ 0\le \mathcal R_p(v,h) \le C_p \left( \|v\|_p^{p-2}\|h\|_p^2 + \|h\|_p^p \right).
$$

The localized It\^o estimate established in the proof of \cref{Thm:heat-exis}, applied to \eqref{eq:5.heat-form} up to $\theta_n(t)$, gives
\begin{equation}
\label{eq:5.pathwise-energy}
\mathcal E_n(t) \le C_0\|u_n(0)\|_p^p + \mathcal R_n(t)+ C\sup_{0\le s\le\theta_n(t)} |M_n^W(s)| + C\sup_{0\le s\le\theta_n(t)} |M_n^{J,1}(s)| + C\sup_{0\le s\le\theta_n(t)} |M_n^{J,2}(s)|,
\end{equation}
where $\mathcal R_n(t)\ge0$ contains the finite-variation contributions,
\begin{align*}
\mathcal E_n(t) := \sup_{0\le s\le\theta_n(t)} &\|u_n(s)\|_p^p + c_{p,\nu} \int_0^{\theta_n(t)} \left\| \nabla|u_n(s)|^{p/2} \right\|_2^2\,\dd s , \\
M_n^W(s):= p \int_0^{s\wedge\tau_M^n} \left\langle g_n(r)^*J_p(u_n(r)), \,\dd W_r \right\rangle_{\mathcal U},& \quad M_n^{J,1}(s):= p \int_0^{s\wedge\tau_M^n} \int_Z \left\langle J_p(u_n(r-)), H_n(r,z) \right\rangle \,\widetilde N(\dd r,\dd z), \\
M_n^{J,2}&(s) := \int_0^{s\wedge\tau_M^n} \int_Z \mathcal R_p \bigl( u_n(r-),H_n(r,z) \bigr) \,\widetilde N(\dd r,\dd z).
\end{align*}
The estimates \eqref{eq:5.f-smalltime}--\eqref{eq:5.H2-smalltime} and the deterministic part of the localized It\^o estimate imply
\begin{equation}
\label{eq:5.R-small}
\E\mathcal R_n(t) \le C_{M,K,R} \left( t^{\alpha_f} + t^{\alpha_h} + t^{\rho/p} + t \right),
\end{equation}
uniformly in $n$.

For the Wiener martingale, BDG and H\"older's inequality give
\begin{equation}
\label{eq:5.MW-small}
\begin{aligned}
\E \sup_{0\le s\le\theta_n(t)} |M_n^W(s)| &\le C t^{\frac12-\frac1p} \E \Bigg[ \sup_{0\le r<\theta_n(t)} \|u_n(r)\|_p^{p-1}\left( \int_0^{\theta_n(t)} \|g_n(r)\|_{\mathbb L^p}^p\,\dd r \right)^{1/p} \Bigg] \\
&\le C_{M,K} \left( t^{1/2} + t^{\frac12-\frac1p+\frac{\rho}{p^2}} \right) \le C_{M,K}t^\beta
\end{aligned}
\end{equation}
for some $\beta>0$ depending only on $p$ and the choice of $q_\sigma$.

For the linear jump martingale, Davis' inequality, \eqref{eq:5.preexit-control}, and the compensator identity yield
\begin{equation}
\label{eq:5.MJ1-small}
\begin{aligned}
&\E \sup_{0\le s\le\theta_n(t)} |M_n^{J,1}(s)| \le C \E \left( \int_0^{\theta_n(t)} \int_Z \|u_n(r-)\|_p^{2p-2} \|H_n(r,z)\|_p^2 \,N(\dd r,\dd z) \right)^{1/2} \\
&\le C(MK^p)^{(p-1)/p} \left[ \E \int_0^{\theta_n(t)} \int_Z \|H_n(r,z)\|_p^2 \,\mu(\dd z)\,\dd r \right]^{1/2} \le C_{M,K}K^p\,t^{1/2}.
\end{aligned}
\end{equation}

For the Taylor-remainder martingale, no higher jump moment is required.
Since $\mathcal R_p\ge0$, $$
\begin{aligned}
\sup_{0\le s\le\theta_n(t)} |M_n^{J,2}(s)| \le & \int_0^{\theta_n(t)} \int_Z \mathcal R_p(u_n(r-),H_n(r,z)) \,N(\dd r,\dd z) + \int_0^{\theta_n(t)} \int_Z \mathcal R_p(u_n(r-),H_n(r,z)) \,\mu(\dd z)\,\dd r .
\end{aligned}
$$ Hence
\begin{equation}
\label{eq:5.MJ2-small}
\begin{aligned}
\E \sup_{0\le s\le\theta_n(t)} |M_n^{J,2}(s)| &\le 2 \E \int_0^{\theta_n(t)} \int_Z \mathcal R_p(u_n(r-),H_n(r,z)) \,\mu(\dd z)\,\dd r \\
&\le C \E \int_0^{\theta_n(t)} \left( 1+\|u_n(r-)\|_p^p \right)\,\dd r \le C_MK^p t .
\end{aligned}
\end{equation}

By the Sobolev inequality on $\R^3$, $\|u_n(s)\|_{3p}^p = \bigl\||u_n(s)|^{p/2}\bigr\|_6^2 \le C_S \left\| \nabla|u_n(s)|^{p/2} \right\|_2^2 .$

Since every c\`adl\`ag path satisfies $ \sup_{0\le s\le t} \|u_n(s-)\|_p \le \sup_{0\le s\le t} \|u_n(s)\|_p , $ there exists $C_*=C_*(p,\nu)>0$ such that
\begin{equation}
\label{eq:5.stop-vs-energy}
\sup_{0\le s\le\theta_n(t)} \|u_n(s)\|_p^p + \int_0^{\theta_n(t)} \|u_n(s)\|_{3p}^p\,\dd s \le C_*\mathcal E_n(t).
\end{equation}

Increase $M_0$, if necessary, so that
\begin{equation}
\label{eq:5.M0-margin}
C_*C_0 \sup_{n\ge1} \|u_n(0)\|_p^p \le \frac14M_0K^p .
\end{equation}
Define $ \mathcal S_n(t) := \left\{ \sup_{0\le s\le\theta_n(t)} \|u_n(s)\|_p^p + \int_0^{\theta_n(t)} \|u_n(s)\|_{3p}^p\,\dd s \ge M_0K^p \right\}.
$ By \eqref{eq:5.pathwise-energy}, \eqref{eq:5.stop-vs-energy}, and \eqref{eq:5.M0-margin}, there exists $c>0$, independent of $n$, such that
\begin{align*}
\mathcal S_n(t) \subset& \left\{ \mathcal R_n(t)\ge cM_0K^p \right\} \cup \left\{ \sup_{s\le\theta_n(t)} |M_n^W(s)| \ge cM_0K^p \right\} \\
&\cup \left\{ \sup_{s\le\theta_n(t)} |M_n^{J,1}(s)| \ge cM_0K^p \right\} \cup \left\{ \sup_{s\le\theta_n(t)} |M_n^{J,2}(s)| \ge cM_0K^p \right\}.
\end{align*}
Therefore, by Markov's inequality and \eqref{eq:5.R-small}, \eqref{eq:5.MW-small}, \eqref{eq:5.MJ1-small}, and \eqref{eq:5.MJ2-small}, $$ \sup_{n\ge1} \P(\mathcal S_n(t)) \le C_{M_0,M,K,R} \left( t^{\alpha_f} + t^{\alpha_h} + t^{\rho/p} + t^\beta + t^{1/2} + t \right).
$$ Every exponent on the right-hand side is strictly positive.
So, $ \lim_{t\downarrow0} \sup_{n\ge1} \P(\mathcal S_n(t)) =0, $ which proves \eqref{eq:5_prob_bnd}.
\end{proof}
\begin{lemma}
\label{lem:5.2} Let $p>3$, $K\ge1$, and suppose that $ \|u_0\|_p\le K \ \P\text{-a.s.} $ Let $R\ge1$ be the fixed scalar cutoff level of the present section, and let $u^{(n)}$ be the solution of \eqref{eq:5.1_truncated}.
Then there exist $ t=t(R,M,K)>0 $ and a strictly increasing sequence $\{k(n)\}_{n\ge1}\subset\N$ with $k(n)\uparrow\infty$ such that
\begin{equation}
\label{eq:5.Cauchy_limit}
\lim_{m\to\infty}\sup_{n>m} \E \Bigg[ \sup_{0\le s\le\tau_{n,m}} \|u^{(n)}(s)-u^{(m)}(s)\|_p^p + \int_0^{\tau_{n,m}} \|u^{(n)}(s)-u^{(m)}(s)\|_{3p}^p \,\dd s \Bigg] =0,
\end{equation}
where $ \tau_{n,m} := \tau_M^n\wedge\tau_M^m\wedge t .
$ The dependence of $t$ on the fixed structural parameters $p,\nu,N$ and on the constants in the noise assumptions is suppressed.
\end{lemma}

\begin{proof}

For $n>m$, set $ u_n:=u^{(n)}, \ u_m:=u^{(m)}, \ w:=u_n-u_m, $ and $ P_n:={\mathrm P}_{\le k(n)}, \ P_m:={\mathrm P}_{\le k(m)}, \ P_{n,m}:=P_n-P_m .
$ We write $ \phi_n(s):=\phi(\|u_n(s)\|_p), \ \phi_m(s):=\phi(\|u_m(s)\|_p), $ and $ a_j := P_j\mathcal P \left( \chi\left(\frac{\cdot}{j}\right)u_0 \right).
$

With
\begin{align*}
f_j &:= \phi_j^2 P_j\mathcal P^{(1)} \bigl( u_j\otimes P_ju_j \bigr),\ h_j:= \phi_j^2 P_j\mathcal P \mathcal T(P_ju_j), \text{where} \ \mathcal T(v):=g_N(|v|^2)v, \\
g_j &:= \phi_j^2 P_j\mathcal P\sigma(P_ju_j),\quad H_j(s,z) := \phi_j(s-)^2 P_j\mathcal P G(P_ju_j(s-),z),
\end{align*}
the difference $w$ satisfies
\begin{equation}
\label{eq:5.diff-equation}
\begin{aligned}
\dd w =& \Big[ \nu\Delta w - \nabla\cdot(f_n-f_m) - (h_n-h_m) \Big]\,\dd s + (g_n-g_m)\,\dd W_s + \int_Z (H_n-H_m)(s,z) \,\widetilde N(\dd s,\dd z), \\
\nabla\cdot w =&0, \qquad w(0)=a_n-a_m .
\end{aligned}
\end{equation}

All coefficient estimates below are understood on $[0,\tau_{n,m}]$.
More precisely, when \cref{Thm:heat-exis} is applied, every coefficient is multiplied by $ \mathds 1_{\{s\le\tau_{n,m}\}}.
$ This process is adapted and left-continuous, hence predictable.
The possible post-jump overshoot of $u_n(\tau_M^n)$ or $u_m(\tau_M^m)$ is therefore immaterial for the $\dd s$-integrated coefficients, while the Poisson coefficients are evaluated at the predictable left limits.
In particular,
\begin{equation}
\label{eq:5.2-preexit}
\begin{aligned}
\sup_{s<\tau_{n,m}} \bigl( \|u_n(s)\|_p^p+\|u_m(s)\|_p^p \bigr) \le 2MK^p, & \quad \sup_{s\le\tau_{n,m}} \bigl( \|u_n(s-)\|_p^p+\|u_m(s-)\|_p^p \bigr)\le 2MK^p, \\
\int_0^{\tau_{n,m}} \bigl( \|u_n(s)\|_{3p}^p+\|u_m(s)\|_{3p}^p \bigr)\,\dd s &\le 2MK^p.
\end{aligned}
\end{equation}

Choose
\begin{equation}
\label{eq:5.2-qf}
\max \left\{ \frac p2,\frac{3p}{p+1} \right\} <q_f<\frac{3p}{4}
\end{equation}
\begin{equation}
\label{eq:5.2-qh-revised}
\max \left\{ \frac p3,\frac{3p}{2p+1} \right\} <q_h<\frac{3p}{7}.
\end{equation}
All multiplier bounds used below are uniform in $m,n$.

\noindent \emph{A common estimate for inner spectral defects.} We first derive the interpolation strategy used repeatedly below.
Suppose that $c\ge0$ and $q\in(1,\infty)$ satisfy
\begin{equation}
\label{eq:5.2-spectral-range}
\frac1{3p} < \frac1q-\frac cp < \frac1p .
\end{equation}
Let $ d_{n,m}:=P_{n,m}u_m .
$ Choose $\delta>0$ sufficiently small and define $\ell_\delta$ by $ \frac1q = \frac cp+\frac{\delta}{2} +\frac1{\ell_\delta}.
$ Then $ r_\delta:=(1-\delta)\ell_\delta \in(p,3p) $ for all sufficiently small $\delta$.
Let $\beta_\delta\in(0,1)$ be determined by $ \frac1{r_\delta} = \frac{\beta_\delta}{p} + \frac{1-\beta_\delta}{3p}.
$ Lemma~\ref{lem:2.4}, the uniform multiplier bounds and interpolation give
\begin{equation*}
\begin{aligned}
\|d_{n,m}\|_2 &\le \frac{C}{k(n)\wedge k(m)} \|\nabla u_m\|_2, \\
\|d_{n,m}\|_{r_\delta} &\le C \|u_m\|_p^{\beta_\delta} \|u_m\|_{3p}^{1-\beta_\delta}.
\end{aligned}
\end{equation*}
now whenever the remaining factor carries the scalar cutoff and has $L^{p/c}$-norm bounded by $C_R$,
\begin{equation}
\label{eq:5.2-inner-template}
\begin{aligned}
\|\mathcal S\,d_{n,m}\|_q^p \le \frac{C_R} {(k(n)\wedge k(m))^{\delta p}} \|\nabla u_m\|_2^{\delta p} \|u_m\|_{3p}^{b_\delta},
\end{aligned}
\end{equation}
where $\mathcal S$ denotes the generic spatial multiplicative factor $ b_\delta := p(1-\delta)(1-\beta_\delta).
$ At $\delta=0$, condition \eqref{eq:5.2-spectral-range} implies $\beta_0>0$ and hence $b_0<p$.
We may therefore choose $\delta>0$ so small that $\delta p<2,\ \frac{2b_\delta}{2-\delta p}<p $ Young's inequality then yields
\begin{equation}
\label{eq:5.2-inner-template-final}
\|\mathcal S\,d_{n,m}\|_q^p \le \frac{C_R} {(k(n)\wedge k(m))^{\delta p}} \left( 1+\|\nabla u_m\|_2^2+\|u_m\|_{3p}^p \right).
\end{equation}

We shall use \eqref{eq:5.2-inner-template-final} with $ (c,q) = (1,q_f),\ (2,q_h),\ \left(\frac12,p\right),\ (\alpha,p).
$ The ranges \eqref{eq:5.2-qf}--\eqref{eq:5.2-qh-revised} and $\alpha<2/3$ ensure \eqref{eq:5.2-spectral-range} in all four cases.

\medskip \noindent \emph{Convective term.} We use the exact decomposition $f_n-f_m=\sum_{j=1}^6F_j$, where
\begin{align*}
F_1 &= \phi_n(\phi_n-\phi_m) P_n\mathcal P^{(1)} (u_n\otimes P_nu_n), \ F_2 = \phi_n\phi_m P_n\mathcal P^{(1)} (u_n\otimes P_nw), \\
F_3 &= \phi_n\phi_m P_n\mathcal P^{(1)} (u_n\otimes P_{n,m}u_m),\ F_4 = \phi_n\phi_m P_n\mathcal P^{(1)} (w\otimes P_mu_m), \\
F_5 &= \phi_n\phi_m P_{n,m}\mathcal P^{(1)} (u_m\otimes P_mu_m), \ F_6 = (\phi_n-\phi_m)\phi_m P_m\mathcal P^{(1)} (u_m\otimes P_mu_m).
\end{align*}

The terms $F_1,F_2,F_4,F_6$ are estimated exactly by the interpolation argument in [Section $5$, of \cite{MR4908978}].
In the present fixed-$R$ setting, for every $\varepsilon>0$,
\begin{equation}
\label{eq:5.2-f-local}
\E \int_0^{\tau_{n,m}} \sum_{j\in\{1,2,4,6\}} \|F_j(s)\|_{q_f}^p\,\dd s \le \varepsilon \E \int_0^{\tau_{n,m}} \|w(s)\|_{3p}^p\,\dd s + \Lambda_{f,\varepsilon}(t) \E \sup_{s\le\tau_{n,m}} \|w(s)\|_p^p,
\end{equation}
where $ \Lambda_{f,\varepsilon}(t) \longrightarrow0 \ \text{as }t\downarrow0 .
$ For $F_3$, apply \eqref{eq:5.2-inner-template-final} with $c=1$ and $q=q_f$.
Thus, for some $\rho_{f,3}>0$,
\begin{equation}
\label{eq:5.2-F3}
\E \int_0^{\tau_{n,m}} \|F_3(s)\|_{q_f}^p\,\dd s \le \frac{C_R} {(k(n)\wedge k(m))^{\rho_{f,3}}} \mathcal R_m ,
\end{equation}
where, from now on,
\begin{equation}
\label{eq:5.2-Rm-revised}
\mathcal R_m := 1+ \E \int_0^{\tau_{n,m}} \left( 1+\|\nabla u_m(s)\|_2^2 +\|u_m(s)\|_{3p}^p \right)\,\dd s .
\end{equation}

For $F_5$, we retain the outer-defect interpolation used in [Section $5$, of \cite{MR4908978}].
Put $ \kappa:=\frac{p-2}{p+2}, \qquad \frac1{1+\kappa} = \frac12+\frac1p .
$ Choose $ \bar\ell\in(q_f,3p/4) $ sufficiently close to $q_f$ and $\vartheta\in(0,1)$ so that $ \frac1{q_f} = \frac{\vartheta}{1+\kappa} + \frac{1-\vartheta}{\bar\ell}, \ \vartheta p<2.
$ Using Lemma~\ref{lem:2.4} in $L^{1+\kappa}$, the uniform $L^{\bar\ell}$-boundedness of the multipliers, and $ \|\nabla(u_m\otimes P_mu_m)\|_{1+\kappa} \le C \|\nabla u_m\|_2\|u_m\|_p, $ we obtain
\begin{equation}
\label{eq:5.2-F5}
\begin{aligned}
\|F_5\|_{q_f}^p \le \frac{C_R} {(k(n)\wedge k(m))^{\vartheta p}} \|\nabla u_m\|_2^{\vartheta p} \|u_m\|_{2\bar\ell}^{2(1-\vartheta)p}.
\end{aligned}
\end{equation}
Let $\bar\delta\in(0,1)$ satisfy $ \frac1{2\bar\ell} = \frac{\bar\delta}{p} + \frac{1-\bar\delta}{3p}.
$ Taking $\bar\ell$ sufficiently close to $q_f$ makes $\vartheta$ sufficiently small so that, in addition, $\frac{ 4(1-\bar\delta)(1-\vartheta)p }{ 2-\vartheta p } <p$ .
Young's inequality in \eqref{eq:5.2-F5} therefore gives
\begin{equation}
\label{eq:5.2-F5-revised}
\E \int_0^{\tau_{n,m}} \|F_5(s)\|_{q_f}^p\,\dd s \le \frac{C_R} {(k(n)\wedge k(m))^{\vartheta p}} \mathcal R_m .
\end{equation}

Combining \eqref{eq:5.2-f-local}, \eqref{eq:5.2-F3}, and \eqref{eq:5.2-F5-revised}, there is $\rho_f>0$ such that
\begin{equation}
\label{eq:5.2-f-summary}
\E \int_0^{\tau_{n,m}} \|f_n-f_m\|_{q_f}^p\,\dd s \le \varepsilon \E \int_0^{\tau_{n,m}} \|w(s)\|_{3p}^p\,\dd s + \Lambda_{f,\varepsilon}(t) \E \sup_{s\le\tau_{n,m}}\|w(s)\|_p^p + \frac{C_R} {(k(n)\wedge k(m))^{\rho_f}} \mathcal R_m .
\end{equation}

\medskip \noindent \emph{Taming term.} The map $ \mathcal T(v)=g_N(|v|^2)v $ satisfies
\begin{equation}
\label{eq:5.2-taming-Lip-revised}
\begin{aligned}
|\mathcal T(a)| \le C_{N,\nu}|a|^3,\ & |\mathcal T(a)-\mathcal T(b)| \le C_{N,\nu} (|a|^2+|b|^2)|a-b|, \\
|\nabla\mathcal T(v)| &\le C_{N,\nu}|v|^2|\nabla v|.
\end{aligned}
\end{equation}
Decompose $ h_n-h_m=\sum_{j=1}^5T_j, $ where
\begin{align*}
T_1 &= \phi_n(\phi_n-\phi_m) P_n\mathcal P\mathcal T(P_nu_n), \quad T_2 = \phi_n\phi_m P_n\mathcal P \left[ \mathcal T(P_nu_n)-\mathcal T(P_nu_m) \right], \\
T_3 &= \phi_n\phi_m P_n\mathcal P \left[ \mathcal T(P_nu_m)-\mathcal T(P_mu_m) \right], \ T_4 = \phi_n\phi_m P_{n,m}\mathcal P \mathcal T(P_mu_m),\ T_5 = (\phi_n-\phi_m)\phi_m P_m\mathcal P\mathcal T(P_mu_m).
\end{align*}

By \eqref{eq:5.2-taming-Lip-revised}, \eqref{eq:5.2-qh-revised}, interpolation between $L^p$ and $L^{3p}$, and the cutoff bounds,
\begin{equation}
\E \int_0^{\tau_{n,m}} \sum_{j\in\{1,2,5\}} \|T_j(s)\|_{q_h}^p\,\dd s \le \varepsilon \E \int_0^{\tau_{n,m}} \|w(s)\|_{3p}^p\,\dd s+ \Lambda_{h,\varepsilon}(t) \E \sup_{s\le\tau_{n,m}} \|w(s)\|_p^p,
\end{equation}
with $ \Lambda_{h,\varepsilon}(t)\longrightarrow0 \ (t\downarrow0).
$ The strict inequality $q_h<3p/7$ is precisely the subcriticality needed here.

For $T_3$, \eqref{eq:5.2-taming-Lip-revised} reduces the estimate to terms of the form $ \phi_n\phi_m \bigl( |P_nu_m|^2+|P_mu_m|^2 \bigr) P_{n,m}u_m .
$ Applying \eqref{eq:5.2-inner-template-final} with $c=2$ and $q=q_h$ gives some $\rho_{h,3}>0$ such that
\begin{equation*}
\E \int_0^{\tau_{n,m}} \|T_3(s)\|_{q_h}^p\,\dd s \le \frac{C_{R,N,\nu}} {(k(n)\wedge k(m))^{\rho_{h,3}}} \mathcal R_m .
\end{equation*}

It remains to estimate $T_4$.
Choose $r_*\in[1,q_h)$ and $a_1,a_2\in[p,3p]$ such that $\frac1{r_*} = \frac12+\frac1{a_1}+\frac1{a_2}.$ Such a choice exists because $q_h>1$ and $q_h>p/3$.
Next choose $ \ell_h\in(q_h,3p/7) $ sufficiently close to $q_h$ and $\vartheta_h\in(0,1)$ satisfying $ \frac1{q_h} = \frac{\vartheta_h}{r_*} + \frac{1-\vartheta_h}{\ell_h}.
$ Then $\vartheta_h\downarrow0$ as $\ell_h\downarrow q_h$.
From Lemma~\ref{lem:2.4}, \eqref{eq:5.2-taming-Lip-revised}, and interpolation,
\begin{align}
\|P_{n,m}\mathcal P\mathcal T(P_mu_m)\|_{r_*} &\le \frac{C} {k(n)\wedge k(m)} \|\nabla u_m\|_2 \|u_m\|_{a_1}\|u_m\|_{a_2},
\label{eq:5.2-T4-low}
\\
\|P_{n,m}\mathcal P\mathcal T(P_mu_m)\|_{\ell_h} &\le C \|u_m\|_{3\ell_h}^3.
\label{eq:5.2-T4-high}
\end{align}
Interpolating \eqref{eq:5.2-T4-low} and \eqref{eq:5.2-T4-high} gives
\begin{equation}
\label{eq:5.2-T4}
\|T_4\|_{q_h}^p \le \frac{C_R} {(k(n)\wedge k(m))^{\vartheta_hp}} \|\nabla u_m\|_2^{\vartheta_hp} \|u_m\|_{3p}^{b_h},
\end{equation}
for an exponent $b_h=b_h(\ell_h,\vartheta_h)$ continuous in $\vartheta_h$.
At $\vartheta_h=0$, interpolation of $L^{3\ell_h}$ between $L^p$ and $L^{3p}$ gives $ b_h(\ell_h,0)<p \Longleftrightarrow \ell_h<\frac{3p}{7}.
$ By taking $\ell_h$ sufficiently close to $q_h$, we may ensure $ \vartheta_hp<2, \ \frac{2b_h}{2-\vartheta_hp}<p .
$ Young's inequality in \eqref{eq:5.2-T4} therefore gives
\begin{equation*}
\E \int_0^{\tau_{n,m}} \|T_4(s)\|_{q_h}^p\,\dd s \le \frac{C_{R,N,\nu}} {(k(n)\wedge k(m))^{\vartheta_hp}} \mathcal R_m .
\end{equation*}

Thus, for some $\rho_h>0$,
\begin{equation}
\label{eq:5.2-h-summary}
\E \int_0^{\tau_{n,m}} \|h_n-h_m\|_{q_h}^p\,\dd s \le \varepsilon \E \int_0^{\tau_{n,m}} \|w(s)\|_{3p}^p\,\dd s + \Lambda_{h,\varepsilon}(t) \E \sup_{s\le\tau_{n,m}}\|w(s)\|_p^p + \frac{C_{R,N,\nu}} {(k(n)\wedge k(m))^{\rho_h}} \mathcal R_m .
\end{equation}

\medskip \noindent \emph{Wiener coefficient.} The exact decomposition is $ g_n-g_m=\sum_{j=1}^5S_j, $ where
\begin{align*}
S_1 &= \phi_n(\phi_n-\phi_m) P_n\mathcal P\sigma(P_nu_n), S_2 = \phi_n\phi_m P_n\mathcal P \bigl[ \sigma(P_nu_n)-\sigma(P_nu_m) \bigr], S_3 = \phi_n\phi_m P_n\mathcal P \bigl[ \sigma(P_nu_m)-\sigma(P_mu_m) \bigr], \\
S_4 &= \phi_n\phi_m P_{n,m}\mathcal P\sigma(P_mu_m), \quad S_5 = (\phi_n-\phi_m)\phi_m P_m\mathcal P\sigma(P_mu_m).
\end{align*}

The cutoff, growth and weighted Lipschitz estimates in \eqref{eq:2_WienerH} imply
\begin{equation*}
\E \int_0^{\tau_{n,m}} \sum_{j\in\{1,2,5\}} \|S_j(s)\|_{\mathbb L^p}^p\,\dd s \le \varepsilon \E \int_0^{\tau_{n,m}} \|w(s)\|_{3p}^p\,\dd s + \Lambda_{g,\varepsilon}(t) \E \sup_{s\le\tau_{n,m}} \|w(s)\|_p^p,
\end{equation*}
with $ \Lambda_{g,\varepsilon}(t)\longrightarrow0 \ (t\downarrow0).
$ For $S_3$, the weighted Lipschitz estimate reduces the problem to $$ \left\| \bigl( |P_nu_m|+|P_mu_m| \bigr)^{1/2} P_{n,m}u_m \right\|_p .
$$ Hence \eqref{eq:5.2-inner-template-final} with $c=\frac12$, $q=p$ gives
\begin{equation*}
\E \int_0^{\tau_{n,m}} \|S_3(s)\|_{\mathbb L^p}^p\,\dd s \le \frac{C_R} {(k(n)\wedge k(m))^{\rho_{g,3}}} \mathcal R_m
\end{equation*}
for some $\rho_{g,3}>0$.
For $S_4$, Lemma~\ref{lem:2.4} in the $\gamma$-radonifying space and the gradient bound in \eqref{eq:2_WienerH} give $ \|S_4(s)\|_{\mathbb L^p} \le \frac{C} {k(n)\wedge k(m)} \phi_n\phi_m \left( 1+\|P_mu_m(s)\|_{3p/2}^2 \right).
$ Since $ \|u_m\|_{3p/2}^p \le \|u_m\|_p^{p/2} \|u_m\|_{3p}^{p/2}, $ the cutoff bound and Young's inequality yield
\begin{equation*}
\E \int_0^{\tau_{n,m}} \|S_4(s)\|_{\mathbb L^p}^p\,\dd s \le \frac{C_R} {(k(n)\wedge k(m))^p} \mathcal R_m .
\end{equation*}

Consequently, for some $\rho_g>0$,
\begin{equation}
\label{eq:5.2-g-summary}
\E \int_0^{\tau_{n,m}} \|g_n-g_m\|_{\mathbb L^p}^p\,\dd s \le\varepsilon \E \int_0^{\tau_{n,m}} \|w(s)\|_{3p}^p\,\dd s + \Lambda_{g,\varepsilon}(t) \E \sup_{s\le\tau_{n,m}}\|w(s)\|_p^p + \frac{C_R} {(k(n)\wedge k(m))^{\rho_g}} \mathcal R_m .
\end{equation}
\noindent \emph{Compensated Poisson coefficient.} Put $ \mathcal H_{n,m}(s,z):=H_n(s,z)-H_m(s,z).
$ All cutoff factors and states in what follows are evaluated at $s-$.
We have $ \mathcal H_{n,m}=\sum_{j=1}^5J_j, $ where
\begin{align*}
J_1 &= \phi_n(\phi_n-\phi_m) P_n\mathcal P G(P_nu_n,z),\qquad J_2 = \phi_n\phi_m P_n\mathcal P \bigl[ G(P_nu_n,z)-G(P_nu_m,z) \bigr], \\
J_3 &= \phi_n\phi_m P_n\mathcal P \bigl[ G(P_nu_m,z)-G(P_mu_m,z) \bigr], J_4 = \phi_n\phi_m P_{n,m}\mathcal P G(P_mu_m,z), J_5 = (\phi_n-\phi_m)\phi_m P_m\mathcal P G(P_mu_m,z).
\end{align*}

For a predictable $E$-valued coefficient $J$, set
\begin{equation*}
\mathcal J_{n,m}(J) := \E \int_0^{\tau_{n,m}}\int_Z \|J(s,z)\|_p^p \,\mu(\dd z)\,\dd s+ \E \left( \int_0^{\tau_{n,m}}\int_Z \|J(s,z)\|_p^2 \,\mu(\dd z)\,\dd s \right)^{p/2}.
\end{equation*}
For a finite sum,
\begin{equation}
\label{eq:5.2-Poisson-triangle}
\mathcal J_{n,m} \left( \sum_{j=1}^5J_j \right) \le C_p \sum_{j=1}^5 \mathcal J_{n,m}(J_j).
\end{equation}

By \eqref{eq:4.G1}, the Lipschitz property of $\phi$, and \eqref{eq:5.2-preexit},
\begin{equation}
\label{eq:5.2-J15-revised}
\mathcal J_{n,m}(J_1) + \mathcal J_{n,m}(J_5) \le C_{R} \bigl( t+t^{p/2} \bigr) \E \sup_{s\le\tau_{n,m}} \|w(s)\|_p^p .
\end{equation}

For $J_2$, let $\alpha\in[0,2/3)$ be the exponent from \eqref{eq:4.G2} and set $ \theta_G:=1-\frac{3\alpha}{2}\in(0,1].
$ Since an $L^p$-valued c\`adl\`ag path has at most countably many jumps, $ u_j(s-)=u_j(s), \ w(s-)=w(s) \ \text{for }\dd s\text{-a.e.
}s.
$ For $r\in\{2,p\}$ and $\dd s$-a.e.
$s$,
\begin{equation}
\int_Z \|J_2(s,z)\|_p^r\,\mu(\dd z)\le C \left\| \bigl( |P_nu_n(s-)|+|P_nu_m(s-)| \bigr)^\alpha P_nw(s-) \right\|_p^r \le C_R \|w(s-)\|_p^{\theta_Gr} \|w(s)\|_{3p}^{(1-\theta_G)r}.
\end{equation}
Writing $$ X_{n,m} := \sup_{s\le\tau_{n,m}}\|w(s)\|_p^p, \qquad Y_{n,m} := \int_0^{\tau_{n,m}}\|w(s)\|_{3p}^p\dd s, $$ the $r=p$ contribution satisfies $ \E \int_0^{\tau_{n,m}}\int_Z \|J_2\|_p^p\,\mu(\dd z)\,\dd s \le C_R t^{\theta_G} \E \left( X_{n,m}^{\theta_G} Y_{n,m}^{1-\theta_G} \right).
$ Thus, for every $\varepsilon>0$,
\begin{equation*}
\E \int_0^{\tau_{n,m}}\int_Z \|J_2\|_p^p\,\mu(\dd z)\,\dd s \le \varepsilon\E Y_{n,m} + C_{\varepsilon,R}t\E X_{n,m}.
\end{equation*}
For the $r=2$ contribution, since $p>2$,
\begin{equation*}
\E \left( \int_0^{\tau_{n,m}}\int_Z \|J_2\|_p^2 \,\mu(\dd z)\,\dd s \right)^{p/2} \le t^{p/2-1} \E \int_0^{\tau_{n,m}} \left( \int_Z\|J_2\|_p^2\,\mu(\dd z) \right)^{p/2} \dd s\le \varepsilon\E Y_{n,m} + C_{\varepsilon,R}t^{p/2} \E X_{n,m}.
\end{equation*}
Consequently,
\begin{equation}
\label{eq:5.2-J2}
\mathcal J_{n,m}(J_2) \le 2\varepsilon\E Y_{n,m} + C_{\varepsilon,R} (t+t^{p/2}) \E X_{n,m}.
\end{equation}

For $J_3$, Hypothesis~\eqref{eq:4.G2} and \eqref{eq:5.2-inner-template-final}, now with $c=\alpha$ and $q=p$, imply, for some $\rho_{G,3}>0$,
\begin{equation*}
\E \int_0^{\tau_{n,m}}\int_Z \|J_3(s,z)\|_p^p \,\mu(\dd z)\,\dd s \le \frac{C_R} {(k(n)\wedge k(m))^{\rho_{G,3}}} \mathcal R_m .
\end{equation*}
The $r=2$ part is controlled without any higher jump moment: by H\"older in time,
\begin{equation*}
\E \left( \int_0^{\tau_{n,m}}\int_Z \|J_3(s,z)\|_p^2 \,\mu(\dd z)\,\dd s \right)^{p/2} \le t^{p/2-1} \E \int_0^{\tau_{n,m}} \left( \int_Z \|J_3(s,z)\|_p^2 \,\mu(\dd z) \right)^{p/2} \dd s \le \frac{C_R} {(k(n)\wedge k(m))^{\rho_{G,3}}} \mathcal R_m .
\end{equation*}
Hence
\begin{equation}
\label{eq:5.2-J3-revised}
\mathcal J_{n,m}(J_3) \le \frac{C_R} {(k(n)\wedge k(m))^{\rho_{G,3}}} \mathcal R_m .
\end{equation}

For $J_4$, Lemma~\ref{lem:2.4} and \eqref{eq:4.G3} give, for $r\in\{2,p\}$,
\begin{equation*}
\int_Z \|J_4(s,z)\|_p^r\,\mu(\dd z) \le \frac{C} {(k(n)\wedge k(m))^r} \left( 1+\|P_mu_m(s-)\|_{3p/2}^r \right).
\end{equation*}
Since $ \|u_m\|_{3p/2} \le \|u_m\|_p^{1/2} \|u_m\|_{3p}^{1/2}, $ the cutoff bound, \eqref{eq:5.2-preexit}, H\"older's inequality in time, and the $dt$-a.e.
identity $u_m(s-)=u_m(s)$ yield
\begin{equation}
\label{eq:5.2-J4-revised}
\mathcal J_{n,m}(J_4) \le \frac{C_{R,M,K,t}} {(k(n)\wedge k(m))^p} \mathcal R_m .
\end{equation}

Combining \eqref{eq:5.2-J15-revised}, \eqref{eq:5.2-J2}, \eqref{eq:5.2-J3-revised}, \eqref{eq:5.2-J4-revised}, and \eqref{eq:5.2-Poisson-triangle}, we obtain, for some $\rho_G>0$,
\begin{equation}
\label{eq:5.2-G1-summary}
\mathcal J_{n,m}(\mathcal H_{n,m}) \le 2\varepsilon \E \int_0^{\tau_{n,m}} \|w(s)\|_{3p}^p\,\dd s+ \Lambda_{G,\varepsilon}(t) \E \sup_{s\le\tau_{n,m}} \|w(s)\|_p^p + \frac{C_{R,M,K,t}} {(k(n)\wedge k(m))^{\rho_G}} \mathcal R_m ,
\end{equation}
where $ \Lambda_{G,\varepsilon}(t)\longrightarrow0 \ (t\downarrow0).
$

\medskip \noindent \emph{Conclusion of the difference estimate.} Combining \eqref{eq:5.2-f-summary}, \eqref{eq:5.2-h-summary}, \eqref{eq:5.2-g-summary}, and \eqref{eq:5.2-G1-summary}, there exist $\rho>0$ and, for every $\varepsilon>0$, a deterministic function $\Lambda_\varepsilon$ satisfying $ \Lambda_\varepsilon(t)\longrightarrow0 \ (t\downarrow0) $ such that
\begin{equation*}
\begin{aligned}
& \E \int_0^{\tau_{n,m}} \Big( \|f_n-f_m\|_{q_f}^p + \|h_n-h_m\|_{q_h}^p + \|g_n-g_m\|_{\mathbb L^p}^p \Big)\,\dd s + \mathcal J_{n,m}(\mathcal H_{n,m}) \\
&\le C\varepsilon \E \int_0^{\tau_{n,m}} \|w(s)\|_{3p}^p\,\dd s\quad+ \Lambda_\varepsilon(t) \E \sup_{s\le\tau_{n,m}} \|w(s)\|_p^p + \frac{C_{R,M,K,t}} {(k(n)\wedge k(m))^\rho} \mathcal R_m .
\end{aligned}
\end{equation*}

Applying the localized estimate of \cref{Thm:heat-exis} to \eqref{eq:5.diff-equation} yields
\begin{equation}
\label{eq:5_Cauchy}
\begin{aligned}
\E \Bigg[ \sup_{0\le s\le\tau_{n,m}} \|w(s)\|_p^p + \int_0^{\tau_{n,m}} \left\| \nabla|w(s)|^{p/2} \right\|_2^2\,\dd s \Bigg] & \le C\E\|a_n-a_m\|_p^p + C\varepsilon \E \int_0^{\tau_{n,m}} \|w(s)\|_{3p}^p\,\dd s \\
&\quad+ C\Lambda_\varepsilon(t) \E \sup_{0\le s\le\tau_{n,m}} \|w(s)\|_p^p + \frac{C_{R,M,K,t}} {(k(n)\wedge k(m))^\rho} \mathcal R_m .
\end{aligned}
\end{equation}
By the Sobolev inequality on $\R^3$, $\|w(s)\|_{3p}^p = \bigl\||w(s)|^{p/2}\bigr\|_6^2 \le C \left\| \nabla|w(s)|^{p/2} \right\|_2^2 .$ Choose first $\varepsilon>0$ sufficiently small and then $t=t(R,M,K)>0$ sufficiently small so that the second and third terms on the right-hand side of \eqref{eq:5_Cauchy} are absorbed.
It follows that
\begin{equation}
\label{eq:5.2-after-absorb}
\E \Bigg[ \sup_{0\le s\le\tau_{n,m}} \|w(s)\|_p^p + \int_0^{\tau_{n,m}} \|w(s)\|_{3p}^p\,\dd s \Bigg] \le C \E\|a_n-a_m\|_p^p + \frac{C_{R,M,K,t}} {(k(n)\wedge k(m))^\rho} \mathcal R_m .
\end{equation}

It remains to choose the spectral sequence.
By \cref{thm:5.1}, $ \E \int_0^{\tau_{n,m}} \|u_m(s)\|_{3p}^p\,\dd s \le C_{R,M,t}K^p .
$ Moreover, since $ \chi\left(\frac{\cdot}{m}\right)u_0 \in L^2(\R^3)\cap L^p(\R^3) \ \P\text{-a.s.}, $ \cref{lem:4_L2} gives
\begin{equation*}
\E \int_0^{\tau_{n,m}} \|\nabla u_m(s)\|_2^2\,\dd s \le C_t \left( 1+ \E \left\| \chi\left(\frac{\cdot}{m}\right)u_0 \right\|_2^2 \right),
\end{equation*}
where the constant is independent of $k(m)$.
Notice that the right-hand side is finite for every fixed $m$; indeed, compact support and $\|u_0\|_p\le K$ imply $ \E \left\| \chi\left(\frac{\cdot}{m}\right)u_0 \right\|_2^2 <\infty .
$ Hence
\begin{equation*}
\mathcal R_m \le B_m, \qquad B_m := C_{R,M,K,t} \left[ 1+ \E \left\| \chi\left(\frac{\cdot}{m}\right)u_0 \right\|_2^2 \right] <\infty .
\end{equation*}
No uniform bound on $B_m$ is required.
Choose $k(m)\in\N$ recursively so that $ k(m)>k(m-1), \ k(m)^\rho\ge mB_m .
$ Then $k(m)\uparrow\infty$ and, for $n>m$, $ k(n)\wedge k(m)=k(m), $ while
\begin{equation}
\label{eq:5.2-spectral}
\sup_{n>m} \frac{\mathcal R_m} {(k(n)\wedge k(m))^\rho} \le \frac{B_m}{k(m)^\rho} \le \frac1m \longrightarrow0 .
\end{equation}

Finally, \eqref{eq:5_initialbd} gives $ a_n = P_n\mathcal P \left( \chi\left(\frac{\cdot}{n}\right)u_0 \right) \longrightarrow u_0 \ \text{in }L^p(\Omega;E), $ and therefore
\begin{equation}
\label{eq:5.2-initial-Cauchy}
\lim_{m\to\infty} \sup_{n>m} \E \|a_n-a_m\|_p^p =0 .
\end{equation}
Taking $\sup_{n>m}$ in \eqref{eq:5.2-after-absorb} and using \eqref{eq:5.2-spectral} and \eqref{eq:5.2-initial-Cauchy} proves \eqref{eq:5.Cauchy_limit}.
\end{proof}
The endpoint $\alpha=0$ in the proof of \cref{lem:5.2} requires a separate reading of the inner-defect estimate: $(c,q)=(0,p)$ is not covered by the strict upper inequality in \eqref{eq:5.2-spectral-range}.
For this case, choose $\delta>0$ sufficiently small and set $$ \frac1{r_\delta} :=\frac{1/p-\delta/2}{1-\delta}.
$$ Since $p>2$, one has $p<r_\delta<3p$ for all sufficiently small $\delta$.
Writing $$ \frac1{r_\delta} =\frac{\beta_\delta}{p} +\frac{1-\beta_\delta}{3p}, \qquad b_\delta:=p(1-\delta)(1-\beta_\delta), $$ interpolation gives $$
\begin{aligned}
\|P_{n,m}u_m\|_p^p &\le \|P_{n,m}u_m\|_2^{\delta p} \|P_{n,m}u_m\|_{r_\delta}^{(1-\delta)p} \\
&\le \frac{C}{(k(n)\wedge k(m))^{\delta p}} \|\nabla u_m\|_2^{\delta p} \|u_m\|_p^{p(1-\delta)\beta_\delta} \|u_m\|_{3p}^{b_\delta}.
\end{aligned}
$$ Here $\beta_\delta\to1$ and $b_\delta\to0$ as $\delta\downarrow0$.
Choose $\delta$ so that $\delta p<2$ and $2b_\delta/(2-\delta p)<p$.
After multiplication by the cutoff factors occurring in $J_3$, Young's inequality yields $$ (\phi_n\phi_m)^p\|P_{n,m}u_m\|_p^p \le \frac{C_R}{(k(n)\wedge k(m))^{\delta p}} \bigl(1+\|\nabla u_m\|_2^2+\|u_m\|_{3p}^p\bigr).
$$ Hypothesis \eqref{eq:4.G2} with $\alpha=0$ and H\"older's inequality in time now give the two Poisson norm estimates for $J_3$ used in that proof.

In the estimate of $T_4$, the exponent $r_*$ is chosen strictly greater than $1$.
Such a choice is always possible: take $$ \max\left\{1,\frac{2p}{p+4}\right\} <r_*< \min\left\{q_h,\frac{6p}{3p+4}\right\}, \qquad a_1=a_2:=\frac{4r_*}{2-r_*}.
$$ The interval is nonempty for $p>3$ and the prescribed $q_h$, and $a_1,a_2\in(p,3p)$ with $1/r_*=1/2+1/a_1+1/a_2$.
Thus every use of the Leray projection in this estimate is at an exponent strictly between $1$ and $\infty$.

Having established the Cauchy estimate for the regularized solutions, we next extract an almost surely convergent subsequence on a strictly positive random interval.
We retain the stopping-time extraction of \cite[Lemma~5.3]{MR4908978}, with the modifications required by the c\`adl\`ag compensated-Poisson setting.

\begin{lemma}
\label{lem:5.3} Let $p>3$, $K\ge1$, and $R\ge1$, and retain the hypotheses and notation of \cref{thm:5.1,lem:5.2}.
Suppose that $$ \|u_0\|_p\le K \qquad \P\text{-a.s.} $$ Then the spectral sequence can be chosen so that there exist an $(\mathcal F_t)$-stopping time $\tau$, a deterministic subsequence $\{u^{(n_k)}\}_{k\ge1}$, and an adapted $E$-valued c\`adl\`ag process $u$ such that $\P(0<\tau\le1)=1$, $$ u\in L^p\bigl(\Omega;\D([0,\tau];E)\bigr) \cap L^p\bigl(\Omega;L^p(0,\tau;L^{3p})\bigr), $$ and
\begin{equation}
\label{eq:5.3-pathwise-conv}
\lim_{k\to\infty} \left[ \sup_{0\le s\le\tau} \|u^{(n_k)}(s)-u(s)\|_p^p + \int_0^\tau \|u^{(n_k)}(s)-u(s)\|_{3p}^p\,\dd s \right] =0 \qquad \P\text{-a.s.}
\end{equation}
Moreover,
\begin{equation}
\label{eq:5.3-energy}
\E\left[ \sup_{0\le s\le\tau}\|u(s)\|_p^p + \int_0^\tau \left\|\nabla\bigl(|u(s)|^{p/2}\bigr)\right\|_2^2 \,\dd s \right] \le C_RK^p,
\end{equation}
where $C_R$ may depend on $R,p,\nu,N$ and the structural constants of the noise coefficients, but is independent of $K$.
\end{lemma}

\begin{proof}
We adapt the extraction argument of \cite[Lemma~5.3]{MR4908978}, retaining the closed stopping intervals required for c\`adl\`ag trajectories.

For $r\ge0$, write $$ \|v\|_{\mathcal X_r} := \left( \sup_{0\le s\le r}\|v(s)\|_p^p + \int_0^r\|v(s)\|_{3p}^p\dd s \right)^{1/p}, $$ whenever the right-hand side is finite.
This norm satisfies the triangle inequality and is nondecreasing in $r$.
For the approximating solutions, it is finite on every bounded time interval almost surely.

Choose $$ M_1>\max\{2^pM_0,1\}, \qquad M>2^{p+1}M_1.
$$ These constants are fixed independently of $K$.
Since $M_1/2^p>M_0$, \cref{thm:5.1} gives
\begin{equation}
\label{eq:5.3-smalltime}
\lim_{r\downarrow0}\sup_{n\ge1} \P\left( \|u^{(n)}\|_{\mathcal X_{r\wedge\tau_M^n}}^p \ge \frac{M_1}{2^p}K^p \right) =0.
\end{equation}

For this choice of $M$, select the spectral sequence and the deterministic time $t_*>0$ supplied by \cref{lem:5.2}.
Set $$ \bar t:=t_*\wedge1, \qquad \tau_{n,m}:=\tau_M^n\wedge\tau_M^m\wedge\bar t.
$$ The Cauchy estimate in \cref{lem:5.2} permits the selection of a deterministic, strictly increasing sequence $\{n_k\}_{k\ge1}$ such that
\begin{equation}
\label{eq:5.3-geometric}
\E\left[ \|u^{(n_{k+1})}-u^{(n_k)} \|_{\mathcal X_{\tau_{n_{k+1},n_k}}}^p \right] \le 8^{-kp}, \qquad k\ge1.
\end{equation}

For $k\ge1$, define $$ a_k:=M_1^{1/p}K+2^{-k}, \qquad \eta_k := \inf\left\{ r>0: \|u^{(n_k)}\|_{\mathcal X_r}\ge a_k \right\}, \qquad \inf\varnothing:=\infty.
$$ The process $r\mapsto\|u^{(n_k)}\|_{\mathcal X_r}^p$ is adapted, nondecreasing, and right-continuous.
Consequently, $\eta_k$ is an $(\mathcal F_t)$-stopping time.
The initial bound \eqref{eq:5.M0-initial} and right-continuity at zero imply that $\eta_k>0$ almost surely.

Furthermore, $$ a_k^p \le 2^{p-1}(M_1+1)K^p < MK^p.
$$ The definitions of the two exit times therefore give $$ \eta_k\le\tau_M^{n_k} \qquad \P\text{-a.s.} $$ Equality is allowed: a single jump may cross both localizing levels.

Define the events $$ \Omega_k := \left\{ \|u^{(n_{k+1})}-u^{(n_k)} \|_{\mathcal X_{\tau_{n_{k+1},n_k}}} \ge4^{-k} \right\}, $$ and, for $j\ge2$, set $$ \Gamma_j:=\bigcap_{k\ge j}\Omega_k^c, \qquad \Gamma:=\bigcup_{j\ge2}\Gamma_j.
$$ By \eqref{eq:5.3-geometric} and Markov's inequality, $$ \P(\Omega_k)\le2^{-kp}.
$$ Thus the Borel--Cantelli lemma yields $\Gamma_j\uparrow\Gamma$ and $\P(\Gamma)=1$.

We claim that, on $\Gamma_j$, $$ \eta_{k+1}\wedge\bar t \le \eta_k\wedge\bar t, \qquad k\ge j.
$$ Suppose that this fails for some $k\ge j$.
Then $\eta_k<\bar t$ and $\eta_k<\eta_{k+1}$.
Since $\eta_i\le\tau_M^{n_i}$, it follows that $$ \eta_k\le\tau_{n_{k+1},n_k}.
$$ The exit-time definition and right-continuity give $$ \|u^{(n_k)}\|_{\mathcal X_{\eta_k}}\ge a_k, \qquad \|u^{(n_{k+1})}\|_{\mathcal X_{\eta_k}}<a_{k+1}.
$$ Because $\omega\in\Omega_k^c$, the triangle inequality implies $$ a_k \le \|u^{(n_k)}\|_{\mathcal X_{\eta_k}} < a_{k+1}+4^{-k} < a_k, $$ where the last inequality holds for $k\ge2$.
This contradiction proves the claim.
In particular, the argument uses the lower bound at an exit time and does not require an upper bound on its post-jump value.

Set $$ \tau := \liminf_{k\to\infty}(\eta_k\wedge\bar t) = \sup_{j\ge1}\inf_{k\ge j}(\eta_k\wedge\bar t).
$$ Under the usual conditions on the filtration, countable infima and suprema of stopping times are stopping times.
Hence $\tau$ is an $(\mathcal F_t)$-stopping time.
The preceding monotonicity gives, on $\Gamma_j$,
\begin{equation}
\label{eq:5.3-tauGamma}
\tau = \lim_{k\to\infty}(\eta_k\wedge\bar t) \le \eta_k\wedge\bar t \le \tau_M^{n_k}\wedge\bar t, \qquad k\ge j.
\end{equation}

We next prove that $\tau>0$ almost surely.
Fix $\ell\in(0,\bar t)$.
The almost-sure convergence in \eqref{eq:5.3-tauGamma} implies $$ \mathds 1_{\{\tau<\ell\}} \le \liminf_{k\to\infty} \mathds 1_{\{\eta_k<\ell\}} \qquad \P\text{-a.s.} $$ Therefore, $$ \P(\tau<\ell) \le \liminf_{k\to\infty}\P(\eta_k<\ell).
$$ On $\{\eta_k<\ell\}$, the inequality $\eta_k\le\tau_M^{n_k}$ and monotonicity of the path norm give $$ \|u^{(n_k)}\|_{\mathcal X_{\ell\wedge\tau_M^{n_k}}}^p \ge a_k^p \ge \frac{M_1}{2^p}K^p.
$$ Consequently, $$ \P(\tau<\ell) \le \sup_{n\ge1} \P\left( \|u^{(n)}\|_{\mathcal X_{\ell\wedge\tau_M^n}}^p \ge \frac{M_1}{2^p}K^p \right).
$$ Letting $\ell\downarrow0$ and using \eqref{eq:5.3-smalltime}, we obtain $\P(\tau=0)=0$.

We now construct the limit.
Fix $\omega\in\Gamma_j$.
For $m>k\ge j$, \eqref{eq:5.3-tauGamma} and the triangle inequality yield $$ \|u^{(n_m)}-u^{(n_k)}\|_{\mathcal X_\tau} \le \sum_{i=k}^{m-1} \|u^{(n_{i+1})}-u^{(n_i)}\|_{\mathcal X_\tau} \le \sum_{i=k}^{m-1}4^{-i}.
$$ Thus the subsequence is almost surely Cauchy both uniformly on $[0,\tau]$ with values in $E$ and in $L^p(0,\tau;L^{3p})$.

For $s\in[0,\bar t]$, define $$ U_k(s):=u^{(n_k)}(s\wedge\tau).
$$ Each $U_k$ is adapted and $E$-valued c\`adl\`ag.
The preceding estimate shows that $\{U_k\}$ converges uniformly on $[0,\bar t]$ almost surely to an $E$-valued process $U$.
Define $U=0$ on the exceptional null set, which belongs to $\mathcal F_0$ by completeness.
For each deterministic $s$, $U(s)$ is then $\mathcal F_s$-measurable.
Moreover, a uniform limit of c\`adl\`ag paths is c\`adl\`ag, and $U$ is stopped at $\tau$ because every $U_k$ is stopped there.

Let $u$ be this stopped process, extended constantly after $\bar t$.
In particular, $$ \sup_{0\le s\le\tau} \|u^{(n_k)}(s)-u(s)\|_p \longrightarrow0 \qquad \P\text{-a.s.} $$ The initial convergence \eqref{eq:5_initialbd} gives $u(0)=u_0$ almost surely.
Uniform convergence also controls the left limits: $$ \sup_{0\le s\le\tau} \|u^{(n_k)}(s-)-u(s-)\|_p \le \sup_{0\le s\le\tau} \|u^{(n_k)}(s)-u(s)\|_p \longrightarrow0, $$ with the convention $v(0-)=v(0)$.
Since $u$ is adapted and c\`adl\`ag, $u_-$ is predictable.

Completeness of $L^p(0,\tau(\omega);L^{3p})$ gives a limit $\widetilde u$ in that space for almost every $\omega$.
Both this convergence and the uniform $L^p$ convergence imply distributional convergence on $(0,\tau(\omega))\times\mathbb R^3$.
Hence $\widetilde u=u$ almost everywhere, proving \eqref{eq:5.3-pathwise-conv}.

It remains to establish the energy estimate.
Write $$ A_k := \sup_{0\le s\le\tau}\|u^{(n_k)}(s)\|_p^p, \qquad D_k := \int_0^\tau \left\| \nabla\bigl(|u^{(n_k)}(s)|^{p/2}\bigr) \right\|_2^2\dd s.
$$ For $k\ge j$, \eqref{eq:5.3-tauGamma} implies on $\Gamma_j$ that $\tau\le\tau_M^{n_k}\wedge1$.
Applying \cref{thm:5.1} with the fixed deterministic horizon $T=1$, we obtain
\begin{equation}
\label{eq:5.3-energy1}
\sup_{k\ge j} \E\left[ \mathds 1_{\Gamma_j}(A_k+D_k) \right] \le C_RK^p.
\end{equation}
Here $C_R$ is independent of $j,k$, and $K$, since $M$ has already been fixed independently of $K$.
The events $\Gamma_j$ are used only to restrict expectations and samplewise arguments.

By Fatou's lemma and \eqref{eq:5.3-energy1}, $$ \E\left[ \mathds 1_{\Gamma_j}\liminf_{k\to\infty}D_k \right] \le C_RK^p.
$$ In particular, $\liminf_kD_k<\infty$ almost surely on $\Gamma_j$.

Fix a trajectory for which this finiteness and the previous convergence statements hold.
Set $$ f_k:=|u^{(n_k)}|^{p/2}, \qquad f:=|u|^{p/2}.
$$ For $a,b\in L^p(\mathbb R^3;\mathbb R^3)$, $$ \bigl\||a|^{p/2}-|b|^{p/2}\bigr\|_2 \le C_p \left( \|a\|_p^{p/2-1} + \|b\|_p^{p/2-1} \right) \|a-b\|_p.
$$ The uniform convergence on $[0,\tau]$ therefore gives $$ f_k\longrightarrow f \quad\text{strongly in } L^2(0,\tau;L^2(\mathbb R^3)).
$$

Choose a subsequence realizing $\liminf_kD_k$.
Along this subsequence the gradients $\nabla f_k$ are bounded in $L^2((0,\tau)\times\mathbb R^3;\mathbb R^3)$.
After a further extraction they converge weakly.
The strong convergence of $f_k$ identifies their weak limit with the distributional gradient of $f$.
Thus $ f\in L^2(0,\tau;H^1(\mathbb R^3)) $ and weak lower semicontinuity gives $$ D := \int_0^\tau \left\|\nabla\bigl(|u(s)|^{p/2}\bigr)\right\|_2^2 \dd s \le \liminf_{k\to\infty}D_k.
$$ This subsequence is used only in the deterministic lower-semicontinuity argument for the fixed trajectory; the deterministic sequence defining $u$ is unchanged.

Uniform convergence also implies $$ A_k\longrightarrow A:=\sup_{0\le s\le\tau}\|u(s)\|_p^p.
$$ Consequently, another application of Fatou's lemma, together with \eqref{eq:5.3-energy1}, yields $$ \E\left[ \mathds 1_{\Gamma_j}(A+D) \right] \le \liminf_{k\to\infty} \E\left[ \mathds 1_{\Gamma_j}(A_k+D_k) \right] \le C_RK^p.
$$ Since $\Gamma_j\uparrow\Gamma$ and $\P(\Gamma)=1$, monotone convergence proves \eqref{eq:5.3-energy}.

Finally, the Sobolev inequality on $\mathbb R^3$ gives $$ \|u(s)\|_{3p}^p = \bigl\||u(s)|^{p/2}\bigr\|_6^2 \le C \left\|\nabla\bigl(|u(s)|^{p/2}\bigr)\right\|_2^2 $$ for almost every $(s,\omega)$ with $0<s<\tau(\omega)$.
Integrating and using \eqref{eq:5.3-energy} establishes the asserted $L^p(\Omega;L^p(0,\tau;L^{3p}))$ regularity.
On the other hand, the stopped process $$ U(t):=u(t\wedge\tau),\qquad 0\le t\le1, $$ was constructed above as the uniform limit of adapted $E$-valued c\`adl\`ag processes and is therefore itself adapted and c\`adl\`ag.
Moreover, \eqref{eq:5.3-energy} gives $$ \E\sup_{0\le t\le1}\|U(t)\|_p^p = \E\sup_{0\le s\le\tau}\|u(s)\|_p^p <\infty.
$$ Thus, one gets $ u\in L^p\bigl(\Omega;\D([0,\tau];E)\bigr).
$
\end{proof}
We next establish pathwise uniqueness for the scalar-cutoff equation
\begin{equation}
\label{eq:5.23}
\begin{aligned}
\dd u =& \Big[ \nu\Delta u- \phi_u(t)^2 \mathcal P \bigl( (u\cdot\nabla)u \bigr) - \phi_u(t)^2 \mathcal P \bigl( g_N(|u|^2)u \bigr) \Big]\,\dd t \\
&+ \phi_u(t)^2 \mathcal P \sigma(u)\,\dd W_t + \int_Z (\phi_u(t-))^2 \mathcal P G(u(t-),z) \,\widetilde N(\dd t,\dd z), \\
\nabla\cdot u =&0, \\
u(0,x) =& u_0(x), \qquad \P\text{-a.s.},
\end{aligned}
\end{equation}

The identification of the limit furnished by \cref{lem:5.3} with a solution of \eqref{eq:5.23} is completed in the proof of \cref{eq:2_Main_thm}.
It uses convergence of the nonlinear coefficients and of both stochastic integrals on the stopped interval.

\begin{theorem}
[Pathwise uniqueness] \label{thm:5.4} Let $p>3$ and $T>0$, and assume the standing coefficient hypotheses.
Fix the scalar cutoff $\phi$ in \eqref{eq:5.23}.
Let $(u^{(1)},\tau_1)$ and $(u^{(2)},\tau_2)$ be local strong solutions of this equation on the same stochastic basis, with the same initial datum and driving noises.
Suppose that, for $i=1,2$,
\begin{equation}
\label{eq:5.4-solution-energy}
\E\left[ \sup_{0\le s\le T\wedge\tau_i}\|u^{(i)}(s)\|_p^p +\int_0^{T\wedge\tau_i} \left\|\nabla\bigl(|u^{(i)}(s)|^{p/2}\bigr)\right\|_2^2 \,\dd s \right]<\infty.
\end{equation}
Then, with $\tau_*:=T\wedge\tau_1\wedge\tau_2$, $$ \P\left( u^{(1)}(t)=u^{(2)}(t) \text{ for every }t\in[0,\tau_*] \right)=1.
$$
\end{theorem}

\begin{proof}
Represent the solutions by their stopped c\`adl\`ag extensions and set $$ w:=u^{(1)}-u^{(2)},\qquad D_p(a):=\left\|\nabla\bigl(|a|^{p/2}\bigr)\right\|_2^2.
$$ In particular, $w(0)=0$ almost surely.
For $i=1,2$, put $$ Q_i(t):= \sup_{0\le s\le t\wedge\tau_i}\|u^{(i)}(s)\|_p^p +\int_0^{t\wedge\tau_i}D_p(u^{(i)}(s))\dd s, \qquad 0\le t\le T.
$$ For an integer $M\ge2$, define $$ \widehat\eta_i^M :=\inf\{t\in[0,T]:Q_i(t)\ge M^p\},\qquad \eta^M:=\widehat\eta_1^M\wedge\widehat\eta_2^M\wedge\tau_*, \qquad \inf\varnothing:=\infty.
$$ The processes $Q_i$ are adapted, nondecreasing, and right-continuous, so these are stopping times.
The time $\eta^M$ may vanish when the initial norm already reaches the localizing level.
This causes no difficulty because the initial difference is zero.

With suprema over empty intervals understood as zero, the localization gives
\begin{equation}
\label{eq:5.4-preexit}
\begin{aligned}
\sup_{0\le s<\eta^M}\|u^{(i)}(s)\|_p^p&\le M^p, &\sup_{0<s\le\eta^M}\|u^{(i)}(s-)\|_p^p&\le M^p, \\
\int_0^{\eta^M}D_p(u^{(i)}(s))\,\dd s&\le M^p, &\int_0^{\eta^M}\|u^{(i)}(s)\|_{3p}^p\,\dd s&\le C M^p.
\end{aligned}
\end{equation}
The last bound follows from the Sobolev inequality.
No upper bound on $u^{(i)}(\eta^M)$ is asserted.

Write $\phi_i(s):=\phi(\|u^{(i)}(s)\|_p)$ and $\mathcal T(a):=g_N(|a|^2)a$.
On the common lifetime, the difference equation is $$ \dd w=(\nu\Delta w+\nabla\cdot f+h)\dd s +g\dd W_s +\int_Z H(s,z)\widetilde N(\dd s,\dd z), $$ where $$
\begin{aligned}
f&=-\phi_1^2\mathcal P^{(1)}(u^{(1)}\otimes u^{(1)}) +\phi_2^2\mathcal P^{(1)}(u^{(2)}\otimes u^{(2)}), \\
h&=-\phi_1^2\mathcal P\mathcal T(u^{(1)}) +\phi_2^2\mathcal P\mathcal T(u^{(2)}), \\
g(s)&=\phi_1(s-)^2\mathcal P\sigma(s,u^{(1)}(s-)) -\phi_2(s-)^2\mathcal P\sigma(s,u^{(2)}(s-)), \\
H(s,z)&=\phi_1(s-)^2\mathcal PG(s,u^{(1)}(s-),z) -\phi_2(s-)^2\mathcal PG(s,u^{(2)}(s-),z).
\end{aligned}
$$ The possible $\omega$-dependence is suppressed, and the stochastic coefficients are represented predictably.

Choose $$ \max\left\{\frac p2,\frac{3p}{p+1}\right\}<q_f<\frac{3p}{4}, \qquad \max\left\{\frac p3,\frac{3p}{2p+1}\right\}<q_h<\frac{3p}{7}.
$$ For $0<S\le T\wedge1$, set $$ \theta_M(S):=S\wedge\eta^M,\qquad X_M(S):=\sup_{0\le s\le\theta_M(S)}\|w(s)\|_p^p, \qquad Y_M(S):=\int_0^{\theta_M(S)}\|w(s)\|_{3p}^p\dd s.
$$ The nonspectral convection, taming, and Wiener estimates in the proof of \cref{lem:5.2}, together with \eqref{eq:5.4-preexit}, give, for every $\varepsilon>0$,
\begin{equation}
\label{eq:5.4-nonjump}
\E\int_0^{\theta_M(S)} \bigl(\|f(s)\|_{q_f}^p+\|h(s)\|_{q_h}^p +\|g(s)\|_{\mathbb L^p}^p\bigr)\,\dd s \le \varepsilon\E Y_M(S) +\Lambda_{R,M,\varepsilon}(S)\E X_M(S),
\end{equation}
where $\Lambda_{R,M,\varepsilon}(S)\to0$ as $S\downarrow0$.
Here the terms containing a difference of cutoff factors are controlled by $|\phi_1-\phi_2|\le C_\phi\|w\|_p$ and the localized $L^p(0,\eta^M;L^{3p})$ bounds.
The strict upper bounds on $q_f,q_h$, and the assumption $q_\sigma<3p/2$, give positive powers of $S$ in these terms.

For the Poisson term, write $H=H_1+H_2+H_3$, with $$
\begin{aligned}
H_1&=\phi_1(\phi_1-\phi_2)\mathcal PG(u^{(1)},z), \\
H_2&=\phi_1\phi_2\mathcal P \bigl(G(u^{(1)},z)-G(u^{(2)},z)\bigr), \\
H_3&=(\phi_1-\phi_2)\phi_2\mathcal PG(u^{(2)},z).
\end{aligned}
$$ All states and scalar factors in this decomposition are evaluated at $s-$.
For a predictable coefficient $J$, put 
$$
\mathcal J_{M,S}(J):=  \E \ \int_0^{\theta_M(S)}\int_Z \|J(s,z)\|_p^p\mu(\dd z)\dd s  + \E\ \left( \int_0^{\theta_M(S)}\int_Z \|J(s,z)\|_p^2\mu(\dd z)\dd s \right)^{p/2}.$$ By \eqref{eq:4.G1}, the Lipschitz continuity of $\phi$, and \eqref{eq:5.4-preexit}, $$ \mathcal J_{M,S}(H_1)+\mathcal J_{M,S}(H_3) \le C_{R,M}(S+S^{p/2})\E X_M(S).
$$ When $\alpha=0$, \eqref{eq:4.G2} gives the same bound for $\mathcal J_{M,S}(H_2)$.
For $0<\alpha<2/3$, set $\vartheta:=1-3\alpha/2\in(0,1)$.
H\"older's inequality and interpolation give, for $r=2,p$, $$ \int_Z\|H_2(s,z)\|_p^r\mu(\dd z) \le C_M\|w(s)\|_p^{\vartheta r} \|w(s)\|_{3p}^{(1-\vartheta)r} $$ for $\dd s\otimes\P$-almost every point of $(0,\eta^M]$.
In using this bound, we identify current and left-limit states only almost everywhere in time.
No $L^{3p}$-valued left limits are required.
Consequently, $$
\begin{aligned}
\E\int_0^{\theta_M(S)}\int_Z \|H_2(s,z)\|_p^p\,\mu(\dd z)\,\dd s &\le C_M S^\vartheta \E\bigl[X_M(S)^\vartheta Y_M(S)^{1-\vartheta}\bigr] \\
&\le \varepsilon\E Y_M(S) +C_{\varepsilon,M}S\E X_M(S), \\
\E\left( \int_0^{\theta_M(S)}\int_Z \|H_2(s,z)\|_p^2\,\mu(\dd z)\,\dd s \right)^{p/2} &\le C_M S^{(p-3\alpha)/2} \E\bigl[X_M(S)^\vartheta Y_M(S)^{1-\vartheta}\bigr] \\
&\le \varepsilon\E Y_M(S) +C_{\varepsilon,M}S^{(p-3\alpha)/(2\vartheta)} \E X_M(S).
\end{aligned}
$$ The finite-sum inequality for the two Poisson norms therefore yields
\begin{equation}
\label{eq:5.4-Levy-est}
\mathcal J_{M,S}(H) \le C\varepsilon\E Y_M(S) +\Lambda^J_{R,M,\varepsilon}(S)\E X_M(S), \qquad \Lambda^J_{R,M,\varepsilon}(S)\longrightarrow0.
\end{equation}

We explain the application of the linear estimate.
Multiply $f,h,g,H$ by the predictable indicator $\mathds 1_{(0,\theta_M(S)]}$ and solve the resulting linear heat equation on $[0,S]$ with zero initial datum.
The preceding bounds and the assumed energy regularity ensure the required integrability.
The resulting solution agrees with $w$ on $[0,\theta_M(S)]$: their difference satisfies the homogeneous heat equation with zero initial datum on that interval, and both equations contain the same terminal Poisson jump.
Thus restriction of the estimate in \cref{Thm:heat-exis}, using its $\nu\Delta$ version, gives by \eqref{eq:5.4-nonjump} and \eqref{eq:5.4-Levy-est}
\begin{equation}
\label{eq:5.4-energy-difference}
\E X_M(S) +c_{p,\nu}\E\int_0^{\theta_M(S)}D_p(w(s))\,\dd s \le C\varepsilon\E Y_M(S) +C\bigl(\Lambda_{R,M,\varepsilon}(S) +\Lambda^J_{R,M,\varepsilon}(S)\bigr) \E X_M(S).
\end{equation}
The Sobolev inequality gives $Y_M(S)\le C_S\int_0^{\theta_M(S)}D_p(w(s))\,\dd s$.
Choose $\varepsilon$ sufficiently small to absorb this dissipation term, and then choose $S_M\in(0,T\wedge1]$ sufficiently small to absorb the remaining supremum term.
It follows from \eqref{eq:5.4-energy-difference} that $$ w(s)=0\quad\text{for every }s\in[0,S_M\wedge\eta^M] \qquad\P\text{-a.s.} $$

The same argument can be iterated on a deterministic partition $0=t_0<\cdots<t_J=T$ with mesh at most $S_M$.
At $t_j$, restrict the difference equation to the event $A_j:=\{\eta^M>t_j\}\in\mathcal F_{t_j}$.
Inductively, $\mathds 1_{A_j}w(t_j)=0$.
The bounds in \eqref{eq:5.4-preexit} remain valid on the remaining interval, with the same constants.
Apply the preceding argument to the translated equation, with every forcing term multiplied by $\mathds 1_{A_j}$.
The translated Poisson integral is taken over $(t_j,t_j+s]$, so its lower endpoint is excluded and the jump already contained in the initial state at $t_j$ is not counted again.
Finite induction gives
\begin{equation}
\label{eq:5.4-localized}
w(s)=0\quad\text{for every }s\in[0,\eta^M] \qquad\P\text{-a.s.}
\end{equation}

By \eqref{eq:5.4-solution-energy}, $Q_i(T)<\infty$ almost surely.
Hence, for almost every trajectory, $\eta^M=\tau_*$ for every sufficiently large integer $M$.
Taking the countable intersection of the full-probability events in \eqref{eq:5.4-localized} proves the assertion, including the common terminal value.
\end{proof}

\begin{proof}
[Proof of \cref{eq:2_Main_thm}] Choose $\chi_0\in C^\infty([0,\infty);[0,1])$ such that $$ \chi_0(r)=1 \quad\text{for }0\le r\le2, \qquad \chi_0(r)=0 \quad\text{for }r\ge4, $$ and set $ \phi_R(r):=\chi_0(r/R), \ R\ge1.
$ Thus $\phi_R=1$ on $[0,2R]$ and $\phi_R=0$ on $[4R,\infty)$.

We first identify the limit for bounded initial data.
Fix $R\ge1$ and a divergence-free $\xi\in L^p(\Omega,\mathcal F_0;E)$ satisfying $\|\xi\|_p\le R$ almost surely.
Apply \cref{lem:5.3} to the approximations \eqref{eq:5.1_truncated} with initial datum $\xi$ and scalar cutoff $\phi_R$.
Write $$ v_k:=u^{(n_k)}, \qquad S_k:=P_{n_k}, $$ and denote the resulting stopping time and limit by $\bar\tau$ and $v$, respectively.
Represent $v$ by its stopped c\`adl\`ag extension.
Then $\P(0<\bar\tau\le1)=1$ and
\begin{equation}
\label{eq:2-main-path-limit}
\sup_{0\le s\le\bar\tau}\|v_k(s)-v(s)\|_p^p + \int_0^{\bar\tau}\|v_k(s)-v(s)\|_{3p}^p\,\dd s \longrightarrow0 \qquad \P\text{-a.s.}
\end{equation}
Moreover, $v$ has the energy regularity stated in \cref{lem:5.3}, and $v_k(0)\to\xi$ in $E$ almost surely.

The range of an $E$-valued c\`adl\`ag path on a compact interval has compact closure.
Consequently, the uniform boundedness and strong convergence of $S_k$ on $E$ imply $$ \sup_{0\le s\le\bar\tau} \|(S_k-I)v(s)\|_p \longrightarrow0 \qquad \P\text{-a.s.} $$ Strong convergence on $L^{3p}$ and dominated convergence also give $$ \int_0^{\bar\tau} \|(S_k-I)v(s)\|_{3p}^p\dd s \longrightarrow0 \qquad \P\text{-a.s.} $$ Together with \eqref{eq:2-main-path-limit}, these facts show that the same convergence holds with $v_k$ replaced by $S_kv_k$.

Set $$ a_k(s):=\phi_R(\|v_k(s)\|_p)^2, \qquad a(s):=\phi_R(\|v(s)\|_p)^2, \qquad \mathcal T(w):=g_N(|w|^2)w.
$$ The cutoff factors $a_k$ converge uniformly to $a$ on $[0,\bar\tau]$, and their left limits converge uniformly as well.

For clarity, write the regularized convective flux in components: $$ \mathcal B_k := \bigl((S_kv_k)_i(v_k)_j\bigr)_{i,j=1}^3, \qquad \mathcal B := \bigl(v_iv_j\bigr)_{i,j=1}^3.
$$ With the convention $(\nabla\cdot F)_i=\sum_j\partial_jF_{ij}$, the identity $\nabla\cdot v_k=0$ gives $$ \nabla\cdot\mathcal B_k = (v_k\cdot\nabla)S_kv_k.
$$ Define $$
\begin{aligned}
f_k(s) &:=-a_k(s)S_k\mathcal P^{(1)}\mathcal B_k(s), & f(s) &:=-a(s)\mathcal P^{(1)}\mathcal B(s), \\
h_k(s) &:=-a_k(s)S_k\mathcal P\mathcal T(S_kv_k(s)), & h(s) &:=-a(s)\mathcal P\mathcal T(v(s)), \\
g_k(s) &:=a_k(s-)S_k\mathcal P\sigma(S_kv_k(s-)), & g(s) &:=a(s-)\mathcal P\sigma(v(s-)), \\
H_k(s,z) &:=a_k(s-)S_k\mathcal PG(S_kv_k(s-),z), & H(s,z) &:=a(s-)\mathcal PG(v(s-),z).
\end{aligned}
$$ The possible $(\omega,s)$-dependence of $\sigma$ and $G$ is suppressed.
The stochastic coefficients are predictable.
Their left-state representation agrees with the Wiener coefficients in the approximate equation for $\dd s\otimes\P$-almost every $(s,\omega)$.

Choose $$ \max\left\{\frac p2,\frac{3p}{p+1}\right\} <q_f<\frac{3p}{4}, \qquad \max\left\{\frac p3,\frac{3p}{2p+1}\right\} <q_h<\frac{3p}{7}.
$$ These intervals are nonempty because $p>3$.
Let $\ell_f$ satisfy $$ \frac1{q_f}=\frac1p+\frac1{\ell_f}.
$$ Then $p<\ell_f<3p$ and $p<3q_h<9p/7$.
Interpolation in space and H\"older's inequality in time therefore show that both $v_k$ and $S_kv_k$ converge to $v$ in $$ L^p(0,\bar\tau;L^{\ell_f}) \cap L^{3p}(0,\bar\tau;L^{3q_h}) \qquad \P\text{-a.s.} $$ The product estimate for the convective flux and $$ |\mathcal T(x)-\mathcal T(y)| \le C_{N,\nu}(|x|^2+|y|^2)|x-y|, \qquad x,y\in\mathbb R^3, $$ then give convergence of the corresponding nonlinear terms in $L^p(0,\bar\tau;L^{q_f})$ and $L^p(0,\bar\tau;L^{q_h})$.

For the noise coefficients, we use the weighted Lipschitz assumptions.
If $0<\delta<2/3$, H\"older's inequality gives $$ \bigl\|(|b|+|c|)^\delta(b-c)\bigr\|_p \le (\|b\|_p+\|c\|_p)^\delta \|b-c\|_{p/(1-\delta)}.
$$ Writing $\theta_\delta=1-3\delta/2>0$, interpolation yields $$ \|w\|_{p/(1-\delta)} \le \|w\|_p^{\theta_\delta} \|w\|_{3p}^{1-\theta_\delta}.
$$ Thus convergence in the norm appearing in \eqref{eq:2-main-path-limit} implies convergence in $L^p(0,\bar\tau;L^{p/(1-\delta)})$.
Apply this observation with $\delta=1/2$ for the Wiener coefficient and with $\delta=\alpha$ for the jump coefficient when $\alpha>0$.
When $\alpha=0$, the required convergence follows directly from \eqref{eq:4.G2}.

All estimates involving $L^{3p}$ are understood for Lebesgue-almost every time.
Here we use $v_k(s-)=v_k(s)$ and $v(s-)=v(s)$ in $E$ for almost every $s$; no $L^{3p}$-valued left limits are assumed.
The outer multipliers are removed by their strong convergence in the corresponding coefficient spaces.
Combining these observations with the uniform convergence of the scalar cutoffs gives
\begin{equation}
\label{eq:2-main-coeff-limit}
\begin{aligned}
&\int_0^{\bar\tau} \left( \|f_k-f\|_{q_f}^p + \|h_k-h\|_{q_h}^p + \|g_k-g\|_{\mathbb L^p}^p \right)\dd s \\
&\quad+ \sum_{r\in\{2,p\}} \int_0^{\bar\tau} \left( \int_Z \|H_k(s,z)-H(s,z)\|_p^r \,\mu(\dd z) \right)^{p/r}\dd s \longrightarrow0 \qquad \P\text{-a.s.}
\end{aligned}
\end{equation}

The limiting coefficients have the required integrability.
Indeed, the cutoff bounds, the growth assumptions, and the preceding interpolation ranges give $$ \|f(s)\|_{q_f}^p + \|h(s)\|_{q_h}^p + \|g(s)\|_{\mathbb L^p}^p \le C_R\bigl(1+\|v(s)\|_{3p}^p\bigr) $$ for $\dd s\otimes\P$-almost every $(s,\omega)$ with $s<\bar\tau$, and $$ \int_Z\|H(s,z)\|_p^r\mu(\dd z) \le C_R, \qquad r\in\{2,p\}.
$$ Together with \cref{lem:5.3}, these estimates ensure that the limiting stochastic integrals are well defined after multiplication by $\mathds 1_{\{s\le\bar\tau\}}$.

We next justify convergence of the stochastic integrals without assuming convergence in expectation in \eqref{eq:2-main-coeff-limit}.
For $0\le t\le1$, set $$ Q_k(t) :=\int_0^{t\wedge\bar\tau} \Bigg[\|g_k(s)-g(s)\|_{\mathbb L^p}^2 +\int_Z\|H_k(s,z) H(s,z)\|_p^2 \mu(\dd z)+\int_Z \|H_k(s,z)-H(s,z)\|_p^p \mu(\dd z) \Bigg]\dd s.
$$ These processes are adapted, continuous, and nondecreasing.
By \eqref{eq:2-main-coeff-limit} and H\"older's inequality in time, $$ Q_k(1)\longrightarrow0 \qquad \P\text{-a.s.} $$ Define $$
\begin{aligned}
M_k^W(t) &:= \int_0^t \mathds 1_{\{s\le\bar\tau\}} (g_k(s)-g(s))\,\dd W_s, \\
M_k^J(t) &:= \int_0^t\int_Z \mathds 1_{\{s\le\bar\tau\}} (H_k(s,z)-H(s,z)) \,\widetilde N(\dd s,\dd z).
\end{aligned}
$$ For $b>0$, stop these integrals at $$ \zeta_{k,b} := \inf\{t\in[0,1]:Q_k(t)\ge b\}\wedge1, \qquad \inf\varnothing:=\infty.
$$ The Wiener and compensated-Poisson maximal inequalities, applied up to $\zeta_{k,b}$, give, for every $\varepsilon>0$, $$
\begin{aligned}
\P\left( \sup_{t\le1}\|M_k^W(t)\|_p + \sup_{t\le1}\|M_k^J(t)\|_p >\varepsilon \right) \le \P(Q_k(1)\ge b) + \frac{C_p}{\varepsilon^p} \bigl(b^{p/2}+b\bigr).
\end{aligned}
$$ First let $k\to\infty$ and then $b\downarrow0$.
It follows that both stochastic integrals converge to zero in probability, uniformly on $[0,1]$.
The stopping indicators are predictable, and the integrals retain a possible jump at $\bar\tau$.

We may now pass to the limit in the stopped weak formulation of \eqref{eq:5.1_truncated}.
For every $\varphi\in C_c^\infty(\mathbb R^3;\mathbb R^3)$, the deterministic terms converge almost surely, uniformly in time, by \eqref{eq:2-main-path-limit} and \eqref{eq:2-main-coeff-limit}.
The preceding stochastic convergence therefore identifies the limiting identity as $$
\begin{aligned}
\langle v(t\wedge\bar\tau),\varphi\rangle =& \langle\xi,\varphi\rangle + \nu\int_0^{t\wedge\bar\tau} \langle v(s),\Delta\varphi\rangle\,\dd s - \int_0^{t\wedge\bar\tau} \langle f(s),\nabla\varphi\rangle\,\dd s + \int_0^{t\wedge\bar\tau} \langle h(s),\varphi\rangle\,\dd s \\
&+ \int_0^t \mathds 1_{\{s\le\bar\tau\}} \langle g(s)^*\varphi,\dd W_s\rangle_{\mathcal U} + \int_0^t\int_Z \mathds 1_{\{s\le\bar\tau\}} \langle H(s,z),\varphi\rangle \,\widetilde N(\dd s,\dd z).
\end{aligned}
$$ The c\`adl\`ag versions make this identity simultaneous in $t$ for each test function.
A countable determining family of test functions and continuity of the pairings give the stopped weak formulation.
Also, $\nabla\cdot v=0$ follows from the strong $L^p$ convergence of the divergence-free approximations.
Thus $(v,\bar\tau)$ is a local strong solution of \eqref{eq:5.23} with cutoff $\phi_R$ and initial datum $\xi$.

Consequently, \cref{lem:5.3,thm:5.4} yield, for every such $\xi$, a pathwise unique local strong solution $(u_R,\bar\tau_R)$ of the scalar-cutoff equation, with $u_R(0)=\xi$, $0<\bar\tau_R\le1$ almost surely, and
\begin{equation}
\label{eq:2-main-fixed-R-energy}
\E\left[ \sup_{0\le s\le\bar\tau_R}\|u_R(s)\|_p^p + \int_0^{\bar\tau_R} \left\| \nabla\bigl(|u_R(s)|^{p/2}\bigr) \right\|_2^2\,\dd s \right] \le C_RR^p<\infty.
\end{equation}
Only finiteness for each fixed $R$ is needed below.

We now remove the boundedness restriction on the initial datum.
For $R\in\N$, define $$ B_R:=\{R-1\le\|u_0\|_p<R\}, \qquad u_{0,R}:=\mathds 1_{B_R}u_0.
$$ The events $B_R$ belong to $\mathcal F_0$, are pairwise disjoint, and cover $\Omega$ up to a null set.
Moreover, $$ \nabla\cdot u_{0,R}=0, \qquad \|u_{0,R}\|_p<R \qquad \P\text{-a.s.} $$ Let $(u_R,\bar\tau_R)$ be the scalar-cutoff solution constructed above with initial datum $u_{0,R}$.
Using its stopped extension $$ \widetilde u_R(t):=u_R(t\wedge\bar\tau_R), $$ define $$ \eta_R := \inf\left\{ t>0: \sup_{0\le s\le t}\|\widetilde u_R(s)\|_p \ge2R \right\}, \qquad \sigma_R^0:=\bar\tau_R\wedge\eta_R\wedge1.
$$ The running supremum is adapted and right-continuous, so $\eta_R$ is a stopping time.
Since $\|u_{0,R}\|_p<R$, right-continuity at zero gives $\sigma_R^0>0$ almost surely.

The exit-time definition and the existence of left limits imply
\begin{equation}
\label{eq:2-main-prejump}
\sup_{0\le s<\sigma_R^0}\|u_R(s)\|_p \le2R, \qquad \sup_{0\le s\le\sigma_R^0}\|u_R(s-)\|_p \le2R.
\end{equation}
Hence $$ \phi_R(\|u_R(s)\|_p)=1 \quad\text{for Lebesgue-a.e.
}s\in[0,\sigma_R^0], $$ and $$ \phi_R(\|u_R(s-)\|_p)=1, \qquad 0<s\le\sigma_R^0.
$$ Thus the stopped scalar-cutoff equation agrees with \eqref{eq:1_Main} on $[0,\sigma_R^0]$, including a possible terminal Poisson jump.
Its initial datum is $u_{0,R}$, which equals $u_0$ on $B_R$.
No bound on the post-jump value $u_R(\sigma_R^0)$ is used.

To obtain a summable energy estimate, set, for $0\le t\le1$, $$ \mathcal E_R(t) := \sup_{0\le s\le t\wedge\sigma_R^0}\|u_R(s)\|_p^p + \int_0^{t\wedge\sigma_R^0} \left\| \nabla\bigl(|u_R(s)|^{p/2}\bigr) \right\|_2^2\dd s.
$$ By \eqref{eq:2-main-fixed-R-energy}, $ \E\left[ \mathds 1_{B_R}\mathcal E_R(1) \right]<\infty.
$ Right-continuity at zero and integrability of the dissipation give $$ \mathcal E_R(t)\longrightarrow\|u_0\|_p^p \quad\text{almost surely on }B_R \quad\text{as }t\downarrow0.
$$ Since $\mathcal E_R(t)\le\mathcal E_R(1)$, dominated convergence yields $$ \lim_{t\downarrow0} \E\left[ \mathds 1_{B_R}\mathcal E_R(t) \right] = \E\left[ \mathds 1_{B_R}\|u_0\|_p^p \right].
$$ Choose a deterministic $\delta_R\in(0,1]$ such that
\begin{equation}
\label{eq:2-main-deltaR}
\E\left[ \mathds 1_{B_R}\mathcal E_R(\delta_R) \right] \le \E\left[ \mathds 1_{B_R}\|u_0\|_p^p \right] + 2^{-R}.
\end{equation}
When $\P(B_R)=0$, any $\delta_R\in(0,1]$ may be used.

Set $$ \sigma_R:=\sigma_R^0\wedge\delta_R, \qquad \tau:=\sum_{R=1}^\infty \mathds 1_{B_R}\sigma_R, $$ and define $$ u(t):=\sum_{R=1}^\infty \mathds 1_{B_R}u_R(t\wedge\sigma_R), \qquad t\ge0.
$$ On the null event $$ B_\infty:=\Omega\setminus\bigcup_{R\ge1}B_R \in\mathcal F_0, $$ set $\tau=0$ and $u=0$.
For every $t\ge0$, $$ \{\tau\le t\} = B_\infty \cup \bigcup_{R\ge1} \bigl(B_R\cap\{\sigma_R\le t\}\bigr) \in\mathcal F_t.
$$ Thus $\tau$ is a stopping time, and $\P(0<\tau\le1)=1$.
Exactly one branch is active outside $B_\infty$; hence $u$ is adapted, divergence-free, c\`adl\`ag, and stopped at $\tau$, with $u(0)=u_0$ almost surely.

We verify the stochastic formulation for the pasted process.
Let $$ C_m:=\bigcup_{R=1}^mB_R=\{\|u_0\|_p<m\}.
$$ For the Poisson coefficient, the identity $$
\mathds 1_{C_m}\mathds 1_{\{s\le\tau\}} \mathcal PG(s,u(s-),z) = \sum_{R=1}^m \mathds 1_{B_R}\mathds 1_{\{s\le\sigma_R\}} \mathcal PG(s,u_R(s-),z)
$$ holds almost everywhere.
Every term on the right is predictable.
The analogous identities hold for the Wiener coefficient and the deterministic terms.
Multiplying the stopped formulation for each branch by $\mathds 1_{B_R}$ and summing over $1\le R\le m$ therefore gives the required formulation on $C_m$.

These finite-shell formulations are consistent as $m$ increases, by locality of stochastic integration with respect to $\mathcal F_0$-measurable events.
More explicitly, the stopping times $$ \lambda_m:=m\mathds 1_{C_m} $$ increase to infinity almost surely.
Since $\tau\le1$, stopping the coefficients at $\lambda_m$ is equivalent, on $(0,\tau]$, to restricting them to $C_m$.
Thus the finite-shell stochastic integrals define the usual local stochastic integrals for the pasted process.
Letting $m\to\infty$ gives the stopped weak formulation of \eqref{eq:1_Main}.
The predictable bound \eqref{eq:2-main-prejump} ensures that a jump at $\sigma_R$, and hence at $\tau$ on $B_R$, is retained with the coefficient of the original equation.

On $B_R$, the energy of the pasted process equals $\mathcal E_R(\delta_R)$.
Consequently, Tonelli's theorem and \eqref{eq:2-main-deltaR} give $$
\begin{aligned}
\E\left[ \sup_{0\le s\le\tau}\|u(s)\|_p^p + \int_0^\tau \left\| \nabla\bigl(|u(s)|^{p/2}\bigr) \right\|_2^2\,\dd s \right] &= \sum_{R=1}^\infty \E\left[ \mathds 1_{B_R}\mathcal E_R(\delta_R) \right] \\
&\le \sum_{R=1}^\infty \E\left[ \mathds 1_{B_R}\|u_0\|_p^p \right] + \sum_{R=1}^\infty2^{-R} = \E\|u_0\|_p^p+1.
\end{aligned}
$$ This proves \eqref{eq:2-main-energy}, in fact with $C=1$ for the selected lifetime.
The Sobolev inequality $$ \|u(s)\|_{3p}^p = \bigl\||u(s)|^{p/2}\bigr\|_6^2 \le C \left\| \nabla\bigl(|u(s)|^{p/2}\bigr) \right\|_2^2 $$ then gives $u\in L^p(\Omega;L^p(0,\tau;L^{3p}))$.
The supremum estimate and the c\`adl\`ag construction give $u\in L^p(\Omega;\D([0,\tau];E))$ under the process-space convention of the paper.

Finally, let $(u^{(1)},\tau_1)$ and $(u^{(2)},\tau_2)$ be two local strong solutions of \eqref{eq:1_Main} with the same initial datum and driving noises, both satisfying the stated local energy regularity.
Write $$ \widetilde u^{(i)}(t) := u^{(i)}(t\wedge\tau_i), \qquad i=1,2.
$$ For $R\in\N$ and $T>0$, define $$
\begin{aligned}
D_R :=\{\|u_0\|_p<R\},\ &\eta_R^{(i)} := \inf\left\{ t>0: \sup_{0\le s\le t} \|\widetilde u^{(i)}(s)\|_p \ge2R \right\}, \\
\theta_{R,T} &:= T\wedge\tau_1\wedge\tau_2 \wedge\eta_R^{(1)}\wedge\eta_R^{(2)}.
\end{aligned}
$$ On $D_R$, both exit times are positive almost surely.
The same pre-jump argument as in \eqref{eq:2-main-prejump} shows that both solutions satisfy the scalar-cutoff equation with cutoff $\phi_R$ on $[0,\theta_{R,T}]$, including its terminal Poisson jump.

Because $D_R\in\mathcal F_0$, the estimates in the proof of \cref{thm:5.4} apply with every expectation multiplied by $\mathds 1_{D_R}$.
This multiplication preserves predictability, and the initial difference is zero.
Hence $$ \mathds 1_{D_R}u^{(1)}(s) = \mathds 1_{D_R}u^{(2)}(s), \qquad 0\le s\le\theta_{R,T}, \qquad \P\text{-a.s.} $$

For each fixed $T$, the stopped c\`adl\`ag paths $\widetilde u^{(1)}$ and $\widetilde u^{(2)}$ are bounded in $E$ on $[0,T]$ almost surely.
Therefore, for almost every $\omega$ and all sufficiently large integers $R$, $$ \omega\in D_R, \qquad \eta_R^{(1)}(\omega)>T, \qquad \eta_R^{(2)}(\omega)>T.
$$ It follows that $$ u^{(1)}(s)=u^{(2)}(s), \qquad 0\le s\le T\wedge\tau_1\wedge\tau_2, \qquad \P\text{-a.s.} $$ Taking the corresponding countable intersection over $T\in\N$ proves indistinguishable equality on the common closed lifetime.
This establishes pathwise uniqueness and completes the proof.
\end{proof}

\begin{theorem}
[Localized Lipschitz dependence on the initial datum] \label{thm:local-continuous-dependence} Let $p>3$, and assume the hypotheses of \cref{eq:2_Main_thm}.
Fix the stochastic basis, the coefficients, and the driving Wiener process and Poisson random measure.
Let $(u,\tau_u)$ and $(v,\tau_v)$ be local strong solutions of \eqref{eq:1_Main}, belonging to the local-energy class of \cref{eq:2_Main_thm}, with divergence-free initial data $$ u_0,v_0\in L^p(\Omega,\mathcal F_0;E).
$$ Fix $T>0$ and $M\ge1$.
Let $\theta$ be a stopping time such that $$ 0\le\theta\le T\wedge\tau_u\wedge\tau_v, \qquad \sup_{0\le s<\theta} \bigl(\|u(s)\|_p+\|v(s)\|_p\bigr)\le M \quad\P\text{-a.s.}, $$ where the supremum over the empty interval is zero.
Then
\begin{equation}
\label{eq:5.5_energy-estimate}
\E\left[ \sup_{0\le s\le\theta}\|u(s)-v(s)\|_p^p +\int_0^\theta \left\|\nabla\bigl(|u(s)-v(s)|^{p/2}\bigr)\right\|_2^2 \,\dd s \right] \le C_{M,T}\E\|u_0-v_0\|_p^p.
\end{equation}
The constant depends only on $M,T,p,\nu,N,\alpha$ and the structural constants in the coefficient assumptions.
It is independent of the initial data, the particular solutions, the comparison time $\theta$, and all auxiliary approximation parameters.
The estimate includes a possible jump at $\theta$.
The same conclusion holds with the dissipation integral replaced by $\int_0^\theta\|u(s)-v(s)\|_{3p}^p\,\dd s$, after adjusting the constant.
\end{theorem}

\begin{proof}
Set $$ w:=u-v,\qquad w_0:=u_0-v_0,\qquad D_p(a):=\left\|\nabla\bigl(|a|^{p/2}\bigr)\right\|_2^2.
$$ The pre-exit assumption and the c\`adl\`ag property imply $$ \|u(s-)\|_p+\|v(s-)\|_p\le M, \qquad 0<s\le\theta.
$$ The corresponding bound for the current states holds for $\dd s\otimes\P$-almost every point of $(0,\theta]$.
These assertions also cover $\theta=0$, since the indicated integration interval is then empty.

Choose $$ \max\left\{\frac p2,\frac{3p}{p+1}\right\}<q_f<\frac{3p}{4}, \qquad \max\left\{\frac p3,\frac{3p}{2p+1}\right\}<q_h<\frac{3p}{7}, $$ and define $$ \frac1{r_f}:=\frac1{q_f}-\frac1p, \qquad \frac1{r_h}:=\frac1{q_h}-\frac2p, \qquad r_G:=\frac{p}{1-\alpha}.
$$ Then $r_f,r_h\in(p,3p)$ and $r_G\in[p,3p)$.

On the common lifetime, the difference equation is $$ \dd w=(\nu\Delta w+\nabla\cdot f+h)\dd s +\Sigma\dd W_s +\int_Z J(s,z)\widetilde N(\dd s,\dd z), $$ where $$
\begin{aligned}
f&=-\mathcal P^{(1)}(u\otimes u-v\otimes v), \\
h&=-\mathcal P\bigl(g_N(|u|^2)u-g_N(|v|^2)v\bigr), \\
\Sigma(s)&=\mathcal P\bigl(\sigma(s,u(s-)) -\sigma(s,v(s-))\bigr), \\
J(s,z)&=\mathcal P\bigl(G(s,u(s-),z)-G(s,v(s-),z)\bigr).
\end{aligned}
$$ The possible $\omega$-dependence is suppressed.
H\"older's inequality, boundedness of the Leray projection, and the coefficient difference assumptions give
\begin{equation}
\label{eq:5.5-coefficient-bounds}
\begin{aligned}
\|f(s)\|_{q_f}&\le C M\|w(s)\|_{r_f},& \|h(s)\|_{q_h}&\le C M^2\|w(s)\|_{r_h}, \\
\|\Sigma(s)\|_{\mathbb L^p} &\le C M^{1/2}\|w(s)\|_{2p},& \int_Z\|J(s,z)\|_p^r\,\mu(\dd z) &\le C M^{\alpha r}\|w(s)\|_{r_G}^r, \qquad r=2,p.
\end{aligned}
\end{equation}
These estimates are asserted for $\dd s\otimes\P$-almost every point of $(0,\theta]$.
The predictable integrands themselves retain the left-limit states; the replacement by current states in these norm estimates uses only $u_-=u$ and $v_-=v$ almost everywhere in time.

We first verify that the linear estimate can be used without assuming dissipation regularity for the difference.
The energy regularity of $u$ and $v$ implies $$ \E\sup_{s\le\theta}\|w(s)\|_p^p +\E\int_0^\theta\|w(s)\|_{3p}^p\dd s<\infty.
$$ For $p\le r\le3p$, put $\beta_r:=\frac32(1-p/r)\in[0,1]$.
Interpolation in space and H\"older's inequality in time and probability yield $$
\begin{aligned}
\E\int_0^\theta\|w(s)\|_r^p\,\dd s \le& T^{1-\beta_r} \left(\E\sup_{s\le\theta}\|w(s)\|_p^p \right)^{1-\beta_r} \left(\E\int_0^\theta\|w(s)\|_{3p}^p\,\dd s \right)^{\beta_r}<\infty,
\end{aligned}
$$ with the endpoint cases read directly.
Moreover, \eqref{eq:5.5-coefficient-bounds} gives, for $0<t\le T$, $$
\begin{aligned}
\E\int_0^{t\wedge\theta}\int_Z \|J(s,z)\|_p^p\,\mu(\dd z)\,\dd s &\le C_M\E\int_0^{t\wedge\theta} \|w(s)\|_{r_G}^p\,\dd s, \\
\E\left( \int_0^{t\wedge\theta}\int_Z \|J(s,z)\|_p^2\,\mu(\dd z)\,\dd s \right)^{p/2} &\le C_M T^{p/2-1}\E\int_0^{t\wedge\theta} \|w(s)\|_{r_G}^p\,\dd s.
\end{aligned}
$$ Thus $\mathds 1_{(0,\theta]}f$, $\mathds 1_{(0,\theta]}h$, $\mathds 1_{(0,\theta]}\Sigma$ and $\mathds 1_{(0,\theta]}J$ satisfy all the integrability requirements of \cref{Thm:heat-exis}.
The indicator $\mathds 1_{(0,\theta]}$ is predictable.

Let $\widehat w$ be the solution on $[0,T]$ of the linear heat equation with these forcing terms and initial datum $w_0$.
On $[0,\theta]$, the difference $\widehat w-w$ satisfies the homogeneous heat equation with zero initial datum.
Pathwise uniqueness for the deterministic heat equation in the $L^p$ class therefore gives $$ \widehat w(t)=w(t),\qquad 0\le t\le\theta, \qquad\P\text{-a.s.} $$ The endpoint is included because both formulations contain the same Poisson jump there.
The process $\widehat w$ continues by the homogeneous heat equation after $\theta$; it is not the constant stopped extension of $w$.

Define $$ Z(t):=\E\sup_{0\le s\le t}\|\widehat w(s)\|_p^p +\E\int_0^t D_p(\widehat w(s))\dd s.
$$ The $\nu\Delta$ version of \cref{Thm:heat-exis}, \eqref{eq:5.5-coefficient-bounds}, and the two preceding Poisson bounds imply $$ Z(t)\le C_T\E\|w_0\|_p^p +C_{M,T}\E\int_0^{t\wedge\theta} \bigl(\|w(s)\|_{r_f}^p+\|w(s)\|_{r_h}^p +\|w(s)\|_{2p}^p+\|w(s)\|_{r_G}^p\bigr) \dd s.
$$ The constants may be chosen uniformly for $0<t\le T$: apply the linear estimate with horizon $T$ to the forcing terms additionally restricted to $(0,t]$, and then restrict the resulting solution to $[0,t]$.

For $p\le r<3p$, the Sobolev inequality and Young's inequality give
\begin{equation}
\label{eq:5.5-interpolation}
C_{M,T}\|a\|_r^p \le \varepsilon D_p(a) +C_{\varepsilon,M,T}\|a\|_p^p,
\end{equation}
whenever $a\in L^p$ and $|a|^{p/2}\in H^1$.
Indeed, for $p<r<3p$ one has $$ \|a\|_r^p \le \|a\|_p^{p(1-\beta_r)}\|a\|_{3p}^{p\beta_r} \le C\|a\|_p^{p(1-\beta_r)}D_p(a)^{\beta_r}, \qquad 0<\beta_r<1.
$$ For $r=p$, \eqref{eq:5.5-interpolation} is immediate without the dissipation term.
In particular this covers $r_G=p$ when $\alpha=0$.

Apply \eqref{eq:5.5-interpolation} to $\widehat w(s)$ for the four exponents in the bound for $Z(t)$.
Since $w=\widehat w$ on $[0,\theta]$, choosing $\varepsilon$ sufficiently small and absorbing the total dissipation contribution gives $$ Z(t)\le C_{M,T}\E\|w_0\|_p^p +C_{M,T}\int_0^t Z(s)\dd s.
$$ Gronwall's inequality yields $Z(T)\le C_{M,T}\E\|w_0\|_p^p$.
Restriction to $[0,\theta]$ proves \eqref{eq:5.5_energy-estimate}, and the Sobolev inequality proves the assertion concerning $L^{3p}$.
Throughout the argument the stopped Poisson integrand is evaluated at the predictable pre-jump states and retains its upper endpoint.
\end{proof}

Theorem~\ref{thm:local-continuous-dependence} gives continuous dependence on common localized intervals.
More precisely, suppose $u_0^{(n)}\to u_0$ in $L^p(\Omega;E)$, and let $u^{(n)}$ and $u$ be the corresponding local solutions on the same stochastic basis with the same coefficients and noises.
For fixed $M,T$, let $\theta_n$ satisfy the comparison assumptions of that theorem for the pair $(u^{(n)},u)$.
Then $$ \E\left[ \sup_{0\le s\le\theta_n}\|u^{(n)}(s)-u(s)\|_p^p +\int_0^{\theta_n}\|u^{(n)}(s)-u(s)\|_{3p}^p\dd s \right]\longrightarrow0.
$$ This statement uses a common stopping time for each pair of solutions.
It makes no assertion about convergence of their local or maximal lifetimes.

\section{Maximal Solution}
\label{sec:maximal-solution}

The local theory established in the preceding section proves a pathwise unique strong solution of \eqref{eq:1_Main} on a strictly positive stochastic interval.
The corresponding lifetime, however, is only a localizing stopping time and need not be maximal.
We now pass from this local description to a canonical maximal one.
The construction is based solely on pathwise uniqueness and stochastic pasting.
Compatible local solutions are identified on their common lifetime and concatenated, which leads to a maximal stopping time obtained from the family of all attainable local lifetimes.

The maximal solution is obtained independently of any continuation criterion.
Since the equation is driven by a compensated Poisson random measure, the maximal process is a priori canonical only on $[0,\tau_{\max})$.
No terminal value at $\tau_{\max}$ is imposed without a separate endpoint argument.
Continuation at a finite maximal lifetime and the corresponding blow-up alternative are therefore left to the global theory developed in the second part of this work.

\begin{definition}
[Maximal local strong solution] \label{def:maximal-solution} Let $E$ defined as before and $p>3.
$ A \emph{maximal local strong solution} of \eqref{eq:1_Main} is a triple $ \bigl(u,\{\tau_n\}_{n\ge1},\tau_{\max}\bigr) $ satisfying the following properties.

\begin{enumerate}
\item[(i)] $\tau_{\max}$ is an $(\mathcal F_t)$-stopping time with $ \P(\tau_{\max}>0)=1, $ and $\{\tau_n\}_{n\ge1}$ is an increasing sequence of $(\mathcal F_t)$-stopping times such that
\begin{equation*}
0<\tau_n\le\tau_{n+1}, \qquad \tau_n\uparrow\tau_{\max} \quad \P\text{-a.s.}
\end{equation*}

\item[(ii)] For every $n\ge1$, the pair $(u,\tau_n)$ is a local strong solution of \eqref{eq:1_Main} with the prescribed initial datum and driving noises, belonging to the local-energy uniqueness class.
In particular,
\begin{equation*}
\E \Bigg[ \sup_{0\le s\le\tau_n}\|u(s)\|_p^p + \int_0^{\tau_n} \left\| \nabla\bigl(|u(s)|^{p/2}\bigr) \right\|_2^2\,\dd s \Bigg] <\infty .
\end{equation*}

\item[(iii)] The process $u$ is canonically defined on the stochastic interval $[0,\tau_{\max})$ and is adapted there.
For every $n\ge1$, its local representative up to $\tau_n$ is $E$-valued and c\`adl\`ag, and the representatives are consistent on common stochastic intervals.

\item[(iv)] If $(v,\sigma)$ is any other local strong solution of \eqref{eq:1_Main} on the same stochastic basis, with the same initial datum and the same driving Wiener process and Poisson random measure, and belonging to the same local-energy uniqueness class, then $\sigma\le\tau_{\max} \ \P\text{-a.s.}$, and
\begin{equation*}
\P\left( v(t)=u(t) \text{ for every }0\le t<\sigma \right) =1.
\end{equation*}
Moreover, on $\{\sigma<\tau_{\max}\}$ the equality extends to the endpoint:
\begin{equation*}
\mathds 1_{\{\sigma<\tau_{\max}\}}v(\sigma) = \mathds 1_{\{\sigma<\tau_{\max}\}}u(\sigma) \qquad \P\text{-a.s.}
\end{equation*}
Here $u(\sigma)$ is well defined on $\{\sigma<\tau_{\max}\}$ through any localization satisfying $\sigma\le\tau_n$.
\end{enumerate}
\end{definition}
\begin{lemma}
[Restart at a bounded stopping time] \label{lem:6.2-restart} Let $p>3$ and assume the hypotheses of the local existence theorem \cref{eq:2_Main_thm}, including \eqref{eq:2_WienerH}, \eqref{Hypo_W2}, \eqref{eq:4.G1}--\eqref{eq:4.G3}, and \eqref{eq:4_G5}.
Suppose that the independent driving noises $(W,N)$ are jointly compatible with $(\mathcal F_t)$: for every deterministic $t\ge0$, their joint increments after $t$ are independent of $\mathcal F_t$.

Let $\rho$ be a bounded $(\mathcal F_t)$-stopping time and let $\xi\in L^p(\Omega,\mathcal F_\rho,\P;E)$.
Set $\mathcal F_s^\rho=\mathcal F_{\rho+s}, \ge0$, and define the shifted noises by
\begin{align}
W_s^\rho(h)&=W_{\rho+s}(h)-W_\rho(h), &&h\in\mathcal U,
\label{eq:6.2-shifted-Wiener}\\
N^\rho((0,s]\times A)&=N((\rho,\rho+s]\times A), &&A\in\mathcal Z,\quad\mu(A)<\infty.
\label{eq:6.2-shifted-PRM}
\end{align}
Then $(\mathcal F_s^\rho)$ satisfies the usual conditions, $W^\rho$ is a cylindrical $\mathcal U$-Wiener process, and $N^\rho$ is a Poisson random measure relative to this filtration, independent of $W^\rho$, with compensator $\dd s\,\mu(\dd z)$.
Write $\widetilde N^\rho(\dd s,\dd z) =N^\rho(\dd s,\dd z)-\dd s\,\mu(\dd z).$

Retaining the possible $(\omega,t)$-dependence of the coefficients, define
\begin{equation}
\label{eq:6.2-shifted-coefficients}
\begin{aligned}
\sigma^\rho(\omega,s,a) &=\sigma(\omega,\rho(\omega)+s,a), \\
G^\rho(\omega,s,a,z) &=G(\omega,\rho(\omega)+s,a,z).
\end{aligned}
\end{equation}

There exist an $(\mathcal F_s^\rho)$-stopping time $\vartheta$, with $0<\vartheta\le1$ almost surely, and a pathwise unique local strong solution $(v,\vartheta)$ of
\begin{equation}
\label{eq:6.2-shifted-equation}
\begin{aligned}
\dd v(s)=&\Big[ \nu\Delta v(s)-\mathcal P((v(s)\cdot\nabla)v(s)) -\mathcal P\bigl(g_N(|v(s)|^2)v(s)\bigr) \Big]\dd s \\
&+\mathcal P\sigma^\rho(s,v(s-))\,\dd W_s^\rho +\int_Z\mathcal PG^\rho(s,v(s-),z) \,\widetilde N^\rho(\dd s,\dd z), \\ v(0)&=\xi \qquad \nabla\cdot \xi =0.
\end{aligned}
\end{equation}
Here and below the dependence on $\omega$ is suppressed.
The solution is divergence-free and satisfies
\begin{equation*}
\label{eq:6.2-shifted-regularity}
v\in L^p\bigl(\Omega;\D([0,\vartheta];E)\bigr) \cap L^p\bigl(\Omega;L^p(0,\vartheta;L^{3p})\bigr),
\end{equation*}
and
\begin{equation}
\label{eq:6.2-shifted-energy}
\E\left[ \sup_{0\le s\le\vartheta}\|v(s)\|_p^p +\int_0^\vartheta \left\|\nabla\bigl(|v(s)|^{p/2}\bigr)\right\|_2^2\,\dd s \right] \le C\left(1+\E\|\xi\|_p^p\right),
\end{equation}
with the structural dependence of the local theorem and with $C$ independent of $\rho$ and all auxiliary cutoff parameters.
Moreover, $\widehat\tau=\rho+\vartheta$ is an $(\mathcal F_t)$-stopping time satisfying $\widehat\tau>\rho$ almost surely, and $\widehat v(t)=v(t-\rho),\ \rho\le t\le\widehat\tau,$ is an adapted c\`adl\`ag solution branch of \eqref{eq:1_Main}, issued from $\xi$ at $\rho$ and driven by the original noises on $(\rho,\widehat\tau]$.
Its extension
\begin{equation}
\label{eq:6.2-original-stopped-extension}
\overline v(t)=
\begin{cases}
0,&0\le t<\rho, \\
v\bigl((t-\rho)\wedge\vartheta\bigr),&t\ge\rho,
\end{cases}
\end{equation}
is adapted and c\`adl\`ag on $[0,\infty)$.
\end{lemma}

\begin{proof}
The stopping-time sigma-field identity $\bigcap_{r>s}\mathcal F_{\rho+r}=\mathcal F_{\rho+s}$ shows that $(\mathcal F_s^\rho)$ is right-continuous; completeness is inherited from $(\mathcal F_t)$.
In particular, $\mathcal F_0^\rho=\mathcal F_\rho$.

For finitely many $h_i\in\mathcal U$ and sets $A_j\in\mathcal Z$ with $\mu(A_j)<\infty$, consider the finite-dimensional process with coordinates $W_t(h_i)$ and $N((0,t]\times A_j)$.
It is a c\`adl\`ag jump process whose increments after each deterministic time are independent of the corresponding $\mathcal F_t$.
But now the new origin is the random time $\rho$.

To transfer deterministic-time independence to $\rho$, approximate $\rho$ from above by finite-valued dyadic stopping times: $ \rho_n = 2^{-n} (\lfloor 2^n \rho \rfloor + 1).
$ Then $$ \rho \le \rho_n < \rho + 2^{-n}, \qquad \rho_n \downarrow \rho.
$$

Each $\rho_n$ takes only finitely many deterministic values.
Now using right-continuity of these coordinates extends this independent-increment property to that stopping time.
Applied at $\rho+s$, this proves that the increments of $(W^\rho,N^\rho)$ after $s$ are independent of $\mathcal F_{\rho+s}$ and have the prescribed Wiener--Poisson law.
Applying this argument to every finite collection of coordinates proves the assertions concerning the shifted noises, including their mutual independence and the compensator of $N^\rho$.
This argument uses only the sigma-finiteness of $(Z,\mathcal Z,\mu)$.

Predictability is preserved by the shift $s\mapsto\rho+s$.
Indeed, this holds for adapted left-continuous processes by evaluation at stopping times, and hence for the predictable sigma-algebra they generate.
The same argument with the state and mark variables gives the required product measurability.
The coefficients in \eqref{eq:6.2-shifted-coefficients} satisfy the hypotheses of \cref{eq:2_Main_thm} relative to $(\mathcal F_s^\rho)$, with the same structural constants.
Since $\xi$ is $\mathcal F_0^\rho$-measurable, that theorem supplies a local strong solution with initial datum $\xi$.
Stopping its lifetime at $1$, if necessary, gives $0<\vartheta\le1$ and preserves the energy estimate.
This proves \eqref{eq:6.2-shifted-energy} and pathwise uniqueness in the local-energy class.
The Sobolev inequality $\|v\|_{3p}^p\lesssim \|\nabla(|v|^{p/2})\|_2^2$ gives the second space in \eqref{eq:6.2-shifted-regularity}.
One can use $v(s-)$ in the Wiener coefficient is because $v=v_-$ for $\dd s\otimes\P$-almost every $(s,\omega)$.

We establish the stopping-time translation explicitly.
For any $(\mathcal F_s^\rho)$-stopping time $S$, $ \{\rho+S<t\} =\bigcup_{q\in\mathbb Q,\ q>0} \bigl(\{S<q\}\cap\{\rho+q<t\}\bigr).
$ Each event on the right belongs to $\mathcal F_t$, since $\{S<q\}\in\mathcal F_{\rho+q}$.
Right-continuity therefore implies that $\rho+S$ is an $(\mathcal F_t)$-stopping time.
The definition of stopped sigma-fields also gives
\begin{equation}
\label{eq:6.2-stopping-field-translation}
\mathcal F_S^\rho=\mathcal F_{\rho+S}.
\end{equation}
Taking $S=\vartheta$ proves the stopping-time assertion for $\widehat\tau$ and the strict inequality follows from $\vartheta>0$.

To verify adaptedness of \eqref{eq:6.2-original-stopped-extension}, fix $t\ge0$ and set $a_t=(t-\rho)^+$.
Then $a_t$ is an $(\mathcal F_s^\rho)$-stopping time, because $ \{a_t\le s\}=\{\rho+s\ge t\}\in\mathcal F_{\rho+s}, \ s\ge0.
$ The stopped process $V(s)=v(s\wedge\vartheta)$ is adapted and c\`adl\`ag relative to this filtration.
Thus $V(a_t)$ is $\mathcal F_{\rho+a_t} =\mathcal F_{\rho\vee t}$-measurable by \eqref{eq:6.2-stopping-field-translation}.
On $\{\rho\le t\}$ this is measurability at time $t$, and hence $\overline v(t)=\mathds 1_{\{\rho\le t\}}V(a_t)$ is $\mathcal F_t$-measurable.
Its paths are c\`adl\`ag, with value $\xi$ at $\rho$ and constant value $v(\vartheta)$ after $\widehat\tau$.
In particular, $\overline v_-$ is predictable.

Set $\chi_\rho(r)=\mathds 1_{\{\rho<r\le\widehat\tau\}}$.
This process is adapted and left-continuous, hence predictable.
For a test function $\varphi\in C_c^\infty(\R^3;\R^3)$, write $ \mathcal A_\varphi(a) =\nu\langle a,\Delta\varphi\rangle -\langle\mathcal P((a\cdot\nabla)a),\varphi\rangle -\langle\mathcal P(g_N(|a|^2)a),\varphi\rangle.
$ Translation of the stopped shifted equation gives, for every $t\ge0$, almost surely,
\begin{equation}
\label{eq:6.2-original-time-weak}
\begin{aligned}
\langle\overline v(t),\varphi\rangle =&\mathds 1_{\{\rho\le t\}}\langle\xi,\varphi\rangle +\int_0^t\chi_\rho(r) \mathcal A_\varphi(\overline v(r))\,\dd r +\int_0^t\chi_\rho(r) \left\langle \bigl(\mathcal P\sigma(r,\overline v(r-))\bigr)^*\varphi, \dd W_r \right\rangle_{\mathcal U} \\
& \qquad+\int_0^t\int_Z\chi_\rho(r) \left\langle\mathcal PG(r,\overline v(r-),z),\varphi\right\rangle \,\widetilde N(\dd r,\dd z).
\end{aligned}
\end{equation}
Here all integrands are set equal to zero outside $(\rho,\widehat\tau]$.
For completeness, the translation identities hold first for elementary predictable integrands such as a shifted interval $(a,b]$ becomes $(\rho+a,\rho+b]$, and a coefficient measurable at $\mathcal F_a^\rho$ is measurable at $\mathcal F_{\rho+a}$.
The identities for $W^\rho$ and $N^\rho$ then follow directly from \eqref{eq:6.2-shifted-Wiener}--\eqref{eq:6.2-shifted-PRM}; the compensator transforms by the ordinary change of variables $r=\rho+s$.
Approximation in the defining stochastic-integral norms, after localization, extends these identities to the integrands in \eqref{eq:6.2-shifted-equation}.
Evaluating the stopped shifted formulation at $a_t$ and using these identities proves \eqref{eq:6.2-original-time-weak}.
The formulations may be chosen simultaneously in time, for each test function, by their c\`adl\`ag versions.

For $t\ge\rho$, \eqref{eq:6.2-original-time-weak} is exactly the stopped weak formulation issued from $\xi$ at $\rho$.
The indicator $\chi_\rho$ excludes a Poisson mark at $\rho$ and includes a possible mark at $\widehat\tau$, with coefficient evaluated at the left limit.
Thus, when $\xi=u(\rho)$, the restart does not count the jump already contained in $u(\rho)$ a second time.
The local-energy estimate on $[\rho,\widehat\tau]$ follows from \eqref{eq:6.2-shifted-energy} by time translation.
Finally, any competing branch in the same class shifts to a solution with initial datum $\xi$ and noises $(W^\rho,N^\rho)$.
Pathwise uniqueness on the shifted basis therefore gives uniqueness on the original common lifetime.
\end{proof}
\begin{lemma}
[Stochastic pasting] \label{lem:6.3-pasting} Let $p>3$, and assume the hypotheses of \cref{eq:2_Main_thm}.
All local strong solutions below belong to its local-energy uniqueness class and are represented by their stopped c\`adl\`ag versions.

\begin{enumerate}
\item[(i)] Let $(u^{(1)},\tau_1)$ and $(u^{(2)},\tau_2)$ be local strong solutions of \eqref{eq:1_Main} on the same stochastic basis, with the same initial datum $u_0$ and the same driving noises.
Then
\begin{equation}
\label{eq:6.3-overlap}
\P\left( u^{(1)}(t)=u^{(2)}(t) \text{ for every }0\le t\le\tau_1\wedge\tau_2 \right)=1.
\end{equation}
There is a unique local strong solution $(u,\tau_1\vee\tau_2)$ extending both solutions on their closed lifetimes.
The family
\begin{equation*}
\mathscr T =\left\{\sigma:
\begin{array}
{l} \text{there exists a local strong solution }(v,\sigma) \text{ of \eqref{eq:1_Main}} \\
\text{in the above class with the prescribed initial datum and driving noises}
\end{array}
\right\}
\end{equation*}
is upward directed.

\item[(ii)] Assume additionally the joint compatibility condition of \cref{lem:6.2-restart}.
Let $(u,\rho)$ be a local strong solution of \eqref{eq:1_Main}, where $\rho$ is bounded.
Let $\widehat v$ be the restarted branch supplied by \cref{lem:6.2-restart}, issued from $\xi=u(\rho)$ and defined on $[\rho,\widehat\tau]$, with $\widehat\tau>\rho$ almost surely.
Then $u$ and $\widehat v$ have a unique pasted extension $(w,\widehat\tau)$ in the same local-energy class.
A possible Poisson jump at $\rho$ is counted exactly once, and a possible terminal jump at $\widehat\tau$ is retained.
\end{enumerate}
\end{lemma}

\begin{proof}
For a process $v$ on a stochastic interval $[\alpha,\beta]$, write $$ \mathcal E(v;\alpha,\beta) :=\E\left[ \sup_{\alpha\le s\le\beta}\|v(s)\|_p^p +\int_\alpha^\beta \left\|\nabla\bigl(|v(s)|^{p/2}\bigr)\right\|_2^2\dd s \right].
$$ For $\varphi\in C_c^\infty(\R^3;\R^3)$, set $ \mathcal A_\varphi(a) :=\nu\langle a,\Delta\varphi\rangle -\langle\mathcal P((a\cdot\nabla)a),\varphi\rangle -\langle\mathcal P(g_N(|a|^2)a),\varphi\rangle.
$ For an adapted c\`adl\`ag process $v$ and a bounded predictable process $\chi$, whenever the integrals are defined, denote
\begin{equation*}
\begin{aligned}
\mathcal I_\varphi(v,\chi)(t) :=&\int_0^t\chi(s)\mathcal A_\varphi(v(s))\,\dd s +\int_0^t\chi(s) \left\langle \bigl(\mathcal P\sigma(s,v(s-))\bigr)^*\varphi,\dd W_s \right\rangle_{\mathcal U} \\
&+\int_0^t\int_Z\chi(s) \left\langle\mathcal PG(s,v(s-),z),\varphi\right\rangle \,\widetilde N(\dd s,\dd z).
\end{aligned}
\end{equation*}
The possible $\omega$-dependence of the coefficients is suppressed.
We use predictable versions of the Wiener integrands.
The identities below hold indistinguishably for each fixed test function.

\emph{(i)} Pathwise uniqueness gives \eqref{eq:6.3-overlap}.
Set $\theta=\tau_1\wedge\tau_2$ and $\tau=\tau_1\vee\tau_2$, and define
\begin{equation*}
u(t) :=u^{(1)}(t\wedge\tau_1) +u^{(2)}(t\wedge\tau_2) -u^{(2)}(t\wedge\theta), \qquad t\ge0.
\end{equation*}
This process is adapted, divergence-free, c\`adl\`ag, and stopped at the positive stopping time $\tau$.
On the full-probability event in \eqref{eq:6.3-overlap}, it agrees with $u^{(i)}$ on $[0,\tau_i]$, $i=1,2$.
In particular,
\begin{equation}
\label{eq:6.3-energy}
\mathcal E(u;0,\tau) \le \mathcal E(u^{(1)};0,\tau_1) +\mathcal E(u^{(2)};0,\tau_2) <\infty.
\end{equation}

Put $$\chi_1=\mathds 1_{(0,\tau_1]},  \chi_{12} =\mathds 1_{(0,\tau_2]} -\mathds 1_{(0,\theta]} =\mathds 1_{(\tau_1,\tau_2]}.$$

These processes are predictable and satisfy $\chi_1+\chi_{12}=\mathds 1_{(0,\tau]}$.
Adding the stopped formulation for $(u^{(1)},\tau_1)$ to the difference between those for $(u^{(2)},\tau_2)$ and $(u^{(2)},\theta)$ yields
\begin{equation}
\label{eq:6.3-pasted-weak}
\begin{aligned}
\langle u(t\wedge\tau),\varphi\rangle =&\langle u_0,\varphi\rangle +\mathcal I_\varphi(u^{(1)},\chi_1)(t) +\mathcal I_\varphi(u^{(2)},\chi_{12})(t) \\
=&\langle u_0,\varphi\rangle +\mathcal I_\varphi(u,\mathds 1_{(0,\tau]})(t).
\end{aligned}
\end{equation}
The second equality follows from agreement of the states and their left limits on the respective integration intervals.
All integrals are inherited from the constituent stopped equations by restriction and addition.
On $\{\tau_1<\tau_2\}$, the Poisson intervals $(0,\tau_1]$ and $(\tau_1,\tau_2]$ are disjoint and retain their upper endpoints.
On $\{\tau_2\le\tau_1\}$ the second interval is empty.
Thus \eqref{eq:6.3-pasted-weak} is the required stopped weak formulation, including terminal jumps.
Together with \eqref{eq:6.3-energy}, it proves existence in the prescribed class.
Uniqueness follows from \cref{eq:2_Main_thm}, and the construction gives $\sigma_1\vee\sigma_2\in\mathscr T$ whenever $\sigma_1,\sigma_2\in\mathscr T$.

\emph{(ii)} The stopped value $\xi=u(\rho)$ is $\mathcal F_\rho$-measurable and satisfies $ \E\|\xi\|_p^p \le\mathcal E(u;0,\rho)<\infty.
$ Hence \cref{lem:6.2-restart} applies.
Let
\begin{equation*}
\overline v(t)=
\begin{cases}
0,&0\le t<\rho, \\
\widehat v(t),&\rho\le t\le\widehat\tau, \\
\widehat v(\widehat\tau),&t>\widehat\tau,
\end{cases}
\qquad \Lambda_\rho(t)=\mathds 1_{\{\rho\le t\}}\xi.
\end{equation*}
The process $\overline v$ is adapted and c\`adl\`ag by \cref{lem:6.2-restart}; the same properties hold for $\Lambda_\rho$ because $\xi$ is $\mathcal F_\rho$-measurable.
Define
\begin{equation*}
w(t) :=u(t\wedge\rho)+\overline v(t)-\Lambda_\rho(t), \qquad t\ge0.
\end{equation*}
Then $w$ is adapted, divergence-free, c\`adl\`ag, and stopped at $\widehat\tau$.
It agrees with $u$ on $[0,\rho]$ and with $\widehat v$ on $[\rho,\widehat\tau]$.
The energy estimate in \cref{lem:6.2-restart} gives
\begin{equation}
\label{eq:6.3-restart-energy-expectation}
\begin{aligned}
\mathcal E(w;0,\widehat\tau) &\le\mathcal E(u;0,\rho) +\mathcal E(\widehat v;\rho,\widehat\tau)\le\mathcal E(u;0,\rho) +C\bigl(1+\E\|\xi\|_p^p\bigr) <\infty.
\end{aligned}
\end{equation}

Set $\chi_\rho :=\mathds 1_{(\rho,\widehat\tau]} =\mathds 1_{(0,\widehat\tau]} -\mathds 1_{(0,\rho]}.$ This process is predictable.
Since $\overline v(t\wedge\rho)=\Lambda_\rho(t)$, the translated formulation in \cref{lem:6.2-restart} reads
\begin{equation*}
\langle\overline v(t)-\Lambda_\rho(t),\varphi\rangle =\mathcal I_\varphi(\overline v,\chi_\rho)(t).
\end{equation*}
Adding the stopped formulation for $(u,\rho)$ and using $\mathds 1_{(0,\rho]}+\chi_\rho =\mathds 1_{(0,\widehat\tau]}$ gives
\begin{equation}
\label{eq:6.3-restarted-pasted-weak}
\begin{aligned}
\langle w(t\wedge\widehat\tau),\varphi\rangle =&\langle u_0,\varphi\rangle +\mathcal I_\varphi(u,\mathds 1_{(0,\rho]})(t) +\mathcal I_\varphi(\overline v,\chi_\rho)(t) \\
=&\langle u_0,\varphi\rangle +\mathcal I_\varphi(w,\mathds 1_{(0,\widehat\tau]})(t).
\end{aligned}
\end{equation}
Here the current states and the predictable left limits agree on the respective integration intervals.
The first Poisson integral includes the jump at $\rho$, whereas the restarted integral is supported on $(\rho,\widehat\tau]$.
Thus the jump at $\rho$ is counted exactly once and the possible terminal jump at $\widehat\tau$ is retained.
Equation \eqref{eq:6.3-restarted-pasted-weak} and \eqref{eq:6.3-restart-energy-expectation} prove that $(w,\widehat\tau)$ is a local strong solution in the prescribed class.
Pathwise uniqueness from \cref{eq:2_Main_thm} proves uniqueness of the pasted extension.
\end{proof}

\begin{proof}
[Proof of \cref{prop:6.4-maximal}] By \cref{eq:2_Main_thm}, the family $\mathscr T$ is nonempty and its elements are almost surely positive.
It is closed under pairwise maxima by \cref{lem:6.3-pasting}\,(i).
Stopping an admissible solution at a smaller positive stopping time preserves both its stopped formulation and its energy bound; in particular, $\sigma\wedge k\in\mathscr T$ for $\sigma\in\mathscr T$ and $k\ge1$.

We first construct a countable family realizing the essential supremum.
Set $h(r)=1-e^{-r}$ for $r\in[0,\infty]$, with $h(\infty)=1$, and let $ a:=\sup_{\sigma\in\mathscr T}\E h(\sigma).
$ Choose $(\sigma_n)\subset\mathscr T$ with $\E h(\sigma_n)\to a$, and set $\kappa_n=\bigvee_{j=1}^n\sigma_j$ and $\kappa=\sup_n\kappa_n$.
Closure under maxima and dominated convergence give $ \E h(\kappa)=a, \ \E h(\kappa\vee\sigma) =\lim_{n\to\infty}\E h(\kappa_n\vee\sigma) \le a, \ \sigma\in\mathscr T.
$ Since $h$ is strictly increasing, $\sigma\le\kappa$ almost surely for every $\sigma\in\mathscr T$.
As $\kappa$ is a countable supremum of elements of $\mathscr T$, it represents their essential supremum.
Define $\tau_n:=\kappa_n\wedge n, \ \tau_{\max}:=\sup_{n\ge1}\tau_n.$ Then $\tau_n\in\mathscr T$, $0<\tau_n\le n$, and $\tau_n\uparrow\kappa$ almost surely.
Moreover, $$ \{\tau_{\max}\le t\} =\bigcap_{n\ge1}\{\tau_n\le t\}\in\mathcal F_t, t\ge0.
$$ This proves \eqref{eq:6.4-maximal-time} and \eqref{eq:6.4-positive-maximal-time}.

For each $n$, choose an admissible solution $(U^{(n)},\tau_n)$, represented by its stopped c\`adl\`ag version on $[0,\infty)$.
By \cref{lem:6.3-pasting}\,(i), after discarding one null set,
\begin{equation}
\label{eq:6.4-all-consistency}
U^{(n)}(t)=U^{(m)}(t), \qquad 0\le t\le\tau_m,\quad m\le n,
\end{equation}
simultaneously for all $m,n$.
Define $$u(t) :=\mathds 1_{\{t\le\tau_1\}}U^{(1)}(t) +\sum_{n=2}^{\infty} \mathds 1_{\{\tau_{n-1}<t\le\tau_n\}}U^{(n)}(t), \qquad t\ge0,$$ with value zero when all indicators vanish.
The interval indicators are predictable and the processes $U^{(n)}$ are progressively measurable.
Since at most one summand is nonzero, $u$ is progressively measurable.
Equation \eqref{eq:6.4-all-consistency} gives
\begin{equation}
\label{eq:6.4-u-equals-un}
u^{\tau_n}=U^{(n)} \quad\text{indistinguishably},\qquad n\ge1.
\end{equation}
Thus every stopped process has the required c\`adl\`ag regularity, satisfies the stopped weak formulation, and inherits \eqref{eq:6.4-local-energy} from $U^{(n)}$.
In particular, the compensated Poisson integral remains supported on $(0,\tau_n]$, with its coefficient evaluated at $u(s-)=U^{(n)}(s-)$ on that interval, and retains a possible jump at $\tau_n$.
Every compact interval contained in $[0,\tau_{\max})$ lies inside some $[0,\tau_n]$ almost surely, which proves the asserted local path regularity.

Let $(v,\sigma)$ be an admissible local solution.
Since $\sigma\in\mathscr T$, the definition of $\tau_{\max}$ gives \eqref{eq:6.4-maximality-time}.
Pathwise uniqueness and \eqref{eq:6.4-u-equals-un} imply, on a common full-probability event, $$ v(t)=u(t), \qquad 0\le t\le\sigma\wedge\tau_n, \quad n\ge1.
$$ For every $t<\sigma\le\tau_{\max}$, some $n$ satisfies $t\le\tau_n$; this proves \eqref{eq:6.4-maximality}.
On $\{\sigma<\tau_{\max}\}$, some $n$ also satisfies $\sigma\le\tau_n$, so the same identity proves \eqref{eq:6.4-maximality2}, including a possible jump at $\sigma$.

More generally, let $(v,(\rho_m)_{m\ge1},\zeta)$ be a local solution with $\rho_m\uparrow\zeta$ and admissible closed localizations $(v^{\rho_m},\rho_m)$.
Then $\rho_m\in\mathscr T$ for every $m$, whence $\zeta\le\tau_{\max}$ almost surely.
Applying \eqref{eq:6.4-maximality} to these localizations and taking a countable intersection gives $v=u$ throughout $[0,\zeta)$ almost surely.

Finally, two maximal constructions dominate each other's localizing lifetimes.
Taking their increasing limits yields \eqref{eq:6.4-maximal-uniqueness-time}.
For every $n,m$, pathwise uniqueness gives $u=\widetilde u$ on $[0,\tau_n\wedge\widetilde\tau_m]$ almost surely.
Intersecting these countably many events and using the convergence of both localizing sequences proves \eqref{eq:6.4-maximal-uniqueness-process}.
\end{proof}

Proposition~\ref{prop:6.4-maximal} completes the maximal local well-posedness theory for \eqref{eq:1_Main}.
It identifies the unique maximal process on $[0,\tau_{\max})$ without asserting an endpoint value or a continuation criterion at $\tau_{\max}$.
The intrinsic continuation theory and the global existence argument under additional $L^2/H^1$ hypotheses will be developed in the subsequent paper.
\begin{corollary}
[Localized continuous dependence of maximal solutions] \label{cor:6-maximal-continuous-dependence} Let $p>3$ and assume the hypotheses of \cref{eq:2_Main_thm}.
Let $ \bigl(u,(\tau_j)_{j\ge1},\tau_{\max}^{u}\bigr), \ \bigl(v,(\sigma_j)_{j\ge1},\tau_{\max}^{v}\bigr) $ be the maximal local strong solutions supplied by \cref{prop:6.4-maximal}, with divergence-free initial data $u_0,v_0\in L^p(\Omega,\mathcal F_0;E)$.
The stochastic basis, coefficients, Wiener process, and Poisson random measure are the same for both solutions.
Fix $T>0$ and $M\ge1$.
Let $\rho$ be a stopping time such that
\begin{equation}
\label{eq:6-common-time}
0\le\rho\le T, \qquad \rho<\tau_{\max}^{u}\wedge\tau_{\max}^{v} \quad\P\text{-a.s.},
\end{equation}
and assume
\begin{equation}
\label{eq:6-stability-exit}
\sup_{0\le s<\rho} \bigl(\|u(s)\|_p+\|v(s)\|_p\bigr) \le M \quad\P\text{-a.s.}
\end{equation}

For a difference process $w=u-v$, set $$ \mathcal E_\rho(w) := \sup_{0\le s\le\rho}\|w(s)\|_p^p ++\int_0^\rho \left\|\nabla\bigl(|w(s)|^{p/2}\bigr)\right\|_2^2 \dd s .
$$ Then, for every $A\in\mathcal F_0$,
\begin{equation}
\label{eq:6-maximal-stability}
\E\bigl[\mathds 1_A\mathcal E_\rho(u-v)\bigr] \le C_{M,T}\, \E\bigl[\mathds 1_A\|u_0-v_0\|_p^p\bigr].
\end{equation}
The constant depends only on $M,T,p,\nu,N$ and the structural constants of the coefficient assumptions.
It is independent of $\rho$, the initial data, and the chosen localizing sequences.

Consequently, let $u^{(n)}$ and $u$ be the maximal solutions with initial data $u_0^{(n)}$ and $u_0$, respectively.
Suppose that $\rho_n$ satisfies
\eqref{eq:6-common-time}--
\eqref{eq:6-stability-exit} for the pair $(u^{(n)},u)$, with the same $M,T$.
If $ u_0^{(n)}\longrightarrow u_0 \quad\text{in probability in }E, $ then $ \mathcal E_{\rho_n}(u^{(n)}-u) \longrightarrow0 \quad\text{in probability}.
$ If the initial convergence holds in $L^p(\Omega;E)$, then, $ \E\, \mathcal E_{\rho_n}(u^{(n)}-u) \longrightarrow0.
$
\end{corollary}

\begin{proof}
The estimate of \cref{thm:local-continuous-dependence} also holds with every expectation weighted by $\mathds 1_A$, where $A\in\mathcal F_0$.
Indeed, its proof applies to the linear difference equation multiplied by $\mathds 1_A$.
This multiplication preserves predictability and commutes with spatial differentiation and stochastic integration.

Set $ \rho_j:=\rho\wedge\tau_j\wedge\sigma_j.
$ The closed representatives of $u$ and $v$ up to $\tau_j$ and $\sigma_j$ belong to the local-energy class.
Moreover, \eqref{eq:6-stability-exit} holds with $\rho_j$ in place of $\rho$.
Hence \cref{thm:local-continuous-dependence} gives $$ \E\bigl[ \mathds 1_A\mathcal E_{\rho_j}(u-v) \bigr] \le C_{M,T} \E\bigl[ \mathds 1_A\|u_0-v_0\|_p^p \bigr], $$ with a constant independent of $j$.

Since $\tau_j\uparrow\tau_{\max}^{u}$, $\sigma_j\uparrow\tau_{\max}^{v}$, and $\rho$ is strictly smaller than both maximal lifetimes, $\rho_j=\rho$ for all sufficiently large $j$, almost surely.
Fatou's lemma therefore yields \eqref{eq:6-maximal-stability}.
In particular, the supremum includes the value at $\rho$ and hence retains a possible terminal jump.

The assertion for convergence in $L^p(\Omega;E)$ follows by taking $A=\Omega$.
For convergence in probability, fix $\varepsilon,\delta>0$ and put $ A_{n,\delta} :=\{\|u_0^{(n)}-u_0\|_p\le\delta\} \in\mathcal F_0.
$ By \eqref{eq:6-maximal-stability} and Markov's inequality, $$
\begin{aligned}
\P\bigl( \mathcal E_{\rho_n}(u^{(n)}-u)>\varepsilon \bigr) \le& \P\bigl( \|u_0^{(n)}-u_0\|_p>\delta \bigr) +\frac{C_{M,T}\delta^p}{\varepsilon}.
\end{aligned}
$$ First let $n\to\infty$ and then $\delta\downarrow0$.
This proves the remaining assertion.
\end{proof}
\section*{Acknowledgment}
The authors would like to express sincere gratitude to Dr.
Ratikanta Panda for his overall guidance, invaluable discussions, and continuous support throughout the preparation of this manuscript.
\bibliographystyle{elsarticle-harv} \bibliography{References}
\end{document}